\documentclass[12pt,a4paper]{article}
\usepackage{amsmath}
\usepackage{amssymb}
\usepackage{amsthm}
\usepackage{amsfonts}
\usepackage{latexsym}

\theoremstyle{plain}
\newtheorem{thm}{Theorem}
\newtheorem{cor}[thm]{Corollary}
\newtheorem{lem}[thm]{Lemma}
\newtheorem{prop}[thm]{Proposition}
\theoremstyle{remark}

\theoremstyle{definition}
\newtheorem{defn}[thm]{Definition}
\newtheorem{notn}[thm]{Notation}
\newtheorem{rem}[thm]{Remark}
\newtheorem{ass}[thm]{Assumption}

\title{Equivariant scaling asymptotics for
Poisson and Szeg\H{o} kernels on Grauert tube boundaries}
\author{Simone Gallivanone and Roberto Paoletti\footnote{\noindent{\bf Address:}
Dipartimento di Matematica e Applicazioni, Universit\`a degli Studi
di Milano Bicocca, Via R. Cozzi 55, 20125 Milano,
Italy; {\bf e-mail}: s.gallivanone@campus.unimib.it,
roberto.paoletti@unimib.it }}
\date{}

\begin{document}
\maketitle

\begin{abstract}
Let $(M,\kappa)$ be a closed and connected real-analytic Riemannian manifold, acted upon by a compact Lie group of isometries $G$.
We consider the following two kinds of equivariant asymptotics along a fixed Grauer tube boundary $X^\tau$ of $(M,\kappa)$.

\noindent
1): Given the induced unitary representation of $G$ on the eigenspaces of the Laplacian of $(M,\kappa)$, these split 
over the irreducible representations of $G$.
On the other hand, the eigenfunctions of the Laplacian of $(M,\kappa)$ admit a simultaneous complexification to some Grauert tube.
We study the asymptotic concentration along $X^\tau$ of the complexified 
eigenfunctions pertaining to a fixed isotypical component.

\noindent
2):
There are furthermore an induced action of $G$ as a group of CR and contact automorphisms
on $X^\tau$, and a corresponding
unitary representation on the Hardy space $H(X^\tau)$. 
The action of $G$ on $X^\tau$ commutes with the homogeneous \lq geodesic flow\rq\,
and the representation on the Hardy space commutes with the elliptic self-adjoint Toeplitz operator induced by the
generator of the goedesic flow. Hence each eigenspace of the latter also splits over the irreducible representations of
$G$. We study the asymptotic concentration of the eigenfunctions in a given isotypical component.

We also give some applications of these asymptotics.

\end{abstract}

\section{Introduction}
\label{sctn:intro}

Let $M$ be a compact and connected $d$-dimensional real-analytic (in the following, $\mathcal{C}^\varpi$) manifold. There exists an essentially
unique complex manifold $(\tilde{M}, J)$, the so-called \textit{Bruhat-Whitney complexification} of $M$, in which $M$ embeds as totally real submanifold (i.e. $T_xM\cap J_x\left(T_xM\right) = \{0\}$ for any $x\in M$) 
\cite{bw}. More precisely, $(\tilde{M},J)$ 
is uniquely determined as a germ of complex manifold along $M$;
in particular, since both $(\tilde{M},J)$ and $(\tilde{M},-J)$ are 
complexifications of $M$, there is an anti-holomorphic involution 
$\sigma:\tilde{M}\rightarrow \tilde{M}$ having $M$ as its fixed locus. 

As proved and discussed in \cite{gs1}, \cite{gs2}, \cite{l}, \cite{ls},
\cite{sz} (see also \cite{pw2} and \cite{pw1}),
the choice of a $\mathcal{C}^\varpi$
Riemannian metric $\kappa$ on $M$ singles out a unique 
$\mathcal{C}^\varpi$ exhaustion function on $\tilde{M}$ (perhaps after passing to a smaller
tubular neighbourhood of $M$ in $\tilde{M}$), with the following
properties:
\begin{enumerate}
    \item $\rho:\tilde{M}\rightarrow [0,+\infty)$, and
    $M=\rho^{-1}(0)$;
    \item $\rho = \rho\circ \sigma$;
    \item $\rho$ is strictly plurisubharmonic, i.e. it induces  
    a K\"ahler form $\Omega$ and  a Riemannian metric
    $\hat{\kappa}$ on $\tilde{M}$ given by
    \begin{equation}
    \label{eqn:Omega and kappa}
    \Omega:=\imath\partial\overline{\partial}\rho,\qquad
    \hat{\kappa} := \Omega(\cdot, J(\cdot)).
    \end{equation} 
    \item $(M,\kappa)$ is a Riemannian submanifold of 
    $(\tilde{M},\hat{\kappa})$;
    \item on $\tilde{M}\setminus M$ the function $\sqrt{\rho}$ 
    satisfies the complex homogeneous Monge-Amp\`{e}re equation:
    \begin{equation*}
        \mathrm{det}\left(\frac{\partial^2\sqrt{\rho}}{\partial z_i\partial\overline{z}_j}\right)=0.
    \end{equation*}
\end{enumerate}
Without pretense of completeness, we refer the reader to \cite{b},
\cite{bh}, \cite{g}, \cite{z07}, \cite{z12}, \cite{z20}, \cite{cr1} and \cite{cr2}
for a wider foundational discussion.


Let $\tau_0^2:=\sup (\rho)$; then 
$\tau_0\in (0,+\infty]$, and
for any $\tau\in (0,\tau_0)$
the (open) \textit{Grauert tube of radius $\tau$} 
and its boundary are given by
\begin{equation}
\label{eqn:MXtau}
\tilde{M}^\tau:=\rho^{-1}\big([0,\tau^2)\big),
\qquad X^\tau:=\partial \tilde{M}^\tau
=\left\{\sqrt{\rho}=\tau   \right\}.
\end{equation}

Grauert tubes have been studied extensively in recent years, 
both for their intrinsic geometric interest and in relation
to the a fundamental holomorphic extension property of the eigenfunctions
of the non-negative Laplacian $\Delta$ of $(M,\kappa)$,
which was first discovered by Bouted de Monvel
\cite{bdm1}; this foundational result was also proved
and analyzed in \cite{leb}, \cite{s1}, \cite{s2},\cite{z20}
(see furthermore the discussion in \cite{gs2}).
Grauert tube techniques have also been applied to the study 
of nodal sets (see \cite{ct16}, \cite{ct18}, \cite{tz09},
\cite{tz21}).

Let us recall 
Boutet de Monvel's foundational result,
as formulated in \cite{leb}.
Denote by
\begin{equation}
\label{eqn:eigenvalue laplacian}
0=\mu_1^2<\mu_2^2<\cdots\uparrow +\infty\qquad 
\text{where}\qquad 
\mu_j\ge 0
\end{equation}
the distinct eigenvalues of $\Delta$; for every $j=1,2,\cdots$, 
let 
$W_j\subset \mathcal{C}^\infty(M)$ be the (finite-dimensional)
eigenspace for the eigenvalue
$\mu_j^2$, and let $(\varphi_{j,k})_{k=1}^{n_j}$ a real orthonormal basis of
$W_j$. 
For $\mu\in \mathbb{R}$, let us set
$\langle\mu\rangle:=\left(1+\mu^2\right)^{1/2}$.

\begin{thm}
\label{thm:holo extension} (Boutet de Monvel)
There exists $\tau_1\in (0,\tau_0]$ such that the following holds:
\begin{enumerate}
\item every eigenfunction 
$\varphi$ of $\Delta$    
admits a holomorphic extension $\tilde{\varphi}$ to 
 $\tilde{M}^{\tau_1}$, hence 
for any $\tau\in (0,\tau_1)$ 
 the restriction 
 $\tilde{\varphi}^\tau:=\left.\tilde{\varphi}\right|_{X^\tau}$ 
 to $X^\tau$ is a CR function;
 \item the sequence of functions
 $$
 \left(e^{-\tau\,\mu_j}\,\langle \mu_j\rangle^{(d-1)/4}\,
 \tilde{\varphi}_{j,k}^\tau\right)_{j\ge 1, \,k=1,\ldots,n_j}
 $$
 is a Riesz basis of $H(X^\tau)$.
\end{enumerate}

\end{thm}

If $\tau\in (0,\tau_0)$, 
on $X^\tau$ we have the following geometric structures
(see \S 3.2 of \cite{p24}, to which we conform for conventions and notation):
\begin{enumerate}
\item a contact form $\alpha^\tau$, given a follows: if 
$\jmath^\tau:X^\tau\hookrightarrow \tilde{M}$ is the inclusion,
\begin{equation}
\label{eqn:contact str}
\alpha^\tau:={\jmath^\tau}^*(\alpha),\qquad
\text{where}\qquad
\alpha:=\Im(\partial\rho);
\end{equation}

\item a CR structure 
\begin{equation}
\label{eqn:defn Htau}
\mathcal{H}^\tau:=TX^\tau\cap J(TX^\tau)=\ker(\alpha^\tau)
\end{equation}
(the maximal complex sub-vector bundle of the tangent bundle $TX^\tau$);

\item the closed symplectic cone sprayed by $\alpha^\tau$ 
in the complement of the
zero section $X^\tau_0$ of the cotangent bundle $T^\vee X^\tau$
of $X^\tau$:
\begin{equation}
\label{eqn:Sigmataucone}
\Sigma^\tau:=\left\{\left(x,r\,\alpha^\tau_x\right)\,:\,
x\in X^\tau,\,r>0\right\}\subseteq T^\vee X^\tau\setminus X^\tau_0;
\end{equation}
\item the smooth vector field $\upsilon_{\sqrt{\rho}}^\tau\in 
\mathfrak{X}(X^\tau)$ given by the restriction to $X^\tau$ of the Hamiltonian vector field $\upsilon_{\sqrt{\rho}}$ of $\sqrt{\rho}$
with respect to $\Omega$;

\item the Reeb vector field $\mathcal{R}^\tau$ of $(X^\tau,\alpha^\tau)$,
given by 
\begin{equation}
\label{eqn:reeb vct field}
\mathcal{R}^\tau:=-\frac{1}{\tau}\,\upsilon_{\sqrt{\rho}}^\tau;
\end{equation}
\item the real vector sub-bundle $\mathcal{T}^\tau$ 
of the tangent bundle $TX^\tau$ generated $\mathcal{R}^\tau$, so that
\begin{equation}
\label{eqn:dicrect sum vert hor}
TX^\tau=\mathcal{T}^\tau\oplus \mathcal{H}^\tau
\end{equation}
(we shall occasionally refer to $\mathcal{T}^\tau$ and
$\mathcal{H}^\tau$ as, respectively, the \textit{$\alpha^\tau$-vertical and
the 
$\alpha^\tau$-horizontal tangent bundles of $X^\tau$});
\item a \lq Riemannian\rq\, volume form $\mathrm{vol}^R_{X^\tau}$,
given by the contraction of the gradient of $\sqrt{\rho}$ with the 
symplectic volume form of $(\tilde{M}\setminus M,\Omega)$, restricted to
$X^\tau$.
\end{enumerate}

In addition, 
on $X^\tau$ we have the following analytic structures:

\begin{enumerate}
\item the Hardy space $H(X^\tau):=\mathrm{ker}(\overline{\partial}_b)\subseteq L^{2}(X^{\tau})$;
\item
the corresponding Szeg\H{o} projector and its distributional kernel,
the so-called Szeg\H{o} kernel:
    \begin{equation}
    \label{eqn:szego proj and ker}
   \Pi^\tau: L^{2}(X^{\tau}) \longrightarrow H(X^\tau),
   \qquad
   \Pi^\tau(\cdot,\cdot)\in \mathcal{D}'(X^\tau\times X^\tau)
   \end{equation}
(what follows is pivoted on the microlocal description of 
$\Pi^\tau$ as a Fourier integral operator in \cite{bs});
    
    \item a privileged collection of CR functions (elements of 
    $H(X^\tau)$),
    given by the restrictions $\tilde{\varphi}^\tau$
    to $X^\tau$ of the holomorphic extensions 
    $\tilde{\varphi}$
    of the eigenfunctions $\varphi$ of $\Delta$ (if $\tau\in (0,\tau_1)$);
\item the self-adjoint first-order elliptic 
    Toeplitz operator
    \begin{equation}
    \label{eqn:DsqrtrhoToeplitz}
    \mathfrak{D}^\tau_{\sqrt{\rho}}:= \Pi^\tau\circ D^\tau_{\sqrt{\rho}}\circ \Pi^\tau, \qquad \text{where}\qquad 
    D^\tau_{\sqrt{\rho}}:=\imath\,\upsilon^\tau_{\sqrt{\rho}}.
    \end{equation}
     
\end{enumerate}

In fact, the principal symbol of $\mathfrak{D}^\tau_{\sqrt{\rho}}$
(i.e., the restriction to $\Sigma^\tau$
in (\ref{eqn:Sigmataucone}) of the principal symbol of 
$D^\tau_{\sqrt{\rho}}$) is strictly positive; 
hence the spectrum of $\mathfrak{D}^\tau_{\sqrt{\rho}}$
is discrete
and accumulates only at $+\infty$ \cite{bg}.


In a series of papers,
Zelditch has laid the foundations and paved the way for the asymptotic study 
of both the complexified eigenfunctions of the Laplacian and of the
eigenfunctions of $\mathfrak{D}^\tau_{\sqrt{\rho}}$
(see e.g. \cite{z07}, \cite{z12}, \cite{z14}, \cite{z17}, \cite{z20}).
These asymptotics are encrypted in certain \lq tempered\rq\, 
complexified
spectral projectors (in the case of the Laplacian) and in a smoothed
spectral projector (in the case of $\mathfrak{D}^\tau_{\sqrt{\rho}}$),
whose definitions we briefly recall.

In the case of the Laplacian, the asymptotic distribution of the
complexified eigenfunctions of $\Delta$, restricted to $X^\tau$,
is encapsulated in the asymptotics for 
$\lambda\rightarrow+\infty$ of
the tempered complexified projector kernels
$P^\tau_{\chi,\,\lambda}\in \mathcal{C}^\infty(X^\tau\times X^\tau)$ 
given by
\begin{equation}
\label{eqn:temepered_proj_kernel_poisson}
        P^\tau_{\chi,\,\lambda}(x,y) := \sum_{j\geq 1} 
        \hat{\chi}(\lambda-\mu_j) e^{-2\tau\mu_j} \sum_{k=1}^{n_j} \tilde{\varphi}^\tau_{j,k}(x)\,\overline{\tilde{\varphi}^\tau_{j,k}(y)},
\end{equation}
where $\chi\in \mathcal{C}^\infty_c(\mathbb{R})$
(the space of compactly supported smooth functions on $\mathbb{R}$) 
and
$\hat{\chi}$ is its Fourier transform. 
Since the restrictions $(\tilde{\varphi}^\tau_{j,k})_k$ are not
orthonormal, the inner sum may not be interpreted as a genuine
projector, even in a rescaled sense.

Similarly, let
$\lambda_1^\tau<\lambda_2^\tau<\cdots\uparrow +\infty$ denote the distinct eigenvalues
of $\mathfrak{D}^\tau_{\sqrt{\rho}}$, and for each 
$j$ let $H(X^\tau)_j\subset H(X^\tau)$ be the eigenspace of
$\lambda_j^\tau$. For $j=1,2,\ldots$, let
$(\rho_{j,k}^\tau)_{k=1}^{\ell_j^\tau}$ be an orthonormal basis of
$H(X^\tau)_j$. The asymptotic distribution of the 
$\rho_{j,k}^\tau$'s and of the $\lambda_j^\tau$'s is probed 
by smoothed projection kernels of the form
$\Pi^\tau_{\chi,\,\lambda}\in \mathcal{C}^\infty(X^\tau\times X^\tau)$
given by
\begin{equation}
\label{eqn:smoothed proj kern}
        \Pi^\tau_{\chi,\,\lambda}(x,y) := 
        \sum_{j\geq 1} \hat{\chi}\left(\lambda-\lambda_j^\tau\right) 
        \,\sum_{k=1}^{ \ell_j^\tau} \rho_{j,k}(x)\,\overline{\rho_{j,k}(y)}.
    \end{equation}
This is the Schwartz kernel of the smoothed projector 
\begin{equation}
\label{eqn:Pi chi lambda}
\Pi^\tau_{\chi,\,\lambda}:= 
\sum_{j\geq 0} \hat{\chi}\left(\lambda-\lambda_j^\tau\right)\Pi^\tau_j,
\end{equation}
where $\Pi^\tau_j:L^2(X^\tau)\rightarrow H(X^\tau)_j$ is the orthogonal
projector.

\begin{rem}
In the following, for notational simplicity we shall omit the suffix $\tau$ from the
spectral data $\lambda_j^\tau$ and $\ell_j^\tau$.
\end{rem}

While (\ref{eqn:temepered_proj_kernel_poisson}) is peculiar to the
Grauert tube setting, (\ref{eqn:smoothed proj kern}) is the
counterpart of classically studied kernels in the spectral theory
of pseudodifferential operators (see e.g. \cite{dg} and \cite{gs});
in the Toeplitz setting, in the geometric context of positive line bundles and with an emphasis
on local scaling asymptotics, they have been used in \cite{p09}, \cite{p10}, \cite{p11},
\cite{p12}, \cite{p17}, \cite{p18}, \cite{zz18}, \cite{zz119},
\cite{zz219}.

In two recent papers (\cite{cr1} and \cite{cr2}), 
Chang and Rabinowitz have 
made groundbreaking progress in pushing forward the analogy
between the line bundle and the Grauert tube settings, providing
local scaling asymptotics for 
(\ref{eqn:temepered_proj_kernel_poisson}) and 
(\ref{eqn:smoothed proj kern}) bearing a striking similarity with
the local scaling asymptotics for the 
Fourier components of Szeg\H{o} kernels 
in the line bundle setting (\cite{z98}, \cite{bsz} and \cite{sz}).
Besides the description of $\Pi^\tau$ as a Fourier integral
operator in \cite{bs}, which underpins the whole theory, 
the approach of Chang and Rabinowitz in based on the following
two pillars. The first is the description, due to Zelditch,
of certain 1-parameter groups of Toeplitz operators
as \lq dynamical Toeplitz operators\rq, which goes back to
\cite{z97} and is pervasive in his treatment of the subject.
The second is a clever use of the \lq Heisenberg local coordinates\rq\,
adapted to a hypersurface in a complex manifold introduced by
Folland and Stein in \cite{fs1} and \cite{fs2}. The approach of
Chang and Rabinowitz has been reviewed and elaborated in
\cite{p24}, where a refinement of their asymptotic
expansions is given in the near-diagonal case, where 
$\chi\in \mathcal{C}^\infty_0\big( (-\epsilon,\epsilon) \big)$.

The aim of this paper is to extend the results of \cite{p24}
to the (off-diagonal and) equivariant setting. 
Suppose given a smooth action
$\mu:G\times M\rightarrow M$ of a 
$d_G$-dimensional
compact Lie group $G$ as a group of
symmetries on $(M,\kappa)$. 
Then there are naturally induced unitary representations of
$G$ on $L^2(M)$ and $H(X^\tau)$, commuting with $\Delta$ and 
$\mathfrak{D}^\tau_{\sqrt{\rho}}$, respectively.
Therefore, for every $j\ge 1$ 
both the eigenspaces $W_j\subset L^2(M)$ of $\Delta$ and 
$H(X^\tau)_j\subset H(X^\tau)$ are invariant. 

\begin{notn}
\label{notn:irred repr char}
Let $\hat{G}$ denote the unitary dual of $G$; for every
$\nu\in \hat{G}$,
we shall adopt the following notation.

\begin{enumerate}
\item $V_\nu$ is the representation space
of $\nu$, and $\dim(\nu):=\dim(V_\nu)$;
\item $\Xi_\nu:G\rightarrow \mathbb{C}$ is the character of $\nu$.
\item If $G$ is unitarily represented on a Hilbert space
$H$, $H_\nu\subseteq H$ will
denote  the $\nu$-th isotypical component of $H$.
\end{enumerate}

\end{notn}

By the Theorem 
of Peter and Weyl (see e.g. \cite{s95})
there are unitary equivariant Hilbert direct sum decompositions
\begin{equation}
\label{eqn:L2Mdec}
L^2(M)=\bigoplus_{\nu\in \hat{G}}L^2(M)_\nu,
\quad W_j=\bigoplus_{\nu\in \hat{G}}W_{j,\nu},
\end{equation}
where $W_{j,\nu}:=W_j\cap L^2(M)_\nu$,
and similarly
\begin{equation}
\label{eqn:L2Mdecompequiv}
H(X^\tau)=\bigoplus_{\nu\in \hat{G}}H(X^\tau)_\nu,
\quad H(X^\tau)_j=\bigoplus_{\nu\in \hat{G}}H(X^\tau)_{j,\nu},
%
\end{equation}
where
$H(X^\tau)_{j,\nu}:=H(X^\tau)_j\cap H(X^\tau)_\nu$.

For every $(j,\nu)$ such that $W_{j,\nu}\neq (0)$, let 
$(\varphi_{j,\nu,k})_{k=1}^{n_{j,\nu}}$ be a real orthonormal
basis of $W_{j,\nu}$. The equivariant analogue of 
(\ref{eqn:temepered_proj_kernel_poisson}) is the smooth kernel
$P^\tau_{\chi,\nu,\lambda}(\cdot,\cdot)\in \mathcal{C}^\infty(X^\tau\times X^\tau)$
given by
\begin{equation}
\label{eqn:temepered_proj_kernel_poisson_equiv}
        P^\tau_{\chi,\nu,\lambda}(x,y) := \sum_{j\geq 1} 
        \hat{\chi}(\lambda-\mu_j) e^{-2\tau\mu_j} \sum_{k=1}^{n_{j,\nu}} \tilde{\varphi}^\tau_{j,\nu,k}(x)\,
        \overline{\tilde{\varphi}^\tau_{j,\nu,k}(y)}.
\end{equation}
Similarly, for every $(j,\nu)$ for which 
$H(X^\tau)_{j,\nu}\neq (0)$, 
let $(\rho_{j,\nu,k})_{k=1}^{\ell_{j,\nu}}$
be an orthonormal basis of $H(X^\tau)_{j,\nu}$. The 
equivariant analogue of (\ref{eqn:smoothed proj kern}) is the
smooth kernel $\Pi^\tau_{\chi,\nu,\lambda}(\cdot,\cdot)\in 
\mathcal{C}^\infty(X^\tau\times X^\tau)$ given by
\begin{equation}
\label{eqn:smoothed proj kern equiv}
        \Pi^\tau_{\chi,\nu,\lambda}(x,y) := 
        \sum_{j\geq 1} \hat{\chi}\left(\lambda-\lambda_j\right) 
        \,\sum_{k=1}^{ \ell_{j,\nu}} \rho_{j,\nu,k}(x)\,
        \overline{\rho_{j,\nu,k}(y)}.
    \end{equation}
Thus $\Pi^\tau_{\chi,\nu,\lambda}(\cdot,\cdot)$ is the Schwartz kernel
of the smoothed projector
\begin{equation}
\label{eqn:Pi chi nu lambda projector}
\Pi^\tau_{\chi,\nu,\lambda} := 
        \sum_{j\geq 1} \hat{\chi}\left(\lambda-\lambda_j\right) 
        \,\sum_{k=1}^{ \ell_{j,\nu}} \,
        \Pi^\tau_{\nu,j}
=P_\nu\circ \Pi_{\chi,\lambda}, 
\end{equation}
where $\Pi^\tau_{\nu,j}:L^2(X^\tau)\rightarrow H(X^\tau)_{j,\nu}$
and $P_\nu:L^2(X^\tau)\rightarrow L^2(X^\tau)_\nu$
are the orthogonal projectors.

We shall provide local scaling asymptotics for (\ref{eqn:temepered_proj_kernel_poisson_equiv}) and 
(\ref{eqn:smoothed proj kern equiv}).
Unlike \cite{p24}, we are not assuming here that 
$\chi$ is supported near the origin; in the special 
action-free
case, we shall thus recover a refinement of the near-graph 
scaling asymptotics of Chang and Rabinowitz in \cite{cr2}.

The smooth action $\mu:G\times M\rightarrow M$
extends (perhaps after decreasing $\tau_0$) to an 
action $\tilde{\mu}:G\times \tilde{M}^{\tau_0}\rightarrow
\tilde{M}^{\tau_0}$ of $G$ as a group of biholomorphisms.
Since $\mu$ is isometric for $(M,\kappa)$, $\tilde{\mu}$
preserves $\rho$.
Thus $G$ acts as a group of Hamiltonian automorphisms
of the K\"{a}hler manifold $(\tilde{M}^{\tau_0},J,\Omega)$,
with a moment map $\Phi:\tilde{M}^{\tau_0}\rightarrow \mathfrak{g}^\vee$
vanishing identically on $M$.

Since $\tilde{\mu}$ preserves $\rho$,
it also preserves every $X^\tau$, with its 
CR and contact structure, for $\tau\in (0,\tau_0)$;
let $\tilde{\mu}^\tau:G\times X^\tau\rightarrow X^\tau$
be the resulting contact and CR action.
Let us set
\begin{equation}
\label{eqn:defn di Ztau}
Z:=\Phi^{-1}(0)\subseteq \tilde{M}^{\tau_0},
\qquad Z^\tau:=Z\cap X^\tau.
\end{equation}
What follows depends on the following hypothesis.

\begin{ass}
\label{ass:ba}
We shall assume that:
\begin{enumerate}
\item $Z\setminus M\neq \emptyset$ (hence $Z^\tau\neq \emptyset$ for every
$\tau\in (0,\tau_0)$);
\item $\tilde{\mu}^\tau$ is locally free on $Z^\tau$.
\end{enumerate}

\end{ass} 

\begin{rem}
\label{rem:rem on assmpt}
Let us notice the following regarding the two conditions
in Assumption \ref{ass:ba}.
\begin{enumerate}
\item The second requirement is equivalent to $0\in \mathfrak{g}^\vee$ being 
a regular value of $\left.\Phi\right|_{\tilde{M}^{\tau_0}\setminus M}$ (see e.g. \cite{mm}, \cite{gstb}); therefore, it implies that $Z\setminus M$ and $Z^\tau$
are submanifolds of $\tilde{M}$ and $X^\tau$,
respectively.
Our scaling asymptotics rely on a direct sum decomposition
of $TX^\tau$ along $Z^\tau$ that depends on this
smoothness assumption (see \S \ref{sctn:decomp}). 
\item $\tilde{\mu}^\tau$ is \textit{a fortiori} locally free on $Z^\tau$ if $\mu$ itself is locally free, e.g., if $M$ is a principal $G$-bundle;
for more examples, see \S \ref{sctn:Z tau}.
\item Under Assumption \ref{ass:ba}, $d>d_G$ 
(see Corollary \ref{cor:d>d_G} below); for instance, the
case of a compact Lie group acting on itself by, say,
left translations is not covered by the present analysis.
This is because $Z^\tau=\emptyset$ in this case (we shall
consider this specific situation in a separate paper).
\end{enumerate}

\end{rem}
 
Let $\Gamma:\mathbb{R}\times \tilde{M}\rightarrow \tilde{M}$ 
denote the flow
of $\upsilon_{\sqrt{\rho}}$ ($\Gamma$ will be referred to - 
with some abuse of language - as the \lq homogeneous geodesic flow\rq,
\, since it is intertwined with the 
latter by the imaginary time exponential map - see \cite{gs2}, \cite{ls}, \cite{gls}). Its restriction to $X^\tau$
is the flow of $\upsilon_{\sqrt{\rho}}^\tau$, and will be denoted
$\Gamma^\tau:\mathbb{R}\times X^\tau\rightarrow X^\tau$.
Since $\tilde{\mu}^\tau$ and $\Gamma^\tau$ commute,
there is a product action of $G\times \mathbb{R}$ on $X^\tau$;
our first result is that (\ref{eqn:temepered_proj_kernel_poisson_equiv}) and 
(\ref{eqn:smoothed proj kern equiv}) asymptotically concentrate
near certain compact loci $\mathfrak{X}^\tau_\chi
\subset X^\tau\times X^\tau$ determined
by $\chi$, the moment map, and the orbits of the latter action.

\begin{defn}
\label{defn:concentration locus}
If $x\in X^\tau$ and $\chi\in \mathcal{C}^\infty_c(\mathbb{R})$, we set
$$
x^{G\times \chi}:=\left\{\tilde{\mu}^\tau_g\circ \Gamma^\tau_t(x)\,:
\,g\in G,\,t\in \mathrm{supp}(\chi)\right\}.
$$
We then pose
$$
\mathfrak{X}^\tau_\chi:=\left\{(x_1,x_2)\in Z^\tau\times Z^\tau\,:\,
x_1\in x_2^{G\times \chi}\right\}.
$$
We also set
$$
x^{ \chi}:=\left\{ \Gamma^\tau_t(x)\,:
\,t\in \mathrm{supp}(\chi)\right\}.
$$
\end{defn}

\begin{thm}
\label{thm:main 1}
For any $C,\,\epsilon'>0$, we have
$$
P^\tau_{\chi,\nu,\lambda}(x,y) =O\left(\lambda^{-\infty}\right)
\quad\text{and}\quad
\Pi^\tau_{\chi,\nu,\lambda}(x,y) =O\left(\lambda^{-\infty}\right),
$$
uniformly for
$$
\max\left\{\mathrm{dist}_{X^\tau}\left(x,y^{G\times\chi}\right),
 \mathrm{dist}_{X^\tau}\left(x,Z^\tau\right)  \right\}
 \ge C\,\lambda^{\epsilon'-1/2}.
 $$
\end{thm}

We are thus led to consider the asymptotics of
$P^\tau_{\chi,\nu,\lambda}(x,y)$ and 
$\Pi^\tau_{\chi,\nu,\lambda}(x,y)$ when $(x,y)$ ranges in a
shrinking neighourhood of a fixed pair 
$(x_1,x_2)\in \mathfrak{X}^\tau_\chi$.

\subsection{Scaling asymptotics}

\label{sctn:scaling}

The previous estimates motivate the expectation that 
near $\mathfrak{X}^\tau_\chi$ the two kernels
satisfy scaling asymptotics on a scale of $O(\lambda^{-1/2})$ exhibing an exponential decay along directions transverse to 
$\mathfrak{X}^\tau_\chi$. The aim of the following statements is to substantiate this expectation. 

Before giving precise statements, 
it is in order to premise a few general remarks.
While the general approach is heuristically inspired by the line bundle setting, the arguments are significantly more involved in the present context,   
the primary reason being that the geodesic flow is generally not holomorphic. As the reader will appreciate, 
Zelditch' method of  dynamical Toeplitz operators plays an essential conceptual and technical role in dealing with this difficulty. Another key ingredient is the use of
suitable sets of local coordinates (whose construction goes
back to Folland and Stein in \cite{fs1} and \cite{fs2}, and which were first used in this context by Chang and Rabinowitz), specifically adapted to the local
CR geometry of $X^\tau$.

More precisely, as in \cite{p24}
the asymptotics in point will be formulated in suitable systems
of local coordinates at points $x\in X^\tau$, called \textit{normal
Heisenberg local coordinates} (in the following, NHLC's), 
which are a slight specialization
of the Heisenberg local coordinates in \cite{cr1} and \cite{cr2}
(see \S 3.3 of \cite{p24}).

NHLC's centered at $x\in X^\tau$ will be written, in additive notation, 
as 
$x+(\theta,\mathbf{v})$, where $(\theta,\mathbf{v})\in \mathbb{R}\times 
\mathbb{R}^{2d-2}$ belongs to a neighbourhood of the origin. 
In terms of (\ref{eqn:dicrect sum vert hor}), we have
\begin{equation}
\label{eqn:Reeb theta v H}
\left.\frac{\partial}{\partial \theta}\right|_x=\mathcal{R}^\tau(x)
\in \mathcal{T}_x^\tau ,\qquad
\left.\frac{\partial}{\partial \mathbf{v}}\right|_x\in 
\mathcal{H}_x^\tau
\quad\forall\,\mathbf{v}\in \mathbb{R}^{2d-2}.
\end{equation}

Let us fix $(x_1,x_2)\in \mathfrak{X}^\tau_\chi$ and NHLC's on $X^\tau$
at $x_1$ and $x_2$. We shall work in rescaled coordinates and set
\begin{equation}
\label{eqn:rescaled coord 12}
x_{j,\lambda}:=x_j+\left(\frac{\theta_j}{\sqrt{\lambda}},
\frac{\mathbf{v}_j}{\sqrt{\lambda}}   \right)
\qquad (j=1,2).
\end{equation}

The non-holomorphicity, or equivalently the non-unitarity, of the geodesic flow is  encapsulated in the appearance of a (generally) non-unitary matrix $B$ in the local description of
$\Gamma^\tau$ (see (\ref{eqn:defn of M12}) and \S \ref{sctn:psitaux1x2}), and in the 
ensuing oscillatory and Gaussian 
integrals computing the asymptotics. 

In order to obtain tractable computations and extract more intelligible geometric information, we have in some cases restricted the type of rescaled directions.
Specifically, in the near-graph action-free setting of
\cite{cr2} (Theorem \ref{thm:main 3} below) we consider arbitrary rescaled displacements; 
similarly, in the near-diagonal equivariant
case (Theorem \ref{thm:diagonal case})
we allow arbitrary displacements away from $G$-orbits. 
However, in the general equivariant case we consider more specific choices of directions (see below).

\subsubsection{Scaling asymptotics in the action-free case}
To fix ideas and ease of exposition, 
let us first consider separately 
the action-free case (i.e. where $G$ is trivial);
we shall write $x^\chi$ for $x^{G\times \chi}$.
In this case, $Z^\tau=X^\tau$ and by 
Theorem \ref{thm:main 1} we have
$$
P^\tau_{\chi,\lambda}(x,y) =O\left(\lambda^{-\infty}\right)
\quad\text{and}\quad
\Pi^\tau_{\chi,\lambda}(x,y) =O\left(\lambda^{-\infty}\right),
$$
uniformly for
$$
\mathrm{dist}_{X^\tau}\left(x,y^{\chi}\right)
 \ge C\,\lambda^{\epsilon'-1/2}.
 $$

We shall present near-graph 
scaling asymptotics 
for $P^\tau_{\chi,\lambda}(x_{1,\lambda},x_{2,\lambda})$ and 
$\Pi^\tau_{\chi,\lambda}(x_{1,\lambda},x_{2,\lambda})$
refining those in \cite{cr2}, where rescaling is according to
Heisenberg type; furthermore, we shall provide an explicit description
of the leading order term and an estimate on the degree
of the polynomials in the rescaled variables intervening in the
lower order terms of the asymptotic expansions.

The leading order term describing the exponential decay in the 
scaling asymptotics of $P^\tau_{\chi,\nu,\lambda}(x_{1,\lambda},x_{2,\lambda})$ and 
$\Pi^\tau_{\chi,\nu,\lambda}(x_{1,\lambda},x_{2,\lambda})$
for $\lambda\rightarrow +\infty$
depends on a real-quadratic complex-valued form on the vector
subspace
$\mathcal{H}_{x_1}^\tau\times \mathcal{H}_{x_2}^\tau
\subseteq T_{x_1}X^\tau\times T_{x_2}X^\tau$. To describe the latter
form, we need a brief digression.

\begin{defn}
\label{defn:sympl compl mat}
Given a symplectic matrix $A\in \mathrm{Sp}(2d-2)$, we shall set
$$
A_c:=\mathcal{W}\,A\,\mathcal{W}^{-1}\qquad
\text{where}\qquad
\mathcal{W}:=\frac{1}{\sqrt{2}}\,
\begin{pmatrix}
I_{d-1}&\imath\,I_{d-1}\\
I_{d-1}&-\imath\,I_{d-1}
\end{pmatrix}.
$$
Then 
\begin{equation}
\label{eqn:matrice APQ}
A_c=\begin{pmatrix}
P&Q\\
\overline{Q}&\overline{P}
\end{pmatrix},
\end{equation}
where $P$ is invertible and $\|P\,\mathbf{z}\|\ge \|\mathbf{z}\|$, $\forall\,\mathbf{z}\in 
\mathbb{C}^{d-1}$ (\S 4.1 of \cite{fo}).
\end{defn}

\begin{defn}
\label{defn:quadratice form}
Let us identify $\mathbb{R}^{2d-2}\cong \mathbb{C}^{d-1}$ in the usual
norm-preserving manner, so that if $Z_j\in \mathbb{C}^{d-1}$ corresponds
to $\mathbf{v}_j\in \mathbb{R}^{2d-2}$ then
$$
h_0(Z_1,Z_2)=g_0(\mathbf{v}_1,\mathbf{v}_2)
-\imath\,\omega_0(\mathbf{v}_1,\mathbf{v}_2),\qquad
(\mathbf{v}_1,\,\mathbf{v}_2\in \mathbb{C}^{d-1}),
$$
where $h_0$, $g_0$, and $\omega_0$ denote the standard 
Hermitian, Euclidean, and symplectic structures, respectively.
Given $A\in \mathrm{Sp}(2d-2)$, let us define 
$\Psi_A:\mathbb{R}^{2d-2}\times \mathbb{R}^{2d-2}\rightarrow \mathbb{C}$
(or, equivalently, $\Psi_{A_c}:\mathbb{C}^{d-1}\times 
\overline{\mathbb{C}}^{d-1}
\rightarrow \mathbb{C}$)
as follows.
With $P$ and $Q$ as in Definition \ref{defn:sympl compl mat},
\begin{eqnarray*}
\lefteqn{\Psi_{A}(\mathbf{v}_1,\mathbf{v}_2)
=
\Psi_{A_c}(Z_1,Z_2)
}\\
&:=&\frac{1}{2}\,\left(Z_1^\dagger\,\overline{Q}\,P^{-1}\,Z_1
+2\,\overline{Z}_2^\dagger\, P^{-1}\,Z_1-
\overline{Z}_2^\dagger\,P^{-1}\,Q\,\overline{Z}_2-\|Z_1\|^2
-\|Z_2\|^2
\right). 
\end{eqnarray*}
\end{defn}
\begin{rem}
$\Psi_A$ plays an important role in the theory
of the metaplectic representation (see the discussions
in \S 4 of \cite{fo}, 
\cite{dau80}, \cite{zz18}).
\end{rem}

\begin{defn}
\label{defn:psi2}
If $(V,h')$ is a Hermitian complex $k$-dimensional vector space,
so that $h=g_h-\imath\,\omega_h$, where $g_h:=
\Re(h)$ and $\omega_h:=-\Im(h)$ are,
respectively, an Euclidean product and a 
symplectic bilinear form on $V$.
We define $\psi_2^{\omega_h}:V\times V\rightarrow \mathbb{C}$ by 
$$
\psi_2^{\omega_h}(v,v'):=
h(v,v')-\frac{1}{2}\,\|v\|_h^2-\frac{1}{2}\,\|v'\|_h^2
=-\imath\,\omega_h(v,v')-\frac{1}{2}\,\|v-v'\|_h^2,
$$
where $\|v\|_h:=\sqrt{h(v,v)}$.
\end{defn}

\begin{notn}
\label{notn:psi2}
When $(V,h)=\left(\mathbb{C}^k,h_0\right)$, where $h_0$ is the
standard Hermitian product, we shall also view $\psi_2:=\psi_2^{\omega_0}$ as 
being defined on $\mathbb{R}^{2\,k}\times \mathbb{R}^{2k}$, where
$\mathbb{R}^{2k}\cong \mathbb{C}^k$ in the standard manner.
If $\mathbf{v}_j\in \mathbb{R}^{2\,k}$,
$j=1,2$, corresponds to $Z_j\in \mathbb{C}^k$
under the previous identification, we shall equivalently
write $\psi_2(Z_1,Z_2)$ as $\psi_2(\mathbf{v}_1,\mathbf{v}_2)$ 
to emphasize the symplectic structure.

\end{notn}

\begin{rem}
Let us identify the unitary group $U(d-1)$ as the maximal compact subgroup
$\hat{U}(d-1):=\mathrm{Sp}(2d-2)\cap O(2d-2)$ of
$\mathrm{Sp}(2d-2)$ in the standard manner.
Then $A\in \hat{U}(d-1)\leqslant \mathrm{Sp}(2d-2)$ if and only if
 in (\ref{eqn:matrice APQ}) we have 
$Q=0$ and $P\in U(d-1)\leqslant \mathrm{GL}_\mathbb{C}(d-1)$.
Hence if $A\in \hat{U}(d-1)$ then
\begin{eqnarray*}
\Psi_{A}(\mathbf{v}_1,\mathbf{v}_2)
&=&
\overline{Z}_2^\dagger\, P^{-1}\,Z_1
-\frac{1}{2}\,\|Z_1\|^2-\frac{1}{2}\,\|Z_2\|^2\\
&=&h_0(Z_1,P\,Z_2) 
-\frac{1}{2}\,\|Z_1\|^2-\frac{1}{2}\,\|Z_2\|^2\\
&=&\psi_2(Z_1,P\,Z_2)=\psi_2(\mathbf{v}_1,A\,\mathbf{v}_2)
=\psi_2(A^{-1}\,\mathbf{v}_1,\mathbf{v}_2).
\end{eqnarray*}

\end{rem}

\begin{rem}
\label{rem:warning norm}
Let us identify $T_xX^\tau\cong \mathbb{R}\oplus \mathbb{R}^{2d-2}$
through NHLC's at $x$, 
hence $\mathcal{H}^\tau_x\cong \mathbb{R}^{2d-2}
\cong \mathbb{C}^{d-1}$. Then
the previous invariants may be viewed as defined on 
$\mathcal{H}^\tau_x$, and it is natural to expect that they have
an intrinsic geometric meaning. This is so, but the
symplectic and Euclidean pairings
on $\mathcal{H}^\tau_x$ corresponding
to $\omega_0=-\Im(h_0)$ and $g_0=\Re(h_0)$ are not
$\Omega_x$ and $\hat{\kappa}_x$, but rather their halves
$\omega_x:=\frac{1}{2}\,\Omega_x$ and $\tilde{\kappa}_x
:=\frac{1}{2}\,\hat{\kappa}_x$ (see \cite{p24}).
\end{rem}

Let us return to our geometric setting and consider,
in the action-free case, a pair $(x_1,x_2)\in \mathfrak{X}^\tau_\chi$.
Hence there exists $t\in \mathrm{supp}(\chi)$ such that
$x_1=\Gamma^\tau_{t}(x_2)$. It turns out that if 
$\mathrm{supp}(\chi)$ is sufficiently small (shorter than $2\,\epsilon$
for some $\epsilon>0$, say), then for \textit{any}  
$(x_1,x_2)\in \mathfrak{X}^\tau_\chi$ there exists a \textit{unique}  $t_1=t_1(x_1,x_2)\in \mathrm{supp}(\chi)$ such that 
$x_1=\Gamma^\tau_{t_1}(x_2)$
(see Lemma \ref{lem:unique t} below). If furthermore 
NHLC's on $X^\tau$ are chosen at $x_1$ and $x_2$, there is a unique
$B=B_{x_1,x_2}\in \mathrm{Sp}(2d-2)$ such that
\begin{equation}
\label{eqn:defn of M}
\Gamma^\tau_{-t_1}\big(x_1+(\theta,\mathbf{v})\big)=x_2+
\big(\theta+R_3(\theta,\mathbf{v}),\,B\mathbf{v}+
\mathbf{R}_2(\theta,\mathbf{v})    \big),
\end{equation}
where $R_k$ (respecively, $\mathbf{R}_k$) denotes, here and in the following,
a generic real-valued (respectively, vector-valued)
function on an open neighbourhood of the origin of some
Euclidean space, vanishing to $k$-th order at the origin;
(\ref{eqn:defn of M}) 
is a special case of Lemma \ref{lem:Mt1transl} below.

\begin{thm}
\label{thm:main 3}
Assume that
$\mathrm{supp}(\chi)$ is sufficiently small and 
$x_1\in x_2^{\mathrm{supp}(\chi)}$. Let
$t_1=t_1(x_1,x_2)$ be as above, and let $B=B_{x_1,x_2}$
be as in (\ref{eqn:defn of M}). 
Suppose $C>0$ and $\epsilon'\in (0,1/6)$. 
Then, uniformly for
$(\theta_j,\mathbf{v}_j)\in T_{x_j}X^\tau$ with
$\|(\theta_j,\mathbf{v}_j)\|\le C\,\lambda^{\epsilon'}$, there
are asymptotic expansions
\begin{eqnarray*}
\lefteqn{\Pi^\tau_{\chi,\lambda}(x_{1,\lambda},x_{2,\lambda})}\\
&\sim&\frac{1}{\sqrt{2\,\pi}}
\cdot \left(\frac{\lambda}{2\,\pi\,\tau}\right)^{d-1}
\cdot e^{\frac{1}{\tau}\,\left[\imath\,\sqrt{\lambda}\,(\theta_1-\theta_2)
+\Psi_{B^{-1}}(\mathbf{v}_1,\mathbf{v}_2)\right]}
\cdot  e^{-\imath\,\lambda\,t_1}
\\
&&\cdot e^{\imath\,\theta^\tau(x_1,x_2)}\cdot\left[\chi (t_1) +\sum_{k\ge 1}\,\lambda^{-k/2}\,F_k\left(x_1,x_2;\theta_1,\mathbf{v}_1,
\theta_2,\mathbf{v}_2\right)\right],
\nonumber
\end{eqnarray*}
where $e^{\imath\,\theta^\tau(x_1,x_2)},\,e^{\imath\,\tilde{\theta}^\tau(x_1,x_2)}\in S^1$ and 
$F_k(x_1,x_2;\cdot)$, $\widetilde{F}_k(x_1,x_2;\cdot)$ are 
polynomials in the rescaled variables of
degree $\le 3\,k$ and parity $k$. 
\end{thm}

When $x_1=x_2$ and $\chi\in \mathcal{C}^\infty_c\big((-\epsilon,
\epsilon)\big)$ for suitably small $\epsilon>0$ then
$M=I_{2d-2}$, $t_1=0$ and 
$\theta^\tau(x_1,x_2)=0$
Thus we recover the near-diagonal 
scaling asymptotic
in \cite{p24}. Furthermore, we recover the near-graph
scaling asymptotics of \cite{cr2}, 
with an explicit
determination of the leading factor, 
by fixing 
$(\theta_j,\mathbf{v}_j)$ and rescaling according to 
Heisenberg type (that is, in the form
$(\theta_j/\lambda,\mathbf{v}_j/\sqrt{\lambda})$).
The analogue of Theorem 15 for $P^\tau_{\chi,\nu,\lambda}$ is discussed in Section \ref{sctn:scaling Poissontau}.

\subsubsection{A decomposition of
$T_xX^\tau$ in the general equivariant case}
\label{sctn:decomp}

In order to state our results in the general equivariant setting,
we need to introduce a finer decomposition of $T_xX^\tau$,
valid at $x\in Z^\tau$, than
the one dictated by $\alpha^\tau$ in 
(\ref{eqn:dicrect sum vert hor}); this decomposition depends on 
$\mu^\tau$ and $\Phi$, and is the analogue of the decomposition appearing in the equivariant asymptotics in the line bundle setting (see \cite{mz08} and
\cite{p08}). 

\begin{notn}
\label{defn:induced vct field}
For any $\xi\in \mathfrak{g}$ (the Lie algebra of $G$)
we shall denote by $\xi_{X^\tau}$ the induced vector field
on $X^\tau$, and for any $x\in X^\tau$ we shall denote 
by
$\mathfrak{g}_{X^\tau}(x)\subseteq T_xX^\tau$ the tangent space
at $x$ to the $G$-orbit through $x$.
\end{notn}

\begin{rem}
Suppose $x\in Z^\tau$. Then
\begin{enumerate}
\item $\mathfrak{g}_{X^\tau}(x)\subseteq \mathcal{H}^\tau_x$;
\item since $\mathcal{H}^\tau_x=T_xX^\tau\cap J_x(T_xX^\tau)$,
we also have
$J_x\big(\mathfrak{g}_{X^\tau}(x)\big)\subset 
\mathcal{H}^\tau_x$;
\item under Assumption \ref{ass:ba}, 
$\dim \big(\mathfrak{g}_{X^\tau}(x)\big)=d_G$.
\end{enumerate}

\end{rem}

\begin{defn}
\label{defn:vth dec Ztau}
Suppose $x\in Z^\tau$. We set
$$
T_x^vX^\tau:=\mathfrak{g}_{X^\tau}(x),\quad
T_x^tX^\tau:=J_x\big(\mathfrak{g}_{X^\tau}(x)\big),
\quad 
T_x^hX^\tau:=\mathcal{H}^\tau_x\cap \left(T_x^vX^\tau
\oplus T^t_xX^\tau\right)^\perp,
$$
where $\perp$ denotes the Hermitian (equivalently,
symplectic or Riemannian) orthocomplement of the complex
subspace $T_x^vX^\tau
\oplus T^t_xX^\tau\subseteq T_x\tilde{M}$.

We shall refer to $T_x^vX^\tau,\,T_x^tX^\tau,\,T_x^hX^\tau$
as the $\tilde{\mu}^\tau$-vertical, $\tilde{\mu}^\tau$-transverse,
and $\tilde{\mu}^\tau$-horizontal tangent spaces at $x$, 
respectively (the terminology being inspired by the
line bundle setting).

We shall accordingly decompose 
$\mathbf{v}\in \mathcal{H}^\tau_x$ as 
$$
\mathbf{v}=\mathbf{v}^t+\mathbf{v}^v+\mathbf{v}^h,
\quad\text{where}
\quad
\mathbf{v}^t\in T_x^tX^\tau,\,
\mathbf{v}^v\in T^v_xX^\tau,\,
\mathbf{v}^h\in T_x^hX^\tau.
$$

\end{defn}

\begin{rem}
\label{rem:mutau decomp}
At any $x\in X^\tau$, we have direct sum decompositions 
$$
\mathcal{H}^\tau_x=\left(T_x^vX^\tau\oplus 
T^t_x X^\tau \right)\oplus 
T_x^hX^\tau,
$$
and
$$
T_xZ^\tau=\mathcal{T}^\tau_x\oplus \left(T_x^vX^\tau
\oplus T_x^hX^\tau\right),
\qquad
T_xX^\tau=T_x^tX^\tau\oplus T_xZ^\tau,
$$
so that $T_x^tX^\tau$ is the normal space to $Z^\tau$
at $x$ (in $X^\tau$).
\end{rem}

\begin{notn}
\label{rem:NHLC euclid decomp}
Given a choice of NHLC's centered for $X^\tau$ at $x$, we obtain
a direct sum decomposition of 
$\mathbb{R}^{2\,d-1}=\mathbb{R}\oplus \mathbb{C}^{d-1}
\cong T_xX^\tau$, dictated by $\alpha^\tau$, where
$\mathbb{C}^{d-1}\cong \mathcal{H}^{\tau}_x$. 
The finer decomposition of $\mathcal{H}^{\tau}_x$ in 
Definition \ref{defn:vth dec Ztau}, dictated by 
$\tilde{\mu}^\tau$, determines a 
corresponding decomposition of 
$\mathbb{C}^{d-1}$, which - emphasizing dimensions -
we shall write in the form
$$
\mathbb{C}^{d-1}=\mathbb{R}^{d_G}_t\oplus
\mathbb{R}^{d_G}_v\oplus \mathbb{C}^{d-1-d_G}_h,
\quad\text{where}\quad 
\mathbb{R}^{d_G}_t=J_0( \mathbb{R}^{d_G}_v ).
$$
We shall shift from real to complex notation, 
identifying 
$$
\mathbb{C}^{d-1}\cong \mathbb{R}^{2d-2},\quad
\mathbb{C}^{d-1-d_G}_h=\mathbb{R}^{2d-2-2d_G}_h.
$$
Any $\mathbf{u}\in \mathbb{R}^{2d-2}$ may thus be decomposed as
$$
\mathbf{u}=\mathbf{u}^t+\mathbf{u}^v+\mathbf{v}^h,\quad 
\text{where}\quad 
\mathbf{u}^t\in \mathbb{R}^{d_G}_t,\,
\mathbf{u}^v\in \mathbb{R}^{d_G}_v,\,
\mathbf{u}^h\in \mathbb{R}^{2d-2-2d_G}_h.
$$
\end{notn}

\subsubsection{Scaling asymptotics for $\Pi^\tau_{\chi,\nu,\lambda}$}
\label{sctn:scaling Pitau}

For the sake of brevity,
we shall first discuss the asymptotic expansions for
$\Pi^\tau_{\chi,\nu,\lambda}$, and then explain the
necessary changes for $P^\tau_{\chi,\nu,\lambda}$.

Before stating our results, we need to make some recalls
and introduce 
some further notation.

If $x\in X^\tau$, we shall  denote by
$G_x\leqslant G$ the stabilizer subgroup
of $x$, and by $r_x=|G_x|$ its cardinality
($r_x$ is always finite if $x\in Z^\tau$).

Assume as before that Assumption \ref{ass:ba} holds
and that $\mathrm{supp}(\chi)$ is sufficiently small.
Then for any $(x_1,x_2)\in \mathfrak{X}_\chi^\tau$
the following holds.

\begin{enumerate}
\item There exists a unique $t_1=t_1(x_1,x_2)$ such that
$x_1=\tilde{\mu}_g^\tau\circ \Gamma^\tau_{t_1}(x_2)$ for some
$g\in G$ (Lemma \ref{lem:unique t} below). 
\item There are exactly $r_{x_1}$
elements $g_l\in G$
such that $x_1=\tilde{\mu}_{g_l}^\tau\circ \Gamma^\tau_{t_1}(x_2)$
(see Corollary \ref{cor:unique t1chi} and Remark 
\ref{rem:unique t1chi} below).

\item The condition $(x_1,x_2)\in \mathfrak{X}_\chi^\tau$ is
tantamount to $x_1^G\cap x_2^\chi\neq \emptyset$, and one has
\begin{equation}
\label{eqn:defn di x12}
x_1^G\cap x_2^\chi=\{x_{12}\},\quad
\text{where}\quad
x_{12}:= \Gamma^\tau_{t_1}(x_2).
\end{equation}

\end{enumerate}

\begin{defn}
\label{defn:effective volume}
Under Assumption \ref{ass:ba},  
the \textit{effective volume} at $x\in Z^\tau$, denoted
$V_{eff}(x)$, is the volume of the $G$-orbit through
$x$, $x^G\subset X^\tau$ for the induced
Riemannian density (with respect to $\tilde{\kappa}$ - see also
Remark \ref{rem:warning norm}).

\end{defn}

\begin{defn}
\label{eqn:defn di Achi}
Assume as above that $\chi$ has sufficiently small support,
and that $(x_1,x_2)\in \mathfrak{X}_\chi^\tau$.
Let $t_1(x_1,x_2)\in \mathrm{supp}(\chi)$ and $x_{12}\in x_2^\chi$ be
as in (\ref{eqn:defn di x12}).
Given any choice of NHLC's at $x_{12}$ and $x_2$, let $B\in 
\mathrm{Sp}(2d-2)$ be defined by the analogue of
(\ref{eqn:defn of M}):
\begin{equation}
\label{eqn:defn of M12}
\Gamma^\tau_{-t_1}\big(x_{12}+(\theta,\mathbf{v})\big)=x_2+
\big(\theta+R_3(\theta,\mathbf{v}),\,B\mathbf{v}+
\mathbf{R}_2(\theta,\mathbf{v})    \big)
\end{equation}
(see Lemma \ref{lem:Mt1transl} below and the notational remark
preceding Theorem \ref{thm:main 3}).
Referring to the direct sum decomposition in
Notation \ref{rem:NHLC euclid decomp},
let us set 
\begin{eqnarray}
\label{eqn:comp of Achi}
A_\chi(x_1,x_2)
&:=&\int_{\mathbb{R}^{d_G}_t}\,\mathrm{d}\mathbf{u}^t\,
\int_{\mathbb{R}^{d_G}_v}\,\mathrm{d}\mathbf{u}^v\,
\int_{\mathbb{R}^{2d-2-2d_G}_h}\,\mathrm{d}\mathbf{u}^h\nonumber\\
&&\left[         
e^{-\left\|\mathbf{u}^t\right\|^2-\frac{1}{2}\,\left\|\mathbf{u}^h \right\|^2
-\imath\,\omega_0\left(\mathbf{u}^v,\mathbf{u}^t\right)
-\frac{1}{2}\,\| B\mathbf{u}\|^2}
\right].
\end{eqnarray}

\end{defn}

\begin{defn}
\label{defn:FchiBnu}
Let $x_1,\,x_2\in X^\tau$ be as in (\ref{eqn:defn di x12}), 
$M$ as in (\ref{eqn:defn of M12}), 
and $A_\chi (x_1,x_2)$ as in Definition \ref{eqn:defn di Achi}; 
define $P$ by
(\ref{eqn:matrice APQ}), with $A=B$. We set
\begin{eqnarray*}
\mathcal{F}_{\chi}(x_1,x_2)&:=&
\chi(t_1)
\cdot\frac{\left|\det(P)\right|}{r\cdot V_{eff}(x_1)}\cdot 
\frac{A_\chi (x_1,x_2)}{\pi^{d-1}},
\\
\mathcal{B}_\nu(x_1,x_2)_l&:=&\dim(\nu)\cdot
\overline{\Xi_\nu 
\left(g_l\right)},
\end{eqnarray*}

\end{defn}

\begin{rem}
\label{rem:Achi prop}
Let us note the following:
\begin{enumerate}
\item Under the previous assumptions, $A_\chi(x_1,x_2)$ is finite and non-zero, 
and can
in principle be computed in terms of $B$ (see Appendix A of \cite{fo});
\item if $B$ is orthogonal (e.g., if $B$ is the
identity matrix), by an interated Gaussian integration one obtains
$A_\chi(x_1,x_2)=\pi^{d-1}$;
\item if $B$ is orthogonal one also has $|\det(P)|=1$ (since in this case $P$ is a
unitary complex matrix), hence 
$\mathcal{F}_{\chi}(x_1,x_2)=
\chi(t_1)/\big(r\cdot V_{eff}(x_1)\big)$.
\end{enumerate}

\end{rem}

In view of Theorem \ref{thm:main 1},
we expect an exponential decay of 
$\Pi^\tau_{\chi,\nu,\lambda}(x_{1,\lambda},x_{2,\lambda})$
(and $P^\tau_{\chi,\nu,\lambda}(x_{1,\lambda},x_{2,\lambda})$)
along normal displacements to $Z^\tau$, that is, when 
$\mathbf{v}_j=\mathbf{v}_j^t\in T_{x_j}^tX^\tau$,
and we wish to determine the leading exponent governing.
Accordingly,
in the general equivariant case we restrict first to dispacements along directions in $T_{x_j}^tX^\tau$
(Theorem \ref{thm:main 2}, case 1); 
in this case, the leading exponent is determined 
explicitly.
In the same setting, an expansion is also obtained, although with a less explicit determination of the exponent, for more general displacements, provided we take $\mathbf{v}_2^h=0$
(Theorem \ref{thm:main 2}, case 2).
As discussed at the beginning of \S \ref{sctn:scaling}, 
this unequal treatment of $\mathbf{v}_1$ and $\mathbf{v}_2$
is aimed at making the computations more tractable,
while obtaining managable results that are sufficient
for the applications.

\begin{thm}
\label{thm:main 2}
Suppose that $\chi\in \mathcal{C}^\infty_c(\mathbb{R})$
has sufficiently small support, 
and that $(x_1,x_2)\in \mathfrak{X}^\tau_\chi$.
Let $g_1,\ldots,g_{r_{x_1}}\in G$ be the distinct elements such
that $x_1=\mu^\tau_{g_l}\circ \Gamma^\tau_{t_1}(x_2)$,
where $t_1\in \mathbb{R}$ is as in (\ref{eqn:defn di x12}).
Fix $C>0$ and $\epsilon'\in (0,1/6)$.
\begin{enumerate}
\item Uniformly for
$$\theta_j\in \mathbb{R},\quad
\mathbf{v}_j=\mathbf{v}_j^t\in T^t_{x_j}X^\tau,\quad
|\theta_j|,\,\|\mathbf{v}_j^t\|\le C\,\lambda^{\epsilon'}$$ 
we have for $\lambda\rightarrow +\infty$
$$
\Pi_{\chi,\nu,\lambda}^\tau(x_{1,\lambda},x_{2,\lambda})\sim 
\sum_{l=1}^{r_{x_1}}
\Pi_{\chi,\nu,\lambda}^\tau(x_{1,\lambda},x_{2,\lambda})_l
$$
where for each $l$
there is an asymptotic expansion 
\begin{eqnarray*}
\lefteqn{\Pi_{\chi,\nu,\lambda}^\tau(x_{1,\lambda},x_{2,\lambda})_l}\\
&\sim&e^{-\imath\,\lambda\,t_1}\cdot\frac{1}{\sqrt{2\,\pi}}\,
\left( \frac{\lambda}{2\,\pi\,\tau}  \right)^{d-1-d_G/2}
\cdot e^{\frac{1}{\tau}\left[\imath\,\sqrt{\lambda}\,
\,(\theta_1-\theta_2)
-\left(\|\mathbf{v}_1^t\|^2
+\|\mathbf{v}_2^t\|^2\right)\right]}
\nonumber\\
&&\cdot 
\left[e^{\imath\,\theta^\tau_{t_1}(x_1)}\cdot 
\mathcal{F}_{\chi}(x_1,x_2)
\cdot \mathcal{B}_\nu(x_1,x_2)_l+\sum_{k\ge 1}\,\lambda^{-k/2}\,
F_{k,l,\nu}\left(x_1,x_2;\theta_1,\mathbf{v}_1^t,
\theta_2,\mathbf{v}_2^t\right)\right];
\end{eqnarray*}
here
$e^{\imath\,\theta^\tau_{t_1}(x_1)}\in S^1$ and
$F_{k,l,\nu}(x_1,x_2;\cdot)$ is a polynomial in the rescaled variables of
degree $\le 3\,k$ and parity $k$.

\item 
Under the same assumptions and with the same notation,
uniformly for
$$\theta_j\in \mathbb{R},
\quad |\theta_j|\le C\,\lambda^{\epsilon'}
$$
$$
\mathbf{v}_1=\mathbf{v}_1^t+\mathbf{v}_1^h\in 
T^t_{x_1}X^\tau\oplus T^h_{x_1}X^\tau,\quad
\|\mathbf{v}_1\|\le C\,\lambda^{\epsilon'}
$$
$$
\mathbf{v}_2=\mathbf{v}_2^t\in T^t_{x_2}X^\tau,
\quad \|\mathbf{v}_2\|\le C\,\lambda^{\epsilon'},
$$
we have for $\lambda\rightarrow +\infty$
$$
\Pi_{\chi,\nu,\lambda}^\tau(x_{1,\lambda},x_{2,\lambda})\sim 
\sum_{l=1}^{r_{x_1}}\Pi_{\chi,\nu,\lambda}^\tau(x_{1,\lambda},x_{2,\lambda})_l
$$
where for each $l$
there is an asymptotic expansion 
\begin{eqnarray*}
\lefteqn{\Pi_{\chi,\nu,\lambda}^\tau(x_{1,\lambda},x_{2,\lambda})_l}\\
&\sim&e^{-\imath\,\lambda\,t_1}\cdot\frac{1}{\sqrt{2\,\pi}}\,
\left( \frac{\lambda}{2\,\pi\,\tau}  \right)^{d-1-d_G/2}
\cdot
e^{\frac{1}{\tau}
\,\left[\imath\,\sqrt{\lambda}\,(\theta_1-\theta_2)
-\frac{1}{2}\,\langle \mathbf{V},D_l\,\mathbf{V}\rangle\right]}
\nonumber\\
&&\cdot 
\left[e^{\imath\,\theta^\tau_{t_1}(x_1)}\cdot 
\mathcal{F}_{\chi}(x_1,x_2)
\cdot \mathcal{B}_\nu(x_1,x_2)_l+\sum_{k\ge 1}\,\lambda^{-k/2}\,
\tilde{F}_{k,l,\nu}\left(x_1,x_2;\theta_1,
\theta_2,\mathbf{V}\right)\right];
\end{eqnarray*}
here 
$\tilde{F}_{k,l}(x_1,x_2;\cdot)$ is a polynomial in the rescaled variables of
degree $\le 3\,k$ and parity $k$,
$$
\mathbf{V}^\dagger:=\begin{pmatrix}
(\mathbf{v}_1^t)^\dagger&(\mathbf{v}_1^h)^\dagger&
(\mathbf{v}_2^t)^\dagger
\end{pmatrix},
$$
and $D_l=D_l^\dagger$, $\Re(D_l)\gg 0$ for every $l$.

\end{enumerate}

\end{thm}

\begin{rem}
\label{rem:unitary factor Pi}
The unitary factor $e^{\imath\,\theta^\tau_{t}(x)}$  
is a smooth function on $\mathbb{R}\times X^\tau$ and is
related to the description (by Zelditch)
of the unitary Toeplitz operators
$e^{\imath\,t\,\mathfrak{D}^\tau_{\sqrt{\rho}}}$ 
as \lq dynamical
Toeplitz operators\rq\,(recall (\ref{eqn:DsqrtrhoToeplitz})); it is therefore an intrinsic invariant, given \textit{a priori}, of
the CR structure of $X^\tau$.
Similar considerations apply to the unitary factor 
$e^{\imath\,\tilde{\theta}^\tau_{t}(x)}$ appearing
in the asymptotics of $P^\tau_{\chi,\nu,\lambda}$ of
\S \ref{sctn:scaling Poissontau},
with $e^{\imath\,t\,\mathfrak{D}^\tau_{\sqrt{\rho}}}$
replaced by $U_\mathbb{C}(t+2\,\imath\,\tau)$
in (\ref{eqn:UC_spectral}) below - see (\ref{eqn:U_C compos})
and (\ref{eqn:leading coeff poisson}) below.
\end{rem}

Let us dwell on the special case where 
$$
x_1=x_2=x\in Z^\tau\quad
\text{and} \quad\chi\in 
\mathcal{C}^\infty_c\big((-\epsilon,\epsilon)\big),
$$ 
so that
$t_1=0$ and $e^{\imath\,\theta^\tau_{t}(x)}=1$.
Furthermore, we fix one system of NHLC's at $x$, in terms of
which
in (\ref{eqn:rescaled coord 12}) we write
\begin{equation}
\label{eqn:xx12}
x_{1,\lambda}:=x+\left(\frac{\theta_1}{\sqrt{\lambda}},
\frac{\mathbf{v}_1}{\sqrt{\lambda}}   \right),
\quad
x_{2,\lambda}:=x+\left(\frac{\theta_2}{\sqrt{\lambda}},
\frac{\mathbf{v}_2}{\sqrt{\lambda}}   \right).
\end{equation}
In (\ref{eqn:defn of M12}) we then have $M=I_{2d-2}$.

\begin{notn}
\label{notn:stabilizer action}
Suppose $G_{x}=\{\kappa_l\}_{l=1}^{r_{x_1}}$.
For every $l$, $\mathrm{d}_x\tilde{\mu}^\tau_{\kappa_l}:
T_xX^\tau\rightarrow T_xX^\tau$ satisfies
$$
\mathcal{R}^\tau(x)=
\mathrm{d}_x\tilde{\mu}^\tau_{\kappa_l}\left(
\mathcal{R}^\tau(x)\right)\quad\Rightarrow\quad
\mathrm{d}_x\tilde{\mu}^\tau_{\kappa_l}\left(
\left.\frac{\partial}{\partial \theta}\right|_x\right)
=\left.\frac{\partial}{\partial \theta}\right|_x
$$
and $\mathrm{d}_x\tilde{\mu}^\tau_{\kappa_l}\left(
\mathcal{H}^\tau_x\right)= 
\mathcal{H}^\tau_x$
(recall (\ref{eqn:reeb vct field}) and 
(\ref{eqn:Reeb theta v H})).
Thus the action of $\kappa_l\in G_x$ on 
$T_xX^\tau=\mathcal{T}^\tau_x\oplus \mathcal{H}^\tau_x
\cong \mathbb{R}\times \mathbb{R}^{2d-2}$ has the
form
$$
\mathrm{d}_x\tilde{\mu}^\tau_{\kappa_l}(\theta,\mathbf{u})=
(\theta,\mathbf{u})=
\left(\theta,\mathrm{d}_x\tilde{\mu}^\tau_{\kappa_l}(\mathbf{u})
\right).
$$
In the following, we shall adopt the short-hand
$$
\mathbf{u}^{(l)}:=
\mathrm{d}_x\tilde{\mu}^\tau_{\kappa_l^{-1}}(\mathbf{u})
\quad\text{for}\quad \mathbf{u}\in \mathbb{R}^{2d-2}
\cong \mathcal{H}^\tau_x.
$$
Furthermore, if $\mathbf{u}=\mathbf{u}^t+\mathbf{u}^v+\mathbf{u}^h$
then $(\mathbf{u}^t)^{(l)}=(\mathbf{u}^{(l)})^t$, 
$(\mathbf{u}^v)^{(l)}=(\mathbf{u}^{(l)})^v$,
$(\mathbf{u}^h)^{(l)}=(\mathbf{u}^{(l)})^h$.
\end{notn}

\begin{thm}
\label{thm:diagonal case}
Under the assumptions and with the notation 
of Theorem \ref{thm:main 2},
if $\chi\in \mathcal{C}^\infty_c\big((-\epsilon,\epsilon)\big)$
and $x_1=x_2=x\in Z^\tau$, then with the previous notation
the following holds:
uniformly for
$$\theta_j\in \mathbb{R},\quad\mathbf{v}_j=\mathbf{v}_j^h
+\mathbf{v}_j^t\in 
T^h_{x_j}X^\tau\oplus T^t_{x_j}X^\tau,\quad 
|\theta_j|,\,\|\mathbf{v}_j^t\|\le C\,\lambda^{\epsilon'},
$$ 
we have for $\lambda\rightarrow +\infty$
$$
\Pi_{\chi,\nu,\lambda}^\tau(x_{1,\lambda},x_{2,\lambda})\sim 
\sum_{l=1}^{r_{x_1}}\Pi_{\chi,\nu,\lambda}^\tau(x_{1,\lambda},x_{2,\lambda})_l
$$
where for each $l$
there is an asymptotic expansion 
\begin{eqnarray*}
\Pi^\tau_{\chi,\nu,\lambda}(x_{1,\lambda},x_{2,\lambda})_l
&\sim&\frac{1}{\sqrt{2\,\pi}}\cdot
\left( \frac{\lambda}{2\,\pi\,\tau}\right)^{d-1-d_G/2}\, 
\frac{\dim(\nu)}{r_x\cdot V_{eff}(x_1)}\cdot \nonumber\\
&&\cdot e^{\frac{1}{\tau}\,\left[\imath\,\sqrt{\lambda}\,
(\theta_1-\theta_2)
-\left\|\mathbf{v}_1^t\right\|^2
-\left\|\mathbf{v}_2^t\right\|^2+
\psi_2\left((\mathbf{v}^h_1)^{(l)}, \mathbf{v}^h_2 \right)\right]}
\\
&&\cdot \left[\chi (0)\cdot\overline{\Xi_\nu 
\left(\kappa_l\right)}+\sum_{k\ge 1}\,\lambda^{-k/2}\,F_{k,l,\nu}\left(x_1,x_2;\theta_1,\mathbf{v}_1^t,
\theta_2,\mathbf{v}_2^t\right)\right];
\nonumber
\end{eqnarray*}
here $F_{k,l,\nu}(x_1,x_2;\cdot)$ is a polynomial in the rescaled variables of
degree $\le 3\,k$ and parity $k$.
\end{thm}

Theorem \ref{thm:diagonal case} is the analogue in the Grauert tube setting
of the equivariant Szeg\H{o} kernel asymptotics in the line bundle setting of \cite{p08}.

\subsubsection{Scaling asymptotics for $P^\tau_{\chi,\nu,\lambda}$}
\label{sctn:scaling Poissontau}

The scaling asymptotics of $P^\tau_{\chi,\nu,\lambda}$
can be studied, with some adaptations, 
by arguments and techniques similar to
those used for $\Pi^\tau_{\chi,\nu,\lambda}$.

\begin{thm}
\label{thm:complexified Poisson wave asy}
The statements of Theorems \ref{thm:main 1}, \ref{thm:main 3},
\ref{thm:main 2} and
\ref{thm:diagonal case} 
apply with 
$P^\tau_{\chi,\nu,\lambda}$ in place of $\Pi^\tau_{\chi,\nu,\lambda}$,
with the following changes:
\begin{enumerate}
\item the leading order terms of the asymptotic expansions in
Theorems \ref{thm:main 2}, \ref{thm:main 3} and \ref{thm:diagonal case}
are multiplied by $(\lambda/\pi\,\tau)^{-(d-1)/2}$;
\item the real smooth function $\theta^\tau_t$ in the definition of
$F_{\chi,\lambda}$ is replaced by a possibly different function
real smooth function $\tilde{\theta}^\tau_t$.
\end{enumerate}
\end{thm}

In the action-free case, for example, rather than the expansion
in Theorem \ref{thm:main 3} we have
\begin{eqnarray*}
\lefteqn{P^\tau_{\chi,\lambda}(x_{1,\lambda},x_{2,\lambda})}\\
&\sim&\frac{e^{-\imath\,\lambda\,t_1}}{\sqrt{2\,\pi}}
\cdot \left(\frac{1}{2}\right)^{d-1}\cdot 
\left(\frac{\lambda}{\tau\,\pi}\right)^{\frac{d-1}{2}}
\cdot e^{\frac{1}{\tau}\,\left[\imath\,\sqrt{\lambda}\,(\theta_1-\theta_2)
+\Psi_{M^{-1}}(\mathbf{v}_1,\mathbf{v}_2)\right]}
\cdot  
e^{\imath\,\tilde{\theta}^\tau_{t_1}(x)}\,
\\
&&\cdot \left[\chi (t_1) +\sum_{k\ge 1}\,\lambda^{-k/2}\,\widetilde{F}_k\left(x_1,x_2;\theta_1,\mathbf{v}_1,
\theta_2,\mathbf{v}_2\right)\right],
\nonumber
\end{eqnarray*}
for certain polynomials $\widetilde{F}_k(x_1,x_2;\cdot)$ 
in the rescaled variables of
degree $\le 3\,k$ and parity $k$. 
In the yet more particular near-diagonal case, 
the previous expansion corrects the leading order
factor appearing in the expansion in Theorem 7 of \cite{p24} by a power of $\pi^{(d-1)/2}$;
this is related to the symbolic computation in
\S \ref{sctn:complexified eigenf lapl} below
(see Lemma \ref{lem:princ symbol Qtau}).

The original expansion in Theorem 1.9 of \cite{cr2} follows by rescaling according
to Heisenberg type (as in the Introduction
of \cite{p24}).

\subsection{Unrescaled Asymptotics}

The equivariant scaling asymptotic expansions of Theorems \ref{thm:main 2},
\ref{thm:diagonal case},
and \ref{thm:complexified Poisson wave asy}
hold uniformly, say, on compact subsets
of the locus in $Z^\tau$ with principal orbit type, but
it is not \textit{a priori} obvious that they hold uniformly near
the locus where the cardinality of the stabilizer has
a discontinuity. 
In the following Theorem, which rests on the previous ones,
we establish a near-graph 
unrescaled version of the previous
asymptotic expansions; these allow for some
uniform estimates which, in turn, will be useful in the following
applications.

We give an explicit statement and proof in the case of 
$\Pi^\tau_{\chi,\nu,\lambda}$; the 
extension to $P^\tau_{\chi,\nu,\lambda}$
can be carried out by the same arguments used for Theorem
\ref{thm:complexified Poisson wave asy}, and will be left to the reader. 

\begin{thm}
\label{thm:near-graph-unrescaled}
With the previous assumptions and notation, the following holds
for $\lambda\rightarrow +\infty$.
\begin{enumerate}
\item 
For small enough 
$$
(\theta_j,\mathbf{v}_j^t)\in \mathbb{R}\times \mathbb{R}^{d_G}_t\cong \mathcal{T}^\tau_{x_j}\oplus
T^t_{x_j}X^\tau,
$$
there is an asymptotic expansion 
$$
\Pi^\tau_{\chi,\nu,\lambda}
\big(x_1+(\theta_1,\mathbf{v}_1^t),
x_2+(\theta_2,\mathbf{v}_2^t)\big)\sim
\sum_{g_{l}\in G_{x_1}}
\Pi^\tau_{\chi,\nu,\lambda}\big(x_1+(\theta_1,\mathbf{v}_1^t),
x_2+(\theta_2,\mathbf{v}_2^t)\big)_l,
$$
where for each $l$
\begin{eqnarray*}
\lefteqn{\Pi^\tau_{\chi,\nu,\lambda}\big(x_1+(\theta_1,\mathbf{v}_1^t),
x_2+(\theta_2,\mathbf{v}_2^t)\big)_l}\\
&\sim&
e^{-\imath\,\lambda\,t_1}\,\frac{1}{\sqrt{2\,\pi}}\,
\left( \frac{\lambda}{2\,\pi\,\tau}  \right)^{d-1-d_G/2}\\
&&
\cdot e^{\frac{\lambda}{\tau}\,\left[
\imath\,
\,(\theta_1-\theta_2)
-\left(\|\mathbf{v}_1^t\|^2
+\|\mathbf{v}_2^t\|^2\right)+R_3\left(\theta_1,\mathbf{v}_1^t,
\theta_2,\mathbf{v}_2^t\right)\right]}
\nonumber\\
&&\cdot 
\left[e^{\imath\,\theta^\tau_{t_1}(x_1)}\cdot \mathcal{F}_{\chi}(x_1,x_2)
\cdot \mathcal{B}_\nu(x_1,x_2)_l+\sum_{s\ge 1}\,\lambda^{-s}\,
L_{\nu,l,s}\left(x_1,x_2;\theta_1,\mathbf{v}_1^t,
\theta_2,\mathbf{v}_2^t\right)\right],
\nonumber
\end{eqnarray*}
where $R_3$ vanishes to third order at the origin
and 
$$
\Re\left[(\imath\,
\,(\theta_1-\theta_2)
-\left(\|\mathbf{v}_1^t\|^2
+\|\mathbf{v}_2^t\|^2\right)+R_3\left(\theta_1,\mathbf{v}_1^t,
\theta_2,\mathbf{v}_2^t\right)\right]\le 0.
$$ 
\item For small enough 
$(\theta_j,\mathbf{v}_j)\in \mathbb{R}\times \mathbb{R}^{2d-2}\cong 
T_{x_j}X^\tau=
\mathcal{T}^\tau_{x_j}\oplus \mathcal{H}^{\tau}_{x_j}$ 
with
$$
\mathbf{v}_1=\mathbf{v}_1^t+\mathbf{v}_2^h\in 
T^t_{x_1}X^\tau\oplus T^h_{x_1}X^\tau,\quad
\mathbf{v}_2=\mathbf{v}_2^t\in T^t_{x_2}X^\tau,
$$
there is an asymptotic expansion 
$$
\Pi^\tau_{\chi,\nu,\lambda}
\big(x_1+(\theta_1,\mathbf{v}_1),
x_2+(\theta_2,\mathbf{v}_2)\big)\sim
\sum_{g_{l}\in G_{x_1}}
\Pi^\tau_{\chi,\nu,\lambda}\big(x_1+(\theta_1,\mathbf{v}_1),
x_2+(\theta_2,\mathbf{v}_2)\big)_l,
$$
where for each $l$
\begin{eqnarray*}
\lefteqn{\Pi^\tau_{\chi,\nu,\lambda}
\big(x_1+(\theta_1,\mathbf{v}_1),
x_2+(\theta_2,\mathbf{v}_2)\big)_l}\\
&\sim&
e^{-\imath\,\lambda\,t_1}\,\frac{1}{\sqrt{2\,\pi}}\,
\left( \frac{\lambda}{2\,\pi\,\tau}  \right)^{d-1-d_G/2}\\
&&
\cdot e^{\frac{\lambda}{\tau}\,\left[\imath\,
\,(\theta_1-\theta_2)
-\frac{1}{2}\,\langle \mathbf{V},D_l\,\mathbf{V}\rangle+R_3\left(\theta_1,
\theta_2,\mathbf{V}\right)\right]}
\nonumber\\
&&\cdot 
\left[e^{\imath\,\theta^\tau_{t_1}(x_1)}\cdot \mathcal{F}_{\chi}(x_1,x_2)
\cdot \mathcal{B}_\nu(x_1,x_2)_l+\sum_{s\ge 1}\,\lambda^{-s}\,
K_{\nu,l,s}\left(x_1,x_2;\theta_1,
\theta_2,\mathbf{V}\right)\right],
\nonumber
\end{eqnarray*}
\end{enumerate}
where $\mathbf{V}$ is as in Theorem \ref{thm:main 2}, 
$R_3$ vanishes to third order at the origin, and
$$\Re\left[ \imath\,
\,(\theta_1-\theta_2)
-\frac{1}{2}\,\langle \mathbf{V},D_l\,\mathbf{V}\rangle+R_3\left(\theta_1,
\theta_2,\mathbf{V}\right)  \right]\le 0.$$
\end{thm}

\subsection{Applications}
We conclude this Introduction with a sample of some of the
possible applications 
of the previous asymptotics.

\subsubsection{An equivariant global Weyl law for $\mathfrak{D}^\tau_{\sqrt{\rho}}$}\label{sctn:intro weyl Pitau}

The previous expansions may be integrated to obtain global information on the asymptotic
distribution of the $\lambda_j$'s. 
Recalling (\ref{eqn:L2Mdecompequiv}),
let us set
\begin{equation} \label{eqn:defn di Wtaonu}
\mathcal{W}_\nu^\tau(\lambda):=
\sum_{\lambda_j\le \lambda}\,
\dim H^\tau(X)_{j,\nu} \qquad (\lambda\in \mathbb{R}).
\end{equation}

We shall make here the simplifying assumption that $G$ acts freely on $Z^\tau$; then the quotient $Z^\tau/G$
is a compact manifold; it has a uniquely determined Riemannian structure
such that the projection $Z^\tau\rightarrow Z^\tau/G$ is a Riemannian
submersion. Let $\mathrm{vol}(Z^\tau/G)$ be the corresponding
volume of $Z^\tau/G$.

\begin{thm}
\label{thm:weyl law geod}
In the situation of Theorem \ref{thm:diagonal case},
let us assume that $G$ acts freely on
$Z^\tau$. 
Then as $\lambda\rightarrow +\infty$
\begin{eqnarray*}
\mathcal{W}_\nu^\tau(\lambda)&=&
\frac{1}{2^{d_G/2}}\cdot 
\frac{\tau}{d-d_G}\cdot \left( \frac{\lambda}{2\,\pi\,\tau}  \right)^{d-d_G}
\cdot \dim(\nu)^2\cdot
\mathrm{vol}(Z^\tau/G)\nonumber\\
&&
\cdot \left[1+O\left(  \lambda^{-1} \right)\right]\nonumber
\end{eqnarray*}
 
\end{thm}

\subsubsection{Pointwise estimates on eigenfunctions of 
$\mathfrak{D}^\tau_{\sqrt{\rho}}$}

\label{sctn:pointwise est}

Theorem 
\ref{thm:near-graph-unrescaled} has the following 
following straighforward consequence.

\begin{prop}
\label{prop:uniform diagonal estimate}
There exists
 $C_\nu,\,\lambda_\nu^\tau>0$ such that
 $$
\Pi^\tau_{\chi,\nu,\lambda}(x,x)\le C_\nu\,
\left(\frac{\lambda}{\tau}\right)^{d-1-d_G/2}
\qquad \forall\,x\in X^\tau, \,\lambda\ge \lambda_\nu^\tau.
$$

\end{prop}

\begin{cor}
\label{cor:pointwise eigenfunction est}
There exist $C_\nu>0$ and $j_\nu^\tau>0$ such that
for any $j\ge j_\nu^\tau$ and $\lambda\ge \lambda_\nu^\tau$
$$
\sum_k |\rho_{j,\nu,k}(x)|^2\le C_\nu\,\left(
\frac{\lambda_j}{\tau}\right)^{d-1-d_G/2}
\qquad (x\in X^\tau).
$$

\end{cor}

\subsubsection{Operator norm estimates
for $\Pi^\tau_{\chi,\nu,\lambda}$}
\label{scn:opne}
In the action-free case, Chang and Rabinowitz have
established in \cite{cr2}
operator norm estimates on
$\Pi^\tau_{\chi,\lambda}:L^p(X^\tau)\rightarrow 
L^q(X^\tau)$. 
Namely, they proved that
for $2\le p,q\le +\infty$ there
exist constants $C^\tau_{p}>0$ such that 
\begin{equation}
\label{eqn:cr norm estimate}
\left\|\Pi^\tau_{\chi,\lambda}\right\|_{L^p\rightarrow L^q}
\le C^\tau_{p}\,\lambda^{(d-1)\,
\left( \frac{1}{p}-\frac{1}{q}  \right)}.
\end{equation}
This result is the analogue of similar estimates proved in the 
line bundle setting by Shiffman and Zelditch in \cite{sz03},
and the proof follows the same general lines, adapted to the 
Grauert tube context
(for antecedents in the real domain, see \cite{so88}
and \cite{so}).  
More precisely, the argument given
by Chang and Rabinowitz
for (\ref{eqn:cr norm estimate})
is based on the off-graph scaling asymptotics of
$\Pi^\tau_{\chi,\lambda}$, paired with the Shur-Young inequality
\cite{so}.

The goal of this section is to describe
an equivariant extension of (\ref{eqn:cr norm estimate}), 
involving the operator norm of
$\Pi^\tau_{\chi,\nu,\lambda}:L^p(X^\tau)\rightarrow 
L^q(X^\tau)$.

Theorem \ref{thm:near-graph-unrescaled}
allows an adaptation to the equivariant setting of the arguments in \cite{cr2} for the 
proof of (\ref{eqn:cr norm estimate}).

\begin{thm}
\label{thm:norm operator estimate}
Under the previous assumptions, there exists a constant
$\rho^\tau_{p,\nu}>0$ such that
$$\left\|\Pi^\tau_{\chi,\nu,\lambda}\right\|_{L^p\rightarrow L^q}
\le \rho_{p,\nu}^\tau
\left( \frac{\lambda}{\tau}  \right)
^{\left(\frac{1}{p}-\frac{1}{q}\right)\,\left(d-1-\frac{d_G}{2} \right)}.
$$
\end{thm}

\subsubsection{Estimates on the complexifications
$\widetilde{\varphi}^\tau_{j,\nu,k}$'s.}

We discuss some estimates on equivariant complexified
eigenfunctions of $\Delta$, which are 
equivariant analogues of estimates in
\cite{z20}. The proofs are either straightforward, or
minor modifications of the ones for the eigenfunctions
of $\mathfrak{D}^\tau_{\sqrt{\rho}}$.

Let $\chi$ be as in \S \ref{sctn:pointwise est};
Then $P^\tau_{\chi,\nu,\lambda}(x,x)\ge 0$.
Furthermore, as we have remarked an analogue of Theorem 
\ref{thm:near-graph-unrescaled} holds for $P^\tau_{\chi,\nu,\lambda}$,
leading to the following analogue of Proposition \ref{prop:uniform diagonal estimate}:

\begin{prop}
\label{prop}
For some constant $C_\nu>0$ and every
$x\in X^\tau$ and $\lambda\gg 0$ we have
$$
P^\tau_{\chi,\nu,\lambda}\left(
x,x\right)\le C_\nu\,
\left( \frac{\lambda}{\tau}\right)^{(d-1-d_G)/2}.
$$

\end{prop}


With $\lambda=\mu_j$ we deduce the following.

\begin{cor}
If $x\in X^\tau$ and $j\gg 0$,
\begin{equation}
\label{eqn:sup abs val x}
\sum_k 
\left|\widetilde{\varphi}^\tau_{j,\nu,k}(x)\right|^2
\le C_\nu\,e^{2\,\tau\,\mu_j}\,
\left( \frac{\mu_j}{\tau}\right)^{(d-1-d_G)/2}.
\end{equation}
\end{cor}

As in \cite{z20}, \cite{cr1}, \cite{cr2}, let us consider 
the (equivariant) Husimi distributions
\begin{equation}
\label{eqn:husimi equiv}
U_{j,\nu,k}(x):=
\dfrac{\left|\widetilde{\varphi}^\tau_{j,\nu,k}(x)\right|^2}{
\left\|\widetilde{\varphi}^\tau_{j,\nu,k}\right\|_{L^2(X^\tau)}^
2}.
\end{equation}
By Lemma 0.2 of \cite{z20}, 
$$
\left\|\widetilde{\varphi}^\tau_{j,\nu,k}\right\|_{L^2(X^\tau)}^
2\sim D_\tau\,e^{2\,\tau\,\mu_j}\,\mu_j^{-(d-1)/2}
$$
for some constant $D_\tau>0$.
We obtain the following equivariant
refinement of the action-free upper bound in \cite{z20}.

\begin{cor}\label{cor:sup husimi}
Under the previous assumptions, for $j\gg 0$ and every $k$
\begin{equation*}
\sup _{x\in X^\tau}U_{j,\nu,k}(x)
\le D_\tau''\,\mu_j^{d-1-\frac{1}{2}\,d_G}.
\end{equation*}
\end{cor}



We conclude with an $L^2$-analogue of the Weyl law for
$\mathfrak{D}^\tau_{\sqrt{\rho}}$ in Theorem \ref{thm:weyl law geod}.

\begin{prop}
\label{prop:weyl law integrand split 1 int glob 1}
Under the assumptions and with the notation
of Theorem \ref{thm:weyl law geod},
the following asymptotic expansion holds for 
$\lambda\rightarrow +\infty$:
\begin{eqnarray*}
\mathrm{trace}\left(P^\tau_{\chi,\nu,\lambda}\right)&=&
\sum_j\hat{\chi}(\lambda-\mu_j)\,
e^{-2\,\tau\,\mu_j}\,\sum_k \left\|\widetilde{\varphi}^\tau_{j,\nu,k}\right\|_{L^2(X^\tau)}^2\\
&\sim&\frac{1}{\sqrt{2^{d+d_G}\,\pi}}\,
\left( \frac{\lambda}{2\,\pi\,\tau}\right)^{\frac{d-1}{2}-d_G}\,
\cdot \dim(\nu)^2\cdot
\mathrm{vol}(Z^\tau/G)\nonumber\\
&&
\cdot \left[\chi (0)+\sum_{k\ge 1}\,\lambda^{-k}\,F_{k,\nu}
\left(x\right)\right]
\nonumber
\end{eqnarray*}
\end{prop}

Setting $\lambda=\mu_j$, we obtain the following consequence.
\begin{cor}
\label{cor:L2 norm compl est 1}
There is a constant $C_\nu^\tau>0$ 
such that for all $j\gg 0$
$$
\sum_k 
\left\|
\widetilde{\varphi}^\tau_{j,\nu,k}\right\|_{L^2(X^\tau)}^2
\le C_\nu^\tau\,e^{2\,\tau\,\mu_j}\,\mu_j^{\frac{d-1}{2}-d_G}.
$$

\end{cor}

The asymptotic expansion in Proposition  \ref{prop:weyl law integrand split 1 int glob 1} may be translsated into 
information about the asymptotics of
\begin{equation}
\label{eqn:defn Ptaunuweyl}
\mathcal{P}^\tau_\nu(\lambda)=
\sum_{\mu_j\le \lambda}
e^{-2\,\tau\,\mu_j}\,\sum_k \left\|\widetilde{\varphi}^\tau_{j,\nu,k}\right\|_{L^2(X^\tau)}^2,
\end{equation}
essentially by the same
argument  used for Theorem \ref{thm:weyl law geod}.
\begin{thm}
\label{thm:weyl poisson}
In the situation of Theorem \ref{thm:weyl law geod}, let us also assume
that $d\ge 2\,d_G$.
Then 
\begin{eqnarray*}
\mathcal{P}_\nu^\tau(\lambda)&\sim&
\frac{1}{\sqrt{2^{d+1+d_G}}\,\pi}\,\left( \frac{\lambda}{2\,\pi\,\tau}\right)^{\frac{d-1}{2}-d_G}\,
\cdot \dim(\nu)^2\cdot
\mathrm{vol}(Z^\tau/G)\nonumber\\
&&
\cdot \left[\dfrac{\lambda}{\frac{d+1}{2}-d_G}+O(1)
\right].\nonumber
\end{eqnarray*}
\end{thm}

\section{Preliminaries}

We shall identify the tangent and cotangent bundles of $M$, 
$TM$ and $T^\vee M$, by means of $\kappa$.

\subsection{The action on the Grauert tubes}
\label{sctn: grauert and coshpere}

In Section 
\ref{sctn:intro}, the Grauert tubes 
$\tilde{M}^\tau$, $\tau\in (0,\tau_0)$,
have been described 
as strictly pseudoconvex domains in $\tilde{M}$ intrisically associated
to $\kappa$. 
An alternative characterization of $\tilde{M}^\tau$
is in terms of disk bundles
$T^\tau M\subset T^\vee M$ of norm $\tau$ with respect to
$\kappa$, where $T^\vee M$ has been endowed 
in a neighbourhood of the zero section $M_0$ with a
complex structure intrinsically associated to $\kappa$
(\cite{ls}, \cite{l}, \cite{szo}, \cite{gs1}, \cite{gs2}).
This complex structure was called \textit{adapted} in \cite{ls},
and will be denoted $J_{ad}$; it 
is uniquely determined by the condition that
the Riemannian and Monge-Amp\`{e}re foliations coincide
\cite{ls}.

\begin{notn}\label{defn:defn di psigamma}
Let $\gamma:\mathbb{R}\rightarrow M$ be a geodesic, with 
velocity curve $\dot{\gamma}:\mathbb{R}\rightarrow TM$.
Let us define $\psi_\gamma:\mathbb{C}\rightarrow TM$
by 
$$
\psi_\gamma:a+\imath\,b\mapsto N_b\big(\dot{\gamma}(a)\big),
$$
where $N_b$ denotes dilation by the factor $b$.
\end{notn}

\begin{thm}
\label{thm:Jad char}
(Lempert and Sz\"{o}ke) 
Let $J_0$ be the complex structure on $\mathbb{C}$. 
Given any sufficiently small neighbourhood $T'M\subseteq TM$
of the zero section, 
$J_{ad}$ is the unique complex structure on 
$T'M\subseteq TM$ such that 
$\psi_\gamma:\psi_\gamma^{-1}(T'M)\rightarrow T'M$ is 
$(J_0,J_{ad})$-holomorphic for every geodesic $\gamma$ on $(M,\kappa)$.

\end{thm}

As proved in \cite{ls} and \cite{gls}, 
for some tubular neighbourhood $T'M\subseteq T^\vee M$ 
of the zero section
the 
\textit{imaginary time exponential map} 
$E:T'M\rightarrow \tilde{M}$ 
(essentially defined by 
complexifying the ordinary exponential map)
is $(J_{ad},J)$-holomorphic and
intertwines the square norm function with $\rho$ and
the canonical symplectic structure $\Omega_{can}$ on $T^\vee M$
with $\Omega$ (we conform to the notation and conventions in
\cite{p24}, thus $\Omega_{can}=\mathrm{d}\mathbf{q}\wedge\mathrm{d}\mathbf{p}$ in local $(\mathbf{q},\mathbf{p})$ coordinates).
Hence $E$ restricts to an isomorphism of K\"{a}hler manifolds
$$
E^\tau:(T^\tau M,\Omega_{can},J)\cong (\tilde{M}^\tau,\Omega,J).
$$

\begin{notn}
Given $f:M\rightarrow M$ smooth,
$\mathrm{d} f:TM\rightarrow TM$ will denote its 
differential (tangent map). 
\end{notn}

\begin{prop}\label{prop: from isometries to niholomorphisms}
 If $f: M\rightarrow M$ is a Riemannian isometry of $(M,\kappa)$, then 
 $\mathrm{d} f$ restricts to a biholomorphism of $(T^{\tau}M, J_{\mathrm{ad}})$ into itself, for any $\tau\in (0,\tau_0)$.
\end{prop}

\begin{proof}
    The claim is that 
    \begin{equation}
    \label{eqn:df olom}
    \mathrm{d}(\mathrm{d} f) \circ J_{\mathrm{ad}} = J_{\mathrm{ad}}\circ \mathrm{d}(\mathrm{d} f)
    \end{equation}
on $T(T^{\tau}M)$. Let $(m,v)\in TM\setminus M_0$, and let $\gamma:\mathbb{R}\longrightarrow M$ be the unique geodesic
with $\dot{\gamma}(0)=(m,v)$. 
Since $f$ is an isometry, the geodesic corresponding to 
$\mathrm{d} f (m,v)$ is $f\circ \gamma$, and its velocity is 
$\mathrm{d} f(\dot{\gamma})$. Hence, 
$$\mathrm{d} f \circ \psi_{\gamma}(a+\imath\,b) =
\mathrm{d} f \circ N_b\left(\dot{\gamma}(a)\right)
=N_b\circ \mathrm{d} f \left(\dot{\gamma}(a)\right)=
N_b\circ \dot{\overbrace{{f\circ \gamma}}}(a)=
\psi_{f\circ \gamma}(a+\imath\,b).
$$
Thus, by Theorem \ref{thm:Jad char} 
$\mathrm{d} f \circ \psi_{\gamma}=\psi_{f\circ \gamma}$
is $(J_0,J_{\mathrm{ad}})$-holomorphic for every $\gamma$ 
(on the preimage in $\mathbb{C}$ of $T^{\tau}M$).
In other worlds,
    \begin{equation*}
        J_{\mathrm{ad}}\circ \mathrm{d}(\mathrm{d} f \circ \psi_{\gamma}) = \mathrm{d}(\mathrm{d} f\circ \psi_{\gamma}) \circ J_0.
    \end{equation*}
    By the chain rule, 
    \begin{equation*}
        \left(\mathrm{d}(\mathrm{d} f)^{-1}\circ J_{\mathrm{ad}} \circ \mathrm{d}(\mathrm{d} f)\right) \circ \mathrm{d}\psi_{\gamma} = \mathrm{d}\psi_{\gamma}\circ J_0.
    \end{equation*}
    By the uniqueness part in Theorem \ref{thm:Jad char}, 
    \begin{equation*}
        \mathrm{d}(\mathrm{d} f)^{-1}\circ J_{\mathrm{ad}} \circ \mathrm{d}(\mathrm{d} f) = J_0,
    \end{equation*}
i.e. (\ref{eqn:df olom}) holds.
\end{proof}

\begin{cor}
    Any isometry of $(M,\kappa)$ is real-analytic.
\end{cor}

Therefore, if $f$ is an isometry of $(M,\kappa)$, it uniquely
extends
to a holomorphic map $\tilde{f}:\tilde{M}^\tau\rightarrow \tilde{M}$
for any sufficiently small $\tau>0$.

\begin{cor}
For any $\tau\in (0,\tau_0)$ the following holds:
\begin{enumerate}
\item $\tilde{f}$ is defined on $\tilde{M}^\tau$, 
where it is the biholomorphism
\begin{equation*}
    \tilde{f} = E^{\tau}\circ \mathrm{d} f\circ (E^{\tau})^{-1}: \tilde{M}^{\tau}\rightarrow\tilde{M}^{\tau};\end{equation*}
    \item $\tilde{f}\circ \rho=\rho$, $\tilde{f}^*(\Omega)=
    \Omega$;
\item $\tilde{f}$ restricts to a CR and contact automorphism 
$\tilde{f}^\tau: X^{\tau}\rightarrow X^\tau$.

\end{enumerate}

\end{cor}

Applying the above to a smooth action by isometries, we conclude the following.

\begin{prop}\label{prop: basic facts on G action}
    Let $(M,\kappa)$ be a compact $\mathcal{C}^\varpi$ Riemannian manifold, $G$ a compact Lie group, and 
    $\mu:G\times M\rightarrow M$ a smooth action as a group of isometries of $(M,\kappa)$. Then:
    \begin{enumerate}
        \item for every $g\in G$, $\mu_g:M \rightarrow M$ is real-analytic;
        \item for $\tau\in (0,\tau_{0})$, $\mu$ extends to a $G$-action $\tilde{\mu}^\tau:G\times \tilde{M}^{\tau}\rightarrow \tilde{M}^{\tau}$ as a group of automorphisms of the K\"ahler manifold 
        $(\tilde{M}^{\tau}, \Omega, J)$;
\item        $E^\tau$ intertwines $\tilde{\mu}^\tau$ on $\tilde{M}^\tau$ with
        the (co)tangent action on $T^\tau M$;
        \item $(\tilde{\mu}^\tau_g)^*(\alpha) = \alpha,\,\forall\,g\in G$;
        \item $\tilde{\mu}$ restricts to an action $\mu^{\tau}:G\times  X^{\tau}\rightarrow X^{\tau}$ as a group of automorphisms of the contact CR manifold $ X^{\tau}$.
    \end{enumerate}
\end{prop}
\begin{proof}
    The only thing left to prove is that $\tilde{\mu}_g^*(\alpha^{\tau}) = \alpha^{\tau}$ for all $g\in G$. 
    Since $\tilde{\mu}_g$ is holomorphic, preserves $\rho$ and $\alpha = \Im(\partial\rho)$, the thesis follows by restriction to $ X^{\tau}$. 
\end{proof}

The (co)tangent lift of $\mu$ preserves 
the tautological 1-form 
$\lambda_{\mathrm{\mathrm{can}}}$ on $T^\vee M$
(locally given by $\mathbf{p}\,\mathrm{d}\mathbf{q}$), 
and is therefore Hamiltonian with respect to $\Omega_{\mathrm{\mathrm{can}}}$, with moment map 
\begin{equation*}
    \Phi^{\xi}_{TM} = \langle\Phi_{TM},\xi\rangle = \lambda_{\mathrm{\mathrm{can}}}(\xi_{TM}).
\end{equation*}
On the other hand, for every $\tau\in (0,\tau_0)$ 
$$
\alpha=-{E^\tau}^*(\lambda_{can})\quad\text{on}
\quad M^\tau
$$
(see the discussion in the introduction of \cite{p24}).

Given $\xi\in \mathfrak{g}$ and $\tau\in (0,\tau_0)$, 
we shall denote by 
$\xi_{\tilde{M}^{\tau_0}}$ and $\xi_{X^\tau}$ the induced vector
fields on $\tilde{M}^{\tau_0}$ and $X^\tau$, respectively.
Hence $\xi_{\tilde{M}^{\tau_0}}$
is tangent to $X^\tau$ and $\xi_{X^\tau}$ is the restriction of
$\xi_{\tilde{M}^{\tau_0}}$ to $X^\tau$.

By the previous discussion (see also Lemma 22 of \cite{p24}),
we conclude the following.

\begin{cor}\label{cor: the moment map on tubes}
With the preceding assumptions and notation, 
for every $\tau\in (0,\tau_0)$
the following holds. 

\begin{enumerate}
\item $\tilde{\mu}^\tau$ is Hamiltonian for $\Omega$, 
with moment map $\Phi_{\Tilde{M}^\tau}:\tilde{M}^\tau\rightarrow \mathfrak{g}^\vee$ 
given by
$$\varphi^{\xi} := \langle\Phi_{\Tilde{M}^\tau},\xi\rangle= -\alpha(\xi_{\tilde{M}^\tau})\quad 
(\xi\in \mathfrak{g}),
$$
\item For any $\xi\in \mathfrak{g}$, we have
$\xi_{X^\tau}=\xi_{X^\tau}^\sharp-\varphi^\xi\,\mathcal{R}^\tau$,
where $\xi_{X^\tau}^\sharp$ is a smooth section of $\mathcal{H}^\tau$
and we write $\varphi^\xi$ for $\left.\varphi^\xi\right|_{X^\tau}$.
\end{enumerate}

\end{cor}

\begin{notn}
\label{notn:moment map}
We stress the notation used above: $\tilde{\mu}^\tau$ is the 
action of $G$ on $\tilde{M}^\tau$ as a group of K\"{a}hler automorphisms;
$\mu^\tau$ is the restricted action of $G$ 
on $X^\tau$ by contact CR automorphisms.
Furthermore, in the following we shall generally simplify notation
and write $\Phi$ for $\Phi_{\Tilde{M}^\tau}$.
\end{notn}

Since $\rho$ is $G$-invariant, 
$[\xi_{\tilde{M}^{\tau_0}},\upsilon_{\sqrt{\rho}}]=0$, hence
$\Phi$ is constant along the geodesic flow. 
Recalling (\ref{eqn:reeb vct field}) and (\ref{eqn:defn di Ztau}),
we conclude the following.

\begin{cor}
\label{cor:reeb tangent Z}
$\mathcal{R}^\tau$ is tangent to $Z^\tau$.

\end{cor}

\subsection{$\Pi^\tau$ in NHLC's on $X^\tau$}
\label{sctn:szego}


As mentioned in \S \ref{sctn:intro}, 
our computations will be carried out in so-called \textit{normal Heisenberg
local coordinates} on $X^\tau$
(recall the discussion surrounding (\ref{eqn:Reeb theta v H}) and 
(\ref{eqn:rescaled coord 12})); these are induced 
(by projection and restriction)
by suitable holomorphic local coordinates on $\tilde{M}$, in which
the defining equation of $X^\tau$ has a canonical approximate form,
and which are also called NHLC's (on $\tilde{M}$). 

NHLC's were defined in \S 3.3 of \cite{p24}, and 
are a slight specialization of the normal coordinates introduced in
\cite{fs1} and \cite{fs2}, and first put to use
(under the name of Heisenberg local coordinates) in the present setting in
\cite{cr1} and \cite{cr2}. In the line bundle context, Heisenberg local
coordinates have been thoroughly 
used in the study of scaling asymptotics since their appearance in \cite{sz}.

In NHLC's on $\tilde{M}$ centered at $x\in X^\tau$, $\alpha$, $\Omega$ and $\hat{\kappa}$ (and hence also $\omega =\frac{1}{2}\Omega$ and $\tilde{\kappa} = \frac{1}{2}\hat{\kappa}$) have a simple local approximate expression, for which we refer to Proposition 34
of \cite{p24}. Another key point in our discussion is that the phase and the
leading order term of the symbol of the Szeg\H{o} kernel $\Pi^\tau$
can be computed fairly explicitly in NHLC's on $X^\tau$.

As proved in \cite{bs}, $\Pi^\tau$ in
(\ref{eqn:szego proj and ker}) 
is a Fourier integral operator with complex phase.
Its wave front is the anti-diagonal $(\Sigma^{\tau})^{\sharp}$ of
$\Sigma^\tau$ in (\ref{eqn:Sigmataucone}):
\begin{equation*}
    (\Sigma^{\tau})^{\sharp} := \left\{\left(x, r\alpha^\tau_x, x, -r\alpha^\tau_x\right)\ \big|\ x\in X^{\tau}, r>0 \right\}\subseteq T^{\lor}X^{\tau}\setminus(0) \times T^{\lor}X^{\tau}\setminus(0).
\end{equation*}
Up to a smoothing kernel, the distributional kernel $\Pi^\tau(\cdot,\cdot)$ 
has the form: 
\begin{equation}
\label{eqn:szego as FIO}
    \Pi^{\tau}(x,y) \sim \int_{0}^{+\infty} e^{\imath u\psi^{\tau}(x,y)}s^{\tau}(x,y,u) \mathrm{d}u, 
\end{equation}
where
\begin{enumerate}
    \item $s^{\tau}$ is a classical symbol of the form: 
    \begin{equation*}
        s^{\tau}(x,y,u)\sim \sum_{j\geq 0} u^{d-1-j} s_j^{\tau}(x,y);
    \end{equation*}
    \item $\psi^{\tau}$ is a phase of positive type and is determined by:
    \begin{equation*}
        \psi^{\tau} := -\imath \tilde{\phi}^{\tau}\big|_{X^{\tau}\times X^{\tau}} 
    \end{equation*}
    where $\tilde{\phi}$ is the holomorphic extension to $\tilde{M}\times \overline{\tilde{M}}$  of  the defining function $\phi^{\tau} = \rho -\tau^2$  of $X^\tau$ in $\tilde{M}$
(see the discussions in \cite{cr1}, \cite{cr2}, and
 \S3.3.2 of \cite{p24} - and of course \cite{bs}). 
\end{enumerate}

For the following, see Proposition 48 and Theorem 51 of \cite{p24}.

\begin{thm}\label{thm: phase expression in HLC}
    Suppose $x\in X^{\tau}$ and fix a NHLC's on $X^\tau$
    centered at $x$ 
    and defined on an open neighourhood $U^\tau\subseteq X^\tau$ 
    of $x$. Then, if $x',x''\in U^{\tau}$ are of the form $x' = x + (\theta, z')$, $x'' = x+(\eta, z'')$, we have: 
    \begin{equation*}
        \imath \psi^{\tau}(x',x'') = \imath (\theta-\eta) -\frac{1}{4\tau^{2}}(\theta-\eta)^2 + \psi^{\omega_x}_2(z',z'') + R_3(\theta, z',\overline{z'},\eta, z'',\overline{z''}).
    \end{equation*}
    Moreover, in the same chart, the principal term of the symbol satisfies 
   $$s^{\tau}_0(x,x) = \frac{\tau}{(2\pi)^{d}}.$$
\end{thm}

\subsection{The submanifold $Z^{\tau}$}
\label{sctn:Z tau}

Our analysis rests on the assumption that $G$ acts locally freely on
$Z^\tau$ in (\ref{eqn:defn di Ztau}). Thus
Assumption \ref{ass:ba} is certainly
satisfied if $\mu$ itself is locally free and $Z^\tau\neq \emptyset$, for instance if 
$M$ is the total space of a principal $G$-bundle
over a non-trivial base, $p:M\rightarrow N$ with $\dim(N)>0$.
In this case, under the map $E$ in \S \ref{sctn: grauert and coshpere},
$Z$ is the image in $\tilde{M}$ of the orthocomplement of the vertical tangent bundle
$\mathrm{Ver}(p)\subseteq TM$
(that is, of the horizontal tangent bundle $\mathrm{Hor}(p)\subseteq TM$) of $M$ as a 
principal $G$-bundle (more precisely, 
$Z=E\big(\mathrm{Hor}(p)\big)\cap T'M$).
Then the quotient $Z/G$ can be identified 
with a tubular neighbourhood of $N$ in $TN$,
and $Z^\tau/G$ with the bundle of tangent spheres of
radius $\tau$ over $N$.

The case of a principal $G$-bundle is clearly not the only circumstance where the
hypothesis is statisfied; 
for example, if $G$ acts in a Hamiltonian manner on
a compact real-analytic symplectic manifold 
$(R,\omega_R)$, and $0\in \mathfrak{g}^\vee$
is a regular value of the moment map
$\Phi_R:R\rightarrow \mathfrak{g}^\vee$, then 
$G$ acts locally freely on $M:=\Phi_R^{-1}(0)$,
but in most cases the action is not free.

Furthermore, it can happen that Assumption \ref{ass:ba}
is satisfied even if $\mu$ is not locally free.
Here we give two more sufficient conditions,
whose proof is left to the reader.

\begin{lem}
With the previous assumptions and notation, 
$\mu^\tau$ is locally free on $Z^\tau$ under the following
two sets of circumstances. 

\begin{enumerate}
\item $\mu$ is locally free away from a finite set of 
fixed  points.
\item $G$ is Abelian, and if $T_1,\ldots,T_r\leqslant G$
are the subtori that appear as stabilizers of poits in $M$
then
the connected components of the fixed locus of each $T_i$ in $M$
are single $G$-orbits of $\mu$.
\end{enumerate}

\end{lem}

We record a statement in the opposite direction,
whose proof is also left to the reader.

\begin{lem}
    Suppose that there exists a submanifold $N\subseteq M$ of positive dimension such that 
    $\mu_g(n)=n$ for every $g\in G$ and $n\in N$. 
    Then $\mu^\tau$ is not locally free 
    on $Z^\tau$.
\end{lem}

If $G$ acts locally freely on $Z\setminus M$, the
latter is a submanifold of $\tilde{M}$ of codimension $d_G$
(Remark \ref{rem:rem on assmpt}).
Given that $Z$ is obviously transverse to $X^\tau$, 
we conclude the following.

\begin{lem}
Under Assumption \ref{ass:ba},
$ Z^{\tau}$ is a submanifold of $X^{\tau}$ of dimension $2d-1-d_G$.
\end{lem}

Since $G$ acts locally freely on $ Z^{\tau}$,
$d_G\leq\mathrm{dim}(Z^{\tau}) $.
\begin{cor}
\label{cor:d>d_G}
    $d_G\leq d-1$.
\end{cor}

\subsection{Intersections of $G$-orbits and $\mathbb{R}$-orbits}

The orbits of the the geodesic flow
$t\mapsto\Gamma^\tau_t$ on $X^\tau$ are \lq vertical\rq\, 
(i.e., tangent to
$\mathcal{T}^\tau$ in (\ref{eqn:dicrect sum vert hor})); 
in view of (\ref{cor: the moment map on tubes}),
along $Z^\tau$ the orbits of the $G$-action $g\mapsto
\mu^\tau_g$ are \lq horizontal\rq\, (i.e., tangent to $\mathcal{H}^\tau$).
Since under Assumption \ref{ass:ba}
both actions are locally free on $Z^\tau$ and commute, 
the product action of 
$G\times \mathbb{R}$ is also locally free on $Z^\tau$.

Points $x_1,\,x_2\in Z^\tau$ belong to the same $G\times \mathbb{R}$
orbit, that is, $x_1\in x_2^{G\times \mathbb{R}}$, if and only if the 
$\mathbb{R}$-orbit of either one intersects the $G$-orbit of the other,
i.e.
$x_1^G\cap x_2^{\mathbb{R}}\neq \emptyset$, and then the intersection might be infinite. However, given a
compact subset $K\subset\mathbb{R}$ by the above there are
at most finitely many $t\in K$ for which $\Gamma^\tau_t(x_2)\in 
x_1^G$. A uniform statement can be given in terms 
of the size of $K$, but first we need to lay down some useful 
consequences of the local freeness of the action of $G\times \mathbb{R}$
on the compact manifold $Z^\tau$; the proof is left to the reader.

\begin{lem}
\label{lem:technical facts}
Under Assumption \ref{ass:ba}, the following holds.
\begin{enumerate}
\item There exists $R,\,r>0$ such that if
$(g,t)$ belongs to an $R$-neighbourhood of $(e_G,0)$ then
$$
\mathrm{dist}_{X^\tau}\left(y,\mu^\tau_g\circ \Gamma^\tau_t(y)\right)
\ge r\,\mathrm{dist}_{G\times \mathbb{R}}\big((g,t),(e_G,0)\big).
$$
\item There exist $C>0$ such that
$$
\mathrm{dist}_{X^\tau}\left(x,\Gamma^\tau_{t}(x)   \right)
< C\,|t|,\quad \forall\,(x,t)\in Z^\tau\times \mathbb{R}.
$$
\item There exists $D>0$ such that
for any $y\in Z^\tau$ and $\delta>0$ sufficiently small
$$
\mathrm{dist}_{G}(g,G_y)\ge \delta\quad \Rightarrow\quad
\mathrm{dist}_{X^\tau}\left(y,\mu^\tau_g(y)\right)
\ge \frac{\delta}{D}.
$$
\end{enumerate}

\end{lem}

\begin{lem}
\label{lem:unique t}
Assume that
$\epsilon>0$ is sufficiently small and
$\chi\in \mathrm{C}^\infty_c\big((t_0-\epsilon,t_0+\epsilon)\big)$
for some $t_0\in \mathbb{R}$.
Then for any $x\in Z^\tau$ and $x'\in x^{G\times \chi}$ there exists a unique
$t\in \mathrm{supp}(\chi)$ such that
$x'=\mu^\tau_g\circ \Gamma^\tau_t(x)$
for some $g\in G$.

\end{lem}

\begin{proof}[Proof of Lemma \ref{lem:unique t}]
If the Lemma is false, for any $j=1,2,\ldots$ there exist
$x_j,\,x_j'\in Z^\tau$,
$t_j,\,t_j'\in \mathbb{R}$, and $g_j,\,g_j'\in G$ such that
%
such that
$$
0<|t_j-t_j'|<\frac{1}{j}\quad 
\text{and}\quad
x_j'=\mu^\tau_{g_j'}\circ \Gamma^\tau_{t_j'}(x_j)=
\mu^\tau_{g_j}\circ \Gamma^\tau_{t_j}(x_j).
$$
Hence
\begin{equation}
\label{eqn:xjstabilized}
x_j=\mu^\tau_{g_j^{-1}\,g_j'}\circ \Gamma^\tau_{t_j'-t_j}(x_j).
\end{equation}

By statement 2 of Lemma \ref{lem:technical facts},
$$
\mathrm{dist}_{X^\tau}\left(x,\Gamma^\tau_{t_j'-t_j}(x)   \right)
< C/j\quad\forall\,x\in X^\tau.
$$
By the triangle inequality, we conclude that
\begin{eqnarray}
\lefteqn{  \mathrm{dist}_{X^\tau}\left(x_j,\mu^\tau_{g_j^{-1}\,g_j'}(x_j)   \right) }\\
&\le&
\mathrm{dist}_{X^\tau}\left(x_j,\mu^\tau_{g_j^{-1}\,g_j'}\circ \Gamma^\tau_{t_j'-t_j}(x_j)\right)+
\mathrm{dist}_{X^\tau}\left(\mu^\tau_{g_j^{-1}\,g_j'}\circ \Gamma^\tau_{t_j'-t_j}(x_j),
\mu^\tau_{g_j^{-1}\,g_j'}(x_j)
\right)\nonumber\\
&\le&0+C/j=C/j.\nonumber
\end{eqnarray}
Hence, by statement 3 of Lemma \ref{lem:technical facts}, 
$$
\mathrm{dist}_G\left(g_j^{-1}\,g_j',G_{x_j}   \right)
\le (C\cdot D)/j.
$$
We may thus find $\sigma_j\in G_{x_j}$ and $\delta_j\in G$ such that
$$
g_j^{-1}\,g_j'=\delta_j\,\sigma_j,\quad
\mathrm{dist}_G\left(\delta_j,e_G\right)\le D/j.
$$
Then (\ref{eqn:xjstabilized}) implies

\begin{equation}
\label{eqn:xjstabilized1}
x_j=\mu^\tau_{\delta_j\,\sigma_j}\circ \Gamma^\tau_{t_j'-t_j}(x_j)
=\mu^\tau_{\delta_j}\circ \Gamma^\tau_{t_j'-t_j}(x_j),
\end{equation}
and $(\delta_j,t_j'-t_j)$ belongs to a neighbourhood of $(0,e_G)$
of radius $O\left(1/j\right)$.
By statement 1 of Lemma \ref{lem:technical facts}
this is absurd, unless $\delta_j=e_G$ and $t_j=t_j'$.

\end{proof}

\begin{cor}
\label{cor:unique t1chi}
If $\epsilon>0$ is sufficiently small and
$|\mathrm{supp}(\chi)|<2\,\epsilon$, 
then for any  $x_1,\,x_2\in Z^\tau$ we have the following alternative;
either $x_1^G\cap x_2^{\chi}=\emptyset$, or else  
$x_1^G\cap x_2^{\chi}=\{x_{12}\}$ for a unique $x_{12}\in Z^\tau$.

\end{cor}

\begin{rem}
\label{rem:unique t1chi}
In the assumptions of Corollary \ref{cor:unique t1chi}, 
suppose that $x_1^G\cap x_2^{\mathrm{supp}(\chi)}\neq \emptyset$;
equivalently, 
\begin{equation}
\label{eqn:Sigmax12}
\Sigma_\chi(x_1,x_2):=\left\{(g,t)\in G\times \mathrm{supp}(\chi)\,:\,x_1=\mu^\tau_g\circ \Gamma^\tau_{t}(x_2)\right\}\neq \emptyset.
\end{equation}
Let $t_1=t_1(x_1,x_2)\in \mathrm{supp}(\chi)$ be the unique element
whose existence is asserted in Lemma \ref{lem:unique t},
and choose 
$h_1\in G$ such that $(h_1,t_1)\in  \Sigma_\chi(x_1,x_2)$.
Let $G_{x_1}:=\left\{\kappa_l\,:\,l=1,\ldots,r_{x_1}\right\}\leqslant G$ denote the stabilizer subgroup of $x_1$. Then
\begin{equation}
\label{eqn:Sigmax12chi}
\Sigma_\chi(x_1,x_2)=\left\{
(\kappa_l\,h_1,t_1)\,:\,
l=1,\ldots,r_{x_1}
\right\}.
\end{equation}
\end{rem}

\subsection{The action in Heisenberg local coordinates}

As is the Introduction (recall the discussion surrounding
(\ref{eqn:rescaled coord 12})), NHLC's on $X^\tau$ centered
at $x$ will be denoted
in additive notation:
$y(\theta,\mathbf{v})=x+(\theta,\mathbf{v})$.
Given $x\in X^\tau$, we can find an open neighbourhood
$U^\tau\subseteq X^\tau$ of $x$ and 
a smoothly varying family of systems of NHLC's 
$y_{x'}(\theta,\mathbf{v})=x'+(\theta,\mathbf{v})$
centered at $x'\in U^\tau$. 
More explicitly, for some $\delta>0$ 
we have a smooth map
$\Upsilon:U^\tau\times (-\delta,\delta)\times B_{2d-2}(\mathbf{0},\delta)
\rightarrow X^\tau$
such that for every $x'\in U^\tau$ 
the partial function $\Upsilon(x',\cdot,\cdot)$ is a system of NHLC's
centered at $x'$. Hence
$$
y_{x'}(\theta,\mathbf{v})=\Upsilon (x',\theta,\mathbf{u})= x'+(\theta,\mathbf{u}).
$$
We refer to the notation and conventions in \cite{p24}
(see also \cite{cr1} and \cite{cr2}). 
In this section, we shall adapt some arguments
from \cite{pao-IJM} to express the action
of $G\times \mathbb{R}$ on $X^\tau$ in terms of
NHLC's.

Let us first consider the geodesic flow.
The following is a consequence of (105) of \cite{p24}.

\begin{lem}
\label{eqn:local expression geodesic flow}
Suppose $x\in X^\tau$, and let us choose Heisenberg local coordinates on $X^\tau$ centered at $x$. If $y=x+(\theta,\mathbf{u})$,
then 
$$
\Gamma^\tau_t(y)=
x+\Big(\theta-\tau\,t+R_3(\tau,t,\theta,\mathbf{u}),
\mathbf{u}+\mathbf{R}_2(\tau\,t,\theta,\mathbf{u} ) \Big).
$$
\end{lem}

The convention for the meaning of $R_k$ and $\mathbf{R}_k$ 
is as explained in the discussion of (\ref{eqn:defn of M}).

To proceed, we need to relate systems of NHLC's centered at nearby
points. We shall adapt an argument in \cite{pao-IJM} to the present setting.

\begin{lem}
\label{lem:thetajvj}
For $(\theta_,\mathbf{u}_j)\sim (0,\mathbf{0})$,
\begin{eqnarray*}
x+(\theta_2,\mathbf{u}_2)
&=&\big( x+(\theta_1,\mathbf{u}_1)  \big)\\
&&+
\Big(\theta_2-\theta_1+\omega(\mathbf{u}_1,\mathbf{u}_2)+
R_3(\theta_j,\mathbf{u}_j),
\mathbf{u}_2-\mathbf{u}_1+\mathbf{R}_2(\theta_j,\mathbf{u}_j)
\Big).
\end{eqnarray*}

\end{lem}

\begin{proof}
We can write
$$
x+(\theta_2,\mathbf{u}_2)=
\big( x+(\theta_1,\mathbf{u}_1)  \big)+
\big(\beta(\theta_j,\mathbf{u}_j),\mathbf{B}(\theta_j,
\mathbf{u}_j)\big)
$$
for certain smooth functions $\beta$ (real-valued) and
$\mathbf{B}$ (vector-valued), vanishing at the origin.
Let us expand $\beta$ and $\mathbf{B}$:
$$
\beta=\beta_1+\beta_2+R_3,\quad
\mathbf{B}=\mathbf{B}_1+\mathbf{R}_2,
$$
where $\beta_1$ and $\mathbf{B}_1$ are linear, and $\beta_2$ is homogeneous of degree $2$.
It is easily seen that 
\begin{equation}
\label{eqn:1st order terms}
\beta_1(\theta_j,\mathbf{u}_j)=\theta_2-\theta_1,\quad
\mathbf{B}_1(\theta_j,
\mathbf{u}_j)=\mathbf{u}_2-\mathbf{u}_1;
\end{equation}
let us then determine $\beta_2$.

Let us choose $\chi\in \mathcal{C}^\infty_0\big([-\epsilon,\epsilon]\big)$
for some sufficiently small $\epsilon>0$. Keeping
$\theta_j, \mathbf{u}_j$ fixed, for $\lambda>0$
let $x_{j,\lambda}$ be as in (\ref{eqn:xx12})
and apply Theorem 1 in \cite{p24}.

On the one hand, 
for $\lambda\rightarrow +\infty$ we have
\begin{eqnarray}
\label{eqn:asympt main 1}
\lefteqn{  \Pi_{\chi,\lambda}(x_{1,\lambda},x_{2,\lambda}) }\\
&=&\frac{\chi(0)}{\sqrt{2\,\pi}}\cdot 
\left( \frac{\lambda}{2\,\pi\,\tau}  \right)^{d-1}\cdot 
e^{\frac{1}{\tau}\,\left[\imath\,\sqrt{\lambda}\,(\theta_1-\theta_2)
+\psi_2(\mathbf{u}_1,\mathbf{u}_2)\right]}
\cdot \left[1+O\left(\lambda^{-1/2}\right)\right].\nonumber
\end{eqnarray}

On the other hand, 
by 
(\ref{eqn:1st order terms}) we have
\begin{eqnarray}
x_{2,\lambda}&=&x_{1,\lambda}+
\left(\frac{1}{\sqrt{\lambda}}\,\left(\theta_2-\theta_1
+\frac{1}{\sqrt{\lambda}}\,\beta_2(\theta_j,\mathbf{u}_j)+
\sqrt{\lambda}\,\beta_3\left( \frac{\theta_j}{\sqrt{\lambda}},\,
\frac{\mathbf{u}_j}{\sqrt{\lambda}} \right)\right),\right.\nonumber\\
&&\left.
\frac{1}{\sqrt{\lambda}} \,(\mathbf{u}_2-\mathbf{u}_1 )+
\mathbf{R}_3\left(\frac{\theta_j}{\sqrt{\lambda}},
\frac{\mathbf{u}_j}{\sqrt{\lambda}}   \right)\right)
\end{eqnarray}

Hence, again by Theorem 1 in \cite{p24},
\begin{eqnarray}
\label{eqn:asympt main 2}
\lefteqn{  \Pi_{\chi,\lambda}(x_{1,\lambda},x_{2,\lambda}) }\\
&=&\frac{\chi(0)}{\sqrt{2\,\pi}}\cdot 
\left( \frac{\lambda}{2\,\pi\,\tau}  \right)^{d-1}\cdot 
e^{\frac{1}{\tau}\,\left[\imath\,\sqrt{\lambda}\,(\theta_1-\theta_2)
-\imath\,\beta_2(\theta_j,\mathbf{u}_j)
-\frac{1}{2}\,\|\mathbf{u}_1-\mathbf{u}_2\|^2
\right]}
\cdot \left[1+O\left(\lambda^{-1/2}\right)\right].\nonumber
\end{eqnarray}

Thus
$$
-\imath\,\beta_2(\theta_j,\mathbf{u}_j)
-\frac{1}{2}\,\|\mathbf{u}_1-\mathbf{u}_2\|^2=\psi_2(\mathbf{u}_1,\mathbf{u}_2)=-\imath\,\omega_x(\mathbf{u}_1,\mathbf{u}_2)
-\frac{1}{2}\,\|\mathbf{u}_1-\mathbf{u}_2\|^2.
$$
The claim follows.
\end{proof}

\begin{cor}
\label{cor:change of HLCs}
Under the same assumptions,
\begin{eqnarray*}
\lefteqn{\Big(x+(\theta,\mathbf{u})\Big)+\big(\beta,\mathbf{B}\big)   }\\
&=&x+\Big(\theta+\beta-\omega_x(\mathbf{u},\mathbf{B}) +R_3,
\mathbf{u}+\mathbf{B}+\mathbf{R}_2  \Big),
\end{eqnarray*}
where $R_3=R_3(\theta,\beta,
\mathbf{u},\mathbf{B})$, $\mathbf{R}_2=\mathbf{R}_2(\theta,\beta,
\mathbf{u},\mathbf{B})$.
\end{cor}

Let us now express the $G$-action near a point $x\in X^\tau$ in NHLC's.
We consider both the case of a fixed $g\in G_x$ (the stabilizer subgroup of $x$), and that of a 1-parameter
subgroup of $G$ generated by a given $\xi\in \mathfrak{g}$.
Let us premise a remark. If $g\in G_x$, then 
$\mathrm{d}_x\mu^\tau_g:T_xX^\tau\rightarrow T_xX^\tau$ preserves both the
vertical and horizontal tangent vector bundles at $x$; furthermore, 
$\mathrm{d}_x\mu^\tau_g\big( \mathcal{R}^\tau(x)  \big)=\mathcal{R}^\tau(x)$.
If $\mathbf{v}\in \mathcal{H}^\tau (x)$, we shall simplify notation and write
\begin{equation}
\label{eqn:vg}
\mathbf{v}_{g}:=\mathrm{d}_x\mu^\tau_g\big( \mathbf{v} \big).
\end{equation}

\begin{lem}
\label{lem:G action in local HCs}
Suppose $x\in X^\tau$ and choose a system of NHLC's centered at $x$.
Then the following holds.

\begin{enumerate}
\item If $\kappa\in G_x$, 
$$
\mu^\tau_{\kappa^{-1}}\big(x+(\theta,\mathbf{v})\big)=
x+\Big(\theta+R_3(\theta,\mathbf{v}),\mathbf{v}_{\kappa^{-1}}+\mathbf{R}_2(\theta,\mathbf{v})
\Big).
$$
\item If $\xi\sim 0\in \mathfrak{g}$ and
$(\theta,\mathbf{v})\sim (0,\mathbf{0})\in \mathbb{R}\times \mathbb{R}^{2d-2}$, 
then 
$$
\mu^\tau_{e^{-\xi}}\big(x+(\theta,\mathbf{v})\big)=
x+\Big(\theta
+\varphi^\xi(x)+\omega_x\big(\xi_{X^\tau}^\sharp(x),\mathbf{v}\big) 
+R_3,
\mathbf{v}-\xi^\sharp_{X^\tau}(x)+\mathbf{R}_2\Big),
$$
where 
$R_3=R_3(\theta,\mathbf{v},\xi)$, 
$\mathbf{R}_2=\mathbf{R}_2(\theta,\mathbf{v},\xi)$.
\end{enumerate}

\end{lem}

\begin{rem}
If $\mathbf{v}\in \mathcal{H}^\tau (x)$, then
$\omega_x\big(\xi_{X^\tau}^\sharp(x),\mathbf{v}\big) 
=\omega_x\big(\xi_{X^\tau}(x),\mathbf{v}\big)$.
\end{rem}

\begin{proof}
[Proof of 1]
Composing the given system of NHLC's with $\mu^\tau_{\kappa^{-1}}$ yields another system of
NHLC's on $X^\tau$ centered at $x$. Hence, by the discussion in 
\S 3.3.5 of \cite{p24}, the angular coordinates of the two systems only 
differ to third order.
\end{proof}

\begin{proof}
[Proof of 2.]
With the notation of Corollary \ref{cor: the moment map on tubes},
since $\mu^\tau$ and the geodesic flow commute
$$
\left[\xi^\sharp_{X^\tau},\varphi^\xi\,\mathcal{R}^\tau\right]
=\left[\xi^\sharp_{X^\tau},\mathcal{R}^\tau\right]
=0.
$$
Let us denote by $\Upsilon_{t,V}$ the flow at time $t$ of 
a smooth vector field
$V$ on $X^\tau$. 
Then, if $\mu^\tau_{t,\xi} := \Upsilon_{t,\xi_{X^\tau}}$, for any $t\in \mathbb{R}$:
$$
\mu^\tau_{-t,\xi}=\Upsilon_{t,\varphi^\xi\,
\mathcal{R}}\circ \Upsilon_{-t,\xi^\sharp_{X^\tau}}:X^\tau\rightarrow
X^\tau.
$$
For any $x'\in X^\tau$, the curve $t\mapsto \Upsilon_{-t,\xi^\sharp_{X^\tau}}(x)$
is horizontal, since $\xi^\sharp_{X^\tau}$ is tangent to 
$\mathcal{H}^\tau$ (here \lq horizontal\rq\, is the sense 
of (\ref{eqn:dicrect sum vert hor})).
By \S 3.3.7 of \cite{p24},
\begin{eqnarray*}
\lefteqn{\Upsilon_{-t,\xi^\sharp_{X^\tau}}\big(x+(\theta,\mathbf{v})\big)}\\
&=&\big(x+(\theta,\mathbf{v})\big)+\Big(R_3(t),-t\,\xi^\sharp_{X^\tau}(x)+
\mathbf{R}_2(t)\Big).
\end{eqnarray*}
In view of Lemma \ref{lem:thetajvj}, this may be rewritten
\begin{eqnarray*}
\lefteqn{ \Upsilon_{-t,\xi^\sharp_{X^\tau}}\big(x+(\theta,\mathbf{v})\big)  }\\
&=&x+\Big(\theta+t\,\omega_x(\mathbf{v}, \xi^\sharp_{X^\tau}(x) )+R_3,
\mathbf{v}-t\,\xi^\sharp_{X^\tau}(x)+\mathbf{R}_2\Big),
\end{eqnarray*}
where we abridge $R_3=R_3(\theta,t,\mathbf{v})$, 
and similarly for $\mathbf{R}_2$.
Then by Lemma \ref{eqn:local expression geodesic flow}
\begin{eqnarray}
\lefteqn{\mu^\tau_{-t,\xi}\big(x+(\theta,\mathbf{v})\big)
}\\
&=&\Upsilon_{t,\varphi^\xi\,
\mathcal{R}}\Big(x+\Big(\theta+t\,\omega_x(\mathbf{v}, \xi^\sharp_{X^\tau}(x) )+R_3,
\mathbf{v}-t\,\xi^\sharp_{X^\tau}(x)+\mathbf{R}_2\Big)\Big)\nonumber\\
&=&
x+\Big(\theta+t\,\omega_x(\mathbf{v}, \xi^\sharp_{X^\tau}(x) )
+t\,\varphi^\xi \big(x+(\theta,\mathbf{v})\big)
+R_3,
\mathbf{v}-t\,\xi^\sharp_{X^\tau}(x)+\mathbf{R}_2\Big)
\nonumber\\
&=&x+\Big(\theta+t\,\omega_x(\mathbf{v}, \xi^\sharp_{X^\tau}(x) )
+t\,\varphi^\xi(x)+2\,t\,\omega_x\big(\xi_{X^\tau}^\sharp,\mathbf{v}\big) 
+R_3,
\mathbf{v}-t\,\xi^\sharp_{X^\tau}(x)+\mathbf{R}_2\Big)
\nonumber\\
&=&x+\Big(\theta
+t\,\varphi^\xi(x)+t\,\omega_x\big(\xi_{X^\tau}^\sharp,\mathbf{v}\big) 
+R_3,
\mathbf{v}-t\,\xi^\sharp_{X^\tau}(x)+\mathbf{R}_2\Big).
\nonumber
\end{eqnarray}

\end{proof}

\begin{cor}
\label{cor:LHC action Ztau}
In the situation and with the notation 
of Lemma \ref{lem:G action in local HCs}, 
if in addition $x\in Z^\tau$ then
$$
\mu^\tau_{e^{-\xi}}\big(x+(\theta,\mathbf{v})\big)=
x+\Big(\theta
+\omega_x\big(\xi_{X^\tau}^\sharp(x),\mathbf{v}\big) 
+R_3,
\mathbf{v}-\xi^\sharp_{X^\tau}(x)+\mathbf{R}_2\Big).
$$
\end{cor}

We shall need to localize our computations in $X^\tau$ 
near points
$x\in Z^\tau$ and in $G$ near $G_x$.
Let $G_x=\{\kappa_1,\ldots,\kappa_r\}$, and for each $l=1,\ldots,r$
let us parametrize $G$ in the neighbourhood of $\kappa_l$ by setting
\begin{equation}
\label{eqn:Glocal par}
g=e^{\xi}\,\kappa_l,\quad\text{where}\quad
\xi\sim 0\in \mathfrak{g}.
\end{equation}
Furthermore, we shall further abridge notation (\ref{eqn:vg}),
and set
\begin{equation}\label{eqn:vl}
    \mathbf{v}^{(l)}:=\mathbf{v}_{k_l^{-1}}.
\end{equation}
Then, by Corollary \ref{cor:LHC action Ztau},
\begin{eqnarray}
\label{eqn:local action lxi}
\mu^\tau_{g^{-1}}\big( x+(\theta,\mathbf{v})  \big)   
&=&\mu^\tau_{\kappa_l^{-1}}\circ \mu^\tau_{e^{-\xi}}
\big( x+(\theta,\mathbf{v})  \big)\\
&=&\mu^\tau_{\kappa_l^{-1}}\left(x+\Big(\theta
+\omega_x\big(\xi_{X^\tau}^\sharp(x),\mathbf{v}\big) 
+R_3,
\mathbf{v}-\xi^\sharp_{X^\tau}(x)+\mathbf{R}_2\Big)    \right).
\nonumber
\end{eqnarray}
Applying 1. of Lemma \ref{lem:G action in local HCs}, we obtain
(with $g$ as in (\ref{eqn:Glocal par}))
\begin{eqnarray}
\label{eqn:action g thetav}
\lefteqn{\mu^\tau_{g^{-1}}\big( x+(\theta,\mathbf{v})  \big)}\\
&=&x+\Big(\theta
+\omega_x\big(\xi_{X^\tau}^\sharp(x),\mathbf{v}\big) 
+R_3,
\mathbf{v}^{(l)}-\xi^\sharp_{X^\tau}(x)^{(l)}+\mathbf{R}_2\Big).
\nonumber
\end{eqnarray}

\begin{rem}
For any $\xi\in \mathfrak{g}$ and $l=1,\ldots,r$,
$$\xi^\sharp_{X^\tau}(x)^{(l)}=\mathrm{Ad}_{\kappa_l^{-1}}(\xi)_{X^\tau}(x).$$
\end{rem}

\subsection{$\psi^\tau$ near orbit intersections}

Let us now assume that $x_1,x_2\in Z^\tau$ are such that
$x_1^G\cap x_2^\chi\neq \emptyset$, i.e.
$\Sigma_\chi(x_1,x_2)\neq \emptyset$ (recall (\ref{eqn:Sigmax12})).
Thus, assuming $|\mathrm{supp}(\chi)|$ is sufficiently small, 
$\Sigma_\chi(x_1,x_2)$ is as in (\ref{eqn:Sigmax12chi}).
By Corollary \ref{cor:unique t1chi},
$x_1^G\cap x_2^\chi=\{x_{12}\}$, where 
$$x_{12}:=\mu^\tau_{h_1^{-1}}(x_1)=\Gamma^\tau_{t_1}(x_2).$$

We shall need to expand the phase $\psi^\tau$
in (\ref{eqn:szego as FIO}) near 
$(x_{12},x_{12})$ and $(x_2,x_2)$.
We assume given systems of NHLC's on $X^\tau$ centered at
$x_1$ and $x_2$.
The choice of $h_1$ then uniquely determines
NHLC's on $X^\tau$ centered at $x_{12}$, by
the condition that $\mu^\tau_{h_1^{-1}}$ be locally represented by
the identity from a neighbourhood of $x_1$ to one of $x_{12}$.

\subsubsection{$\psi^\tau$ near $(x_{12},x_{12})$}

In the neighbourhood of $\kappa_l\,h_1$ 
(recall (\ref{eqn:Sigmax12chi})),
we can parameterize $G$ by setting
$$
g:=e^\xi\,\kappa_l\,h_1,\quad\text{where}\quad
\xi\sim 0\in \mathfrak{g}.
$$
Making use of (\ref{eqn:action g thetav}), we obtain
\begin{eqnarray}
\label{eqn:mutaugxih1kl}
\lefteqn{\mu^\tau_{g^{-1}}\big( x_1+(\theta_1,\mathbf{v}_1)  \big)=
\mu^\tau_{h_1^{-1}}\circ \mu^\tau_{\kappa_l^{-1}\,e^{-\xi}}
\big( x_1+(\theta_1,\mathbf{v}_1)\big)
}\\
&=&\mu^\tau_{h_1^{-1}}\left( x_1+\Big(\theta_1
+\omega_x\big(\xi_{X^\tau}^\sharp(x),\mathbf{v}_1\big) 
+R_3,
\mathbf{v}_1^{(l)}
-\xi^\sharp_{X^\tau}(x)^{(l)}+\mathbf{R}_2\Big)  \right)\nonumber\\
&=&x_{12}+\Big(\theta_1
+\omega_x\big(\xi_{X^\tau}^\sharp(x),\mathbf{v}_1\big) 
+R_3,
\mathbf{v}^{(l)}_1-\xi^\sharp_{X^\tau}(x)^{(l)}+\mathbf{R}_2\Big).
\nonumber
\end{eqnarray}

In the neighbourhood of $(x_{12},x_{12})$, 
we shall consider pairs of the form 
$$
\left(\mu^{\tau}_{g^{-1}}\big(x_1+(\theta_1,\mathbf{v}_1\big),
x_{12}+\big(\theta,\mathbf{v}\big)\right),\quad
g=e^\xi\,\kappa_l\,h_1;
$$
the first entry is given by (\ref{eqn:mutaugxih1kl}).
By the discussion in \S \ref{sctn:szego}
(see also Proposition 47 of \cite{p24}), we then have:

\begin{eqnarray}
\label{eqn:local phase at x12}
\lefteqn{\imath\,\psi^\tau\left(\mu^{\tau}_{g^{-1}}\big(x_1+(\theta_1,\mathbf{v}_1)\big),
x_{12}+\big(\theta,\mathbf{v}\big)\right)}\\
&=&  \imath\, \psi^\tau \left( x_{12}+\Big(\theta_1
+\omega_x\big(\xi_{X^\tau}^\sharp(x),\mathbf{v}_1\big) 
+R_3,
\mathbf{v}^{(l)}_1
-\xi^\sharp_{X^\tau}(x)^{(l)}+\mathbf{R}_2\Big)           , x_{12}+\big(\theta,\mathbf{v}\big)\right)                  \nonumber\\
&=& \imath\,\Big(\theta_1-\theta
+\omega_x\big(\xi_{X^\tau}^\sharp(x),\mathbf{v}_1\big)\Big)
-\frac{1}{4\,\tau^2}\,\Big(\theta_1-\theta
+\omega_x\big(\xi_{X^\tau}^\sharp(x),\mathbf{v}_1\big)\Big)^2 \nonumber\\
&&+\psi_2\left( \mathbf{v}^{(l)}_1
-\xi^\sharp_{X^\tau}(x)^{(l)},\mathbf{v}   \right)                                          
+R_3 \nonumber\\
&=&\imath\,\Big(\theta_1-\theta
+\omega_x\big(\xi_{X^\tau}^\sharp(x),\mathbf{v}_1\big)\Big)
-\frac{1}{4\,\tau^2}\,(\theta_1-\theta
)^2 +\psi_2\left( \mathbf{v}^{(l)}_1
-\xi^\sharp_{X^\tau}(x)^{(l)},\mathbf{v}   \right)                                          
+R_3, \nonumber
\end{eqnarray}
where 
$R_3=R_3(\theta_1,\theta,\mathbf{v}_1,\mathbf{v},\xi)$.

\subsubsection{$\psi^\tau$ near $(x_2,x_2)$}\label{subsctn: psi near (x2,x2)}
\label{sctn:psitaux1x2}

In the neighbourhood of $(x_2,x_2)$, we shall consider points of the
form 
\begin{equation}
\label{eqn:nearx2x2}
\left( \Gamma^\tau_{-t_1-t}\big(x_{12}+(\theta,\mathbf{v})\big),
x_2+(\theta_2,\mathbf{v}_2)
\right),
\end{equation}
where $t\sim 0$; recall that $\Gamma^\tau_{-t_1}(x_{12})=x_2$.

We aim to express the first entry
in (\ref{eqn:nearx2x2}) in NHLC's on $X^\tau$ centered at $x_2$.
Since $\alpha^\tau$ is invariant by 
the geodesic flow, $\Gamma^\tau_t$ preserves the
vertical and horizontal tangent bundles. In particular, for any
$s\in \mathbb{R}$ and $\mathbf{v}\in \mathcal{H}^\tau_{x_{12}}$ 
with some abuse of notation we have
\begin{equation}
\label{eqn:differential geo flow}
\mathrm{d}_{x_{12}}\Gamma^\tau_{-t_1}\left(s\,\mathcal{R}^\tau
(x_{12})+\mathbf{v})\right)
=s\,\mathcal{R}^\tau
(x_{2})+B\,\mathbf{v},
\end{equation}
where the symplectic matrix $B$ is as in (\ref{eqn:defn of M12}).

\begin{lem}
\label{lem:Mt1transl}
Under the previous assumptions,
\begin{eqnarray*}
\Gamma^\tau_{-t_1}\big(x_{12}+(\theta,\mathbf{v})\big)
&=&x_2+\big(\theta+R_3(\theta,\mathbf{v}),
B\mathbf{v}+\mathbf{R}_2(\theta,\mathbf{v})\big).
\end{eqnarray*} 

\end{lem}

Before proving Lemma \ref{lem:Mt1transl},
let us give a definition.

\begin{defn}
\label{defn:horizontal degree k}
Let $\gamma:(-a,a)\rightarrow X^\tau$ be a smooth curve 
defined for some $a>0$ and set $\gamma(0)=x$. Let 
$k\ge 1$ be an integer. The $\gamma$ will be 
said to be \textit{horizontal to $k$-th order at $x$}
if $\left\langle \alpha,\dot{\gamma}\right\rangle=
O\left(t^k\right)$ for $t\sim 0$.

\end{defn}

\begin{lem}
\label{lem:2nd order horizontal}
In the situation of Definition \ref{defn:horizontal degree k},
the following conditions are equivalent:

\begin{enumerate}
\item $\gamma$ is horizontal to second order at $x$, that is,
$\left\langle \alpha,\dot{\gamma}\right\rangle=
O\left(t^2\right)$ for $t\sim 0$;
\item given any system of NHLC's on $X^\tau$ centered at $x$, 
for an appropriate $\mathbf{v}\in \mathcal{H}^\tau_x$ we have
$$
\gamma (t)=x+\big(R_3(t),t\,\mathbf{v}+\mathbf{R}_2(t)
\big).
$$
\end{enumerate}

\end{lem}

\begin{proof}
NHLC's $(\theta,z')$ on $X^\tau$ centered at $x$ are induced,
under projection and restriction, by NHLC's
$(z_0,z_1,\ldots,z_{d-1})$, on $\tilde{M}$ adapted to $X^\tau$ 
and centered at $x$ (see \S \ref{sctn:szego});
here $z_0=\theta+\imath\,\eta_0$ (with $\theta$ and $\eta_0$ real). 
By Proposition 33 of \cite{p24}, 
\begin{eqnarray*}
\alpha&=&\mathrm{d}\theta_0+\frac{1}{2\,\imath}\,\left[\frac{1}{2\,\tau^2}\,
\left( \overline{z}_0\,\mathrm{d}z_0-z_0\,\mathrm{d}\overline{z}_0  \right)
+\overline{z}'\cdot \mathrm{d}z'-z'\cdot \mathrm{d}\overline{z}'\right]\\
&&+R_2(z,\overline{z}).
\end{eqnarray*}
If $\gamma(t)=\big(\theta(t),z'(t)\big)$ in local coordinates on $X^\tau$, then $\gamma (t)=\big(\theta (t)+\imath\,\eta_0(t),z'(t)\big)$
in the corresponding NHLC's on $\tilde{M}$.
In particular, $z_0(t)=\theta(t)+\imath\,\eta_0(t)$.

By Corollary 35 of \cite{p24}, $\eta_0 (t)=R_2(t)$. 
Hence 
$$
\gamma^*\left( \overline{z}_0\,\mathrm{d}z_0-z_0\,\mathrm{d}\overline{z}_0  \right)=2\,\imath\,\gamma^*(\theta\,\mathrm{d}\eta_0-\eta_0\,\mathrm{d}\theta)=R_2(t).
$$
If $z'(t)=t\,w+\mathbf{R}_2(t)$, then 
\begin{eqnarray*}
\lefteqn{\gamma^*\left(\overline{z}'\cdot \mathrm{d}z'-z'\cdot \mathrm{d}\overline{z}' \right)}\\
&=&\left[\left(t\,\overline{w}+\mathbf{R}_2(t)\right)\cdot
\big(w+R_1(t)\big)-\left(t\,w+\mathbf{R}_2(t)\right)\cdot
\big(\overline{w}+R_1(t)\big)\right]\,\mathrm{d}t\\
&=&R_2(t).
\end{eqnarray*}
Hence,
$\gamma^*(\alpha)=\theta'(t)\,\mathrm{d}t+R_2(t)$.
Thus, $\left\langle \alpha,\dot{\gamma}\right\rangle=
R_2(t)$ if and only if $\theta'(t)=R_2(t)$ if and only if $\theta(t)=
R_3(t)$.

\end{proof}

\begin{proof}
[Proof of Lemma \ref{lem:Mt1transl}]
By Lemma \ref{eqn:local expression geodesic flow}, we have
$$
x_{12}+(\theta,\mathbf{v})=
\Gamma^\tau_{-\theta/\tau}\Big(x_{12}+\big(R_3(\theta,\mathbf{v}),
\mathbf{v}+\mathbf{R}_2(\theta,\mathbf{v})\big)\Big),
$$
for certain functions
$R_3$ and $\mathbf{R}_2$ (with the usual conventions 
about vanishing orders, and with $R_j$, $\mathbf{R}_j$ allowed to
vary from line to line).
By (\ref{eqn:differential geo flow}), we can write
$$
\Gamma_{-t_1}\Big(x_{12}+\big(R_3(\theta,\mathbf{v}),
\mathbf{v}+\mathbf{R}_2(\theta,\mathbf{v})\big)\Big)
=x_2+\big(R_2^{(1)}(\theta,\mathbf{v}),B\mathbf{v}+\mathbf{R}^{(1)}_2(\theta,\mathbf{v})\big),
$$
for certain functions vanishing to the indicated orders.

By Lemma \ref{lem:2nd order horizontal},
the smooth curve $\gamma:(-a,a)\rightarrow X^\tau$, defined for some
$a\ge 1$ by
$$
\gamma (t):=x_{12}+\big(R_3(t\,\theta,t\,\mathbf{v}),
\mathbf{v}+\mathbf{R}_2(t\,\theta,t\,\mathbf{v})\big),
$$
is horizontal to second order at $x_{12}$.
Since the geodesic flow $\Gamma^\tau$ preserves $\alpha^\tau$,
$\Gamma^\tau_{-t_1}\circ \gamma$ is horizontal to second order
at $x_2$. 
Hence
$R_2^{(1)}(t\,\theta,t\,\mathbf{v})=O\left(t^3\right)$,
and therefore $R_2^{(1)}(\theta,\mathbf{v})$ really vanishes to 
third order at $(0,\mathbf{0})$; we shall accordingly replace it by 
$R_3^{(1)}(\theta,\mathbf{v})$.

Therefore,
\begin{eqnarray*}
\Gamma^\tau_{-t_1}\big(x_{12}+(\theta,\mathbf{v})\big)&=&
\Gamma^\tau_{-t_1}\circ\Gamma^\tau_{-\theta/\tau}\Big(x_{12}+\big(R_3(\theta,\mathbf{v}),
\mathbf{v}+\mathbf{R}_2(\theta,\mathbf{v})\big)\Big)\\
&=&\Gamma^\tau_{-\theta/\tau}\circ\Gamma^\tau_{-t_1}\Big(x_{12}+\big(R_3(\theta,\mathbf{v}),
\mathbf{v}+\mathbf{R}_2(\theta,\mathbf{v})\big)\Big)                                  \\
&=&\Gamma^\tau_{-\theta/\tau}\left( x_2+\big(R_3^{(1)}(\theta,\mathbf{v}),B\mathbf{v}+\mathbf{R}^{(1)}_2(\theta,\mathbf{v})\big)   \right)\\
&=&x_2+\Big(\theta+R_3^{(2)}(\theta,\mathbf{v}),B\mathbf{v}+\mathbf{R}^{(2)}_2(\theta,\mathbf{v})\big),  
\end{eqnarray*} 
where in the last equality 
we have made use of Lemma \ref{eqn:local expression geodesic flow}.
\end{proof}

Using one more time Lemma \ref{eqn:local expression geodesic flow}, we
obtain the following upshot.

\begin{cor}
\label{cor:Gamma t t1}
With the previous notation,
\begin{eqnarray*}
\lefteqn{\Gamma^\tau_{-t-t_1}\Big(x_{12}+(\theta,\mathbf{v})\Big)}\\
&=&x_2+\Big(  \theta+\tau\,t+R_3(t,\theta,\mathbf{v}),
B\mathbf{v}+\mathbf{R}_2(t,\theta,\mathbf{v})\Big).
\end{eqnarray*}

\end{cor}

Invoking again Proposition 47 of \cite{p24}
(see also \S \ref{sctn:szego} above), we conclude that

\begin{eqnarray}
\label{eqn:psi at nearx2x2}
\lefteqn{
\imath\,\psi^\tau\left( \Gamma^\tau_{-t_1-t}\big(x_{12}+(\theta,\mathbf{v})\big),
x_2+(\theta_2,\mathbf{v}_2)
\right)}\\
&=&\imath\,\psi^\tau\left(x_2+\Big(  \theta+\tau\,t+R_3(t,\theta,\mathbf{v}),
B\mathbf{v}+\mathbf{R}_2(t,\theta,\mathbf{v})\Big) ,
x_2+(\theta_2,\mathbf{v}_2)    \right)\nonumber\\
&=& \imath\,(  \theta+\tau\,t-\theta_2)
-\frac{1}{4\,\tau^2}   \,(  \theta+\tau\,t-\theta_2)^2
+\psi_2\big(B\mathbf{v},\mathbf{v}_2\big)+R_3(\theta,\mathbf{v},
\theta_2,\mathbf{v}_2).
 \nonumber
\end{eqnarray}

\section{Proof of Theorem \ref{thm:main 1}
(case of $\Pi^\tau_{\chi,\nu,\lambda}$)}

\subsection{Preamble}\label{sctn:preamble1}

In this section, we shall prove Theorem \ref{thm:main 1}
for $\Pi^\tau_{\chi,\nu,\lambda}$; the formal changes 
to the argument needed for 
$P^\tau_{\chi,\nu,\lambda}$ will be described in \S 
\ref{sctn:complexified eigenf lapl}.
Before giving the proof, we shall lay down
some preliminaries and pieces of notation. 


\begin{enumerate}

\item With $\mathfrak{D}^\tau_{\sqrt{\rho}}$ as in (\ref{eqn:DsqrtrhoToeplitz}), we shall denote by
$$
U^\tau_{\sqrt{\rho}}(t):=e^{\imath\,t\,\mathfrak{D}^\tau_{\sqrt{\rho}}}
\qquad (t\in \mathbb{R})
$$
the $1$-parameter group of unitary Toeplitz operators generated by 
$\mathfrak{D}^\tau_{\sqrt{\rho}}$.

\item The distributional kernel of $\Pi^\tau_{\chi,\lambda}$
in (\ref{eqn:Pi chi lambda}) is
related to that of $U^\tau_{\sqrt{\rho}}(t)$ by
\begin{equation}
\label{eqn:PitaunuFIO}
\Pi^\tau_{\chi,\lambda}(x,y)=
\frac{1}{\sqrt{2\,\pi}}\,\int_{-\infty}^{+\infty}
e^{-\imath\,\lambda\,t}\,\chi (t)\,U^\tau_{\sqrt{\rho}}
(t;x,y)\,\mathrm{d}t.
\end{equation}

\item For $t\in \mathbb{R}$, let us denote by 
$\Pi^\tau_{-t}:L^2(X^\tau)\rightarrow L^2(X_\tau)$ the operator 
having distributional kernel
$$
\Pi^\tau_{-t}(x,y):=\Pi^\tau\left(\Gamma^\tau_{-t}(x),y\right).
$$

\item Following Zelditch (see e.g.
\cite{z97} and \cite{z20}),
there exists a zeroth order pseudodifferential operator
$P^\tau_t$ on $X^\tau$ (depending smoothly on $t$), such that 
\begin{equation}
\label{eqn:Ptautcomp}
U^\tau_{\sqrt{\rho}}(t)\sim \Pi^\tau\circ P^\tau_t\circ \Pi^{\tau}_{-t},
\end{equation}
where $\sim$ stands for \lq equal up to smoothing operators\rq. 
More precisely, there exists a classical polyhomogeneous symbol
of the form
\begin{equation}
\label{eqn:semiclassical symbol}
\sigma^\tau_t(x,r)\sim
\sum_{j=0}^{+\infty}\sigma^\tau_{t,j}(x)\,r^{-j},
\end{equation}
such that 
$$P^\tau_t\sim \sigma^\tau_t\left(x,D^\tau_{\sqrt{\rho}}\right).$$
The leading order term in (\ref{eqn:semiclassical symbol})
(equivalently, the principal symbol of $P^\tau_t$)
can be described, up to a unitary factor, as
follows. 

\item 
Let us set $J^\tau_t:=\mathrm{d}\Gamma^\tau_t\circ J\circ \mathrm{d}\Gamma^\tau_{-t}$. Then $J^\tau_t$ is a new CR structure on $X^\tau$,
with corresponding Hardy space
$$
H(X^\tau)_t:={\Gamma^\tau_{-t}}^*\big(H(X^\tau)\big),
$$
and corresponding Szeg\H{o} kernel
$$
\tilde{\Pi}^\tau_t:={\Gamma^\tau_{-t}}^*\circ \Pi^\tau\circ {\Gamma^\tau_{t}}^*.
$$
Thus, $\tilde{\Pi}^\tau_0=\Pi^\tau$; furthermore,
the distributional kernel of $\tilde{\Pi}^\tau_t$ is
$$
\tilde{\Pi}^\tau_t(x,y):=
\left(\Gamma^\tau_{-t}\times \Gamma^\tau_{-t}\right)^*
\left( \Pi^\tau\right) (x,y) 
=\Pi^\tau\left(\Gamma^\tau_{-t}(x),\Gamma^\tau_{-t}(y)  \right).
$$

\item 
For every $x\in X^\tau$, the vacuum states at $x$ associated to $J$ and $J_t$,
denoted $\sigma_{J}^{(x)}$ and $\sigma_{J_t}^{(x)}$, 
are Gaussian functions on
the horizontal tangent space at $x$. 
Their $L^2$-pairing 
$\langle \sigma_J^{(x)},\sigma_{J_t}^{(x)}\rangle$ varies smoothly with $x$ and $t$ and is nowhere vanishing. 
Furthermore, there is a smooth function $\theta^\tau_t(x)$ such that 
\begin{equation}
\label{eqn:correcting factor tau t}
\sigma^\tau_{t,0}(x)=e^{\imath\,\theta^\tau_t(x)}\cdot 
\langle \sigma_J^{(x)},\sigma_{J_t}^{(x)}\rangle^{-1}.
\end{equation}

\item Given the usual description of $\Pi^\tau$ as an FIO with complex
phase of positive type recalled in section \ref{sctn:szego}, we have
$$
\Pi^\tau_{-t}(x,y)\sim
\int_0^{+\infty}
e^{\imath\,v\,\psi^\tau\left(\Gamma^\tau_{-t}(x),y  \right)}\,
s^\tau\left(\Gamma^\tau_{-t}(x),y ,v\right)\,\mathrm{d}v.
$$
It follows that the Schwartz kernel of 
$\mathcal{P}^\tau_t:=P^\tau_t\circ \Pi^\tau_{-t}$ in (\ref{eqn:Ptautcomp}) is given by
\begin{eqnarray}
\label{eqn:mathcalPtau}
\mathcal{P}^\tau_t(x,y)\sim 
\int_0^{+\infty}
e^{\imath\,v\,\psi^\tau\left(\Gamma^\tau_{-t}(x),y  \right)}\,
r^\tau_t\left(x,y ,v\right)\,\mathrm{d}v
\end{eqnarray}
where 
\begin{equation}
\label{eqn:rttauexp}
r^\tau_t\left(x,y ,v\right)\sim
\sum_{j\ge 0}v^{d-1-j}\,r_{t,j}^\tau(x,y),
\end{equation}
and
\begin{equation}
\label{eqn:lead r tau t 0}
r_{t,0}^\tau(x,y)=
\sigma_{t,0}^\tau (x)\cdot s^\tau_0\left(\Gamma^\tau_{-t}(x),y\right).
\end{equation}
\end{enumerate}

\subsection{The proof}

We can now attack the proof of Theorem \ref{thm:main 1}.
Some of the arguments are an equivariant version of others
in \cite{p24}, but we try to make the exposition
reasonably self-contained.
We shall divide the statement of the Theorem in two parts, and
prove each part separately:

\begin{enumerate}
\item $\Pi_{\chi,\nu,\lambda}^\tau(x_1,x_2)=O\left(\lambda^{-\infty}\right)$
uniformly for
$
\mathrm{dist}_{X^\tau}(x_1,x_2^{G\times \mathrm{supp}(\chi)}
)
\ge C\,\lambda^{\epsilon'-\frac{1}{2}}
$.
\item $
\Pi_{\chi,\nu,\lambda}^\tau(x_1,x_2)=O\left(\lambda^{-\infty}\right)$
uniformly for 
$$\max\left\{\mathrm{dist}_{X^\tau}\left(x_1,Z^\tau\right),
\mathrm{dist}_{X^\tau}\left(x_2,Z^\tau\right)\right\}
\ge C\,\lambda^{\epsilon'-\frac{1}{2}}.
$$

\end{enumerate}

\begin{proof}
[Proof of Theorem \ref{thm:main 1}, Part 1.]
By (\ref{eqn:Pi chi nu lambda projector}), 
we have for $x_1,\,x_2\in X^\tau$:
\begin{eqnarray}
\label{eqn:1st expression}
\lefteqn{\Pi^\tau_{\chi,\nu,\lambda}(x_1,x_2)}\\
&=&\dim(\nu)\,
\int_G\,\Xi_\nu 
\left( g^{-1} 
\right)\,\Pi_{\chi,\lambda}
\left( 
\mu^\tau_{g^{-1}}(x_1),x_2 
\right)   
\,\mathrm{d}V_G(g).\nonumber
\end{eqnarray}
In view of (\ref{eqn:PitaunuFIO}), we may reformulate 
(\ref{eqn:1st expression}) as
\begin{eqnarray}
\label{eqn:2nd expression}
\lefteqn{\Pi^\tau_{\chi,\nu,\lambda}(x_1,x_2)}\\
&=&\frac{\dim(\nu)}{\sqrt{2\,\pi}}\,
\int_G\,\mathrm{d}V_G(g)\,\int_{-\infty}^{+\infty}\,\mathrm{d}t\,
\left[\Xi_\nu 
\left( g^{-1} 
\right)\,
e^{-\imath\,\lambda\,t}\,
\chi (t)\,  
U^\tau_{\sqrt{\rho}}\left(t;\mu^\tau_{g^{-1}}(x_1),x_2\right)
\right].\nonumber
\end{eqnarray}
Therefore, given (\ref{eqn:Ptautcomp}), for 
$\lambda\rightarrow +\infty$ we have
\begin{eqnarray}
\label{eqn:3rd expression}
\lefteqn{\Pi^\tau_{\chi,\nu,\lambda}(x_1,x_2)}\\
&\sim&\frac{\dim(\nu)}{\sqrt{2\,\pi}}\,
\int_G\,\mathrm{d}V_G(g)\,\int_{-\infty}^{+\infty}\,\mathrm{d}t
\,
\left[\Xi_\nu 
\left( g^{-1} 
\right)\,
e^{-\imath\,\lambda\,t}\,
\chi (t)\,  
\left(\Pi^\tau\circ P^\tau_t\circ \Pi^\tau_{-t}    \right)
\left(\mu^\tau_{g^{-1}}(x_1),x_2\right)
\right];\nonumber
\end{eqnarray}
here, $\sim$ means \lq has the same asymptotics as\rq.

The singular support of $\left(\Pi^\tau\circ P^\tau_t\circ \Pi^\tau_{-t}    \right)$ is the set of pairs 
$(x,y)\in X^\tau\times X^\tau$ such that
$x=\Gamma^\tau_t(y)$. Hence, $\left(\Pi^\tau\circ P^\tau_t\circ \Pi^\tau_{-t}    \right)$ is smooth at $\left(\mu^\tau_{g^{-1}}(x_1),x_2\right)$
unless $x_1=\mu^\tau_g\circ \Gamma^\tau_t(x_2)$.

Suppose then that $x_1\not\in x_2^{G\times \mathrm{supp}(\chi)}$.
Then the function 
$$
t\mapsto \chi(t)\cdot 
\int_G\,\mathrm{d}V_G(g)
\,
\left[\Xi_\nu 
\left( g^{-1} 
\right)\,
\left(\Pi^\tau\circ P^\tau_t\circ \Pi^\tau_{-t}    \right)
\left(\mu^\tau_{g^{-1}}(x_1),x_2\right)
\right]
$$
is smooth and compactly supported. Hence its Fourier
transform is of rapid decrease.

We conclude the following. Let us define
$$
\mathcal{K}^{G\times \mathrm{supp}(\chi)}:=\left\{
(x_1,x_2)\in X^\tau\times X^\tau\,:\,x_1\in x_2^{G\times \mathrm{supp}(\chi)}\right\},
$$
a compact subset of $X^\tau\times X^\tau$.

\begin{lem}
\label{lem:rapid decreas 1}
Let $X'\Subset X^\tau\times X^\tau$ be any open neighbourhood
of $\mathcal{K}^{G\times \mathrm{supp}(\chi)}$. Then
$$
\Pi^\tau_{\chi,\nu,\lambda}(x_1,x_2)=O\left(\lambda
^{-\infty}\right)\quad
\text{as}\quad \lambda\rightarrow +\infty
$$
uniformly for $(x_1,x_2)\not\in X'$.
\end{lem}

We may thus assume that $x_1$ belongs to an arbitrarily small
neighbourhood of $x_2^{G\times \mathrm{supp}(\chi)}$.
We may rewrite (\ref{eqn:3rd expression}) as follows:

\begin{eqnarray}
\label{eqn:4th expression}
\lefteqn{\Pi^\tau_{\chi,\nu,\lambda}(x_1,x_2)}\\
&\sim&\frac{\dim(\nu)}{\sqrt{2\,\pi}}\,
\int_G\,\mathrm{d}V_G(g)\,\int_{-\infty}^{+\infty}\,\mathrm{d}t
\,\int_{X^\tau}\,\mathrm{d}V_{X^\tau}(y)\nonumber\\
&&\left[\Xi_\nu 
\left( g^{-1} 
\right)\,
e^{-\imath\,\lambda\,t}\,
\chi (t)\,  \Pi^\tau \left(\mu^\tau_{g^{-1}}(x_1),y\right)
\mathcal{P}^\tau_t
\left(y,x_2\right)
\right],\nonumber
\end{eqnarray}
where $\mathcal{P}^\tau_t$ is as in (\ref{eqn:mathcalPtau}).

The singular support of $\Pi^\tau$ is the diagonal in
$X^\tau\times X^\tau$ \cite{bs}. Hence, only a negligible contribution to
the asymptotics of (\ref{eqn:4th expression}) is lost,
if integration in $y$ is restricted to a small neighbourhood
of $x_1^G$. More precisely, let $\varrho_1(g,\cdot)$ be a cut-off function, smoothly varying with $g$, 
identically equal to $1$ sufficiently close to
$\mu^\tau_{g^{-1}}(x_1)$, 
but vanishing identically outside a small open neighbourhood 
of the point. Then only a rapidly decreasing contribution to
the asymptotics is lost, if the integrand in 
(\ref{eqn:4th expression}) is multiplied by $\varrho_1(g,y)$.
Similarly, the singular support of 
$\mathcal{P}^\tau_t$ in (\ref{eqn:mathcalPtau})
is the set of pairs $(x',x'')$ with 
$x'=\Gamma^\tau_t(x'')$. Again, we conclude that the
asymptotics of (\ref{eqn:4th expression}) will be unaltered, 
if the integrand is further multiplied by $\varrho_2(t,y)$, 
where $\varrho_2(t,\cdot)$ varies smoothly with $t$,
 is supported on a small neighbourhood of
$\Gamma^\tau_t(x_2)$, and identically equal to $1$
sufficiently close to it.

Given this, the pairs $\big(\mu^\tau_{g^{-1}}(x_1),y\big)$
and $\big(\Gamma^\tau_{-t}(y),x_2\big)$ belong to small
neighbourhoods of the diagonal. Therefore, on the domain of
integration we may replace $\Pi^\tau$ and $\mathcal{P}^\tau_t$
by the representations as FIO's with complex phase \cite{bs},
perhaps at the cost of losing a negligible contribution to
the asymptotics.

Thus, as $\lambda\rightarrow +\infty$,
\begin{eqnarray}
\lefteqn{\Pi^\tau_{\chi,\nu,\lambda}(x_1,x_2)}\\
&\sim&\frac{\dim(\nu)}{\sqrt{2\,\pi}}\,
\int_G\,\mathrm{d}V_G(g)\,\int_{-\infty}^{+\infty}\,\mathrm{d}t
\,\int_{X^\tau}\,\mathrm{d}V_{X^\tau}(y)\,
\int_0^{+\infty}\,\mathrm{d}u\,
\int_0^{+\infty}\,\mathrm{d}v\nonumber\\
&&\left[\Xi_\nu 
\left( g^{-1} 
\right)\,
e^{-\imath\,\lambda\,t}\,
\chi (t)\,  e^{\imath\,u\,\psi^\tau\left(\mu^\tau_{g^{-1}}(x_1),y\right)+\imath\,v\,\psi^\tau\left(\Gamma^\tau_{-t}(y),x_2\right)}
\right.\nonumber\\
&&\left. 
\varrho_1(g,y)\,\varrho_2(t,y)\,
s^\tau\left(\mu^\tau_{g^{-1}}(x_1),y,u\right)\,
r^\tau_t\left(y,x_2,v\right)
\right],\nonumber
\end{eqnarray}
where $r^\tau_t$ is as in (\ref{eqn:rttauexp}).
Let us now operate the rescaling
$u\mapsto \lambda\,u,\,v\mapsto \lambda\,v$:
\begin{eqnarray}
\label{eqn:5th expression}
\lefteqn{\Pi^\tau_{\chi,\nu,\lambda}(x_1,x_2)}\\
&\sim&\lambda^2\,\frac{\dim(\nu)}{\sqrt{2\,\pi}}\,
\int_G\,\mathrm{d}V_G(g)\,\int_{-\infty}^{+\infty}\,\mathrm{d}t
\,\int_{X^\tau}\,\mathrm{d}V_{X^\tau}(y)\,
\int_0^{+\infty}\,\mathrm{d}u\,
\int_0^{+\infty}\,\mathrm{d}v\nonumber\\
&&\left[\Xi_\nu 
\left( g^{-1} 
\right)\,
\chi (t)\,  e^{\imath\,\lambda\,\left[u\,\psi^\tau\left(\mu^\tau_{g^{-1}}(x_1),y\right)+v\,\psi^\tau
\left(\Gamma^\tau_{-t}(y),x_2\right)-t\right]}
\right.\nonumber\\
&&\left. 
\varrho_1(g,y)\,\varrho_2(t,y)\,
s^\tau\left(\mu^\tau_{g^{-1}}(x_1),y,\lambda\,u\right)\,
r^\tau_t\left(y,x_2,\lambda\,v\right)
\right].\nonumber
\end{eqnarray}

Let us set 
\begin{equation}
\label{eqn:fasePsi'}
\Psi(x_1,x_2;g,t,y,u,v):=
u\,\psi^\tau\left(\mu^\tau_{g^{-1}}(x_1),y\right)+v\,\psi^\tau
\left(\Gamma^\tau_{-t}(y),x_2\right)-t.
\end{equation}

The next step will be to argue that integration in
$u$ and $v$ may be restricted to certain compact 
neighbourhoods of $1/\tau$ in $\mathbb{R}_+$.

\begin{prop}
\label{prop:compact u and v}
There exist cut-off functions $f_1,\,f_2\in \mathcal{C}^\infty_c(\mathbb{R})$, identically equal to $1$ near $1/\tau$,
such that the asymptotics of (\ref{eqn:5th expression}) are
unaltered, if the integral is multiplied by 
$f_1(v)\cdot f_2(u)$.

\end{prop}

For notational simplicity, 
the cut-off $f_1(v)\cdot f_2(u)$ will be implicitly absorbed
in the amplitude.

\begin{proof}
As remarked, by the previous reductions,
on the domain of integration
$\big(\mu^\tau_{g^{-1}}(x_1),y\big)$
and $\left(\Gamma^\tau_{-t}(y,)x_2\right)$ now belong to 
a small
neighbourhood of the diagonal of, 
say, tubular radius $\delta>0$. 
Hence, in local coordinates we have
\begin{eqnarray}
\label{eqn:psidiff diagonal}
\mathrm{d}_{\big(\mu^\tau_{g^{-1}}(x_1),\, y\big)}\psi^\tau
&=&\left(\alpha^\tau_{\mu^\tau_{g^{-1}}(x_1)},
-\alpha^\tau_{\mu^\tau_{g^{-1}}(x_1)}\right)+O(\delta),\\
\mathrm{d}_{\left(\Gamma^\tau_{-t}(y),\, x_2\right)}\psi^\tau
&=&\left(\alpha^\tau_{x_2},
-\alpha^\tau_{x_2}\right)+O(\delta).\nonumber
\end{eqnarray}

Localizing the computation near some 
$\left(\tilde{g},\tilde{t}\right)$, we shall set
$\tilde{x}:=\mu^\tau_{\tilde{g}^{-1}}(x_1)$ 
(so that there is a naturally induced system of NHLC's
centered at $\tilde{x}$)
and
$$
g:=\tilde{g}\,e^\xi,\quad
y=\tilde{x}+(\theta,\mathbf{v}),\quad t=\tilde{t}+a.
$$
In local coordinates we get:
$$
\Gamma^\tau_{-t}(y)=\Gamma^\tau_{-a}
\circ \Gamma_{-\tilde{t}}(y)=\Gamma^\tau_{-a}(x_2)+O(\delta).
$$
Therefore, we conclude from 
Lemma \ref{eqn:local expression geodesic flow} that
$$
\partial_t\Psi=v\,\tau-1+O(\delta).
$$

Since the variable $t$ is compactly supported, it is legitimate to 
integrate by parts in $t$, and we conclude that the contribution
to the asymptotics of (\ref{eqn:5th expression}) of the locus
where $0<v\ll 1/\tau$ or $v\gg 1/\tau$ is rapidly decreasing.

More precisely, we conclude the following.

\begin{lem}
\label{lem:red in v}
The asymptotics of (\ref{eqn:5th expression}) are unchanged, if
the integrand is multiplied by $f_1(v)$, where
$f_1\in \mathcal{C}^\infty_c(\mathbb{R}_+)$ is identically
equal to $1$ on a suitable neighbourhood of $1/\tau$.
\end{lem}

In the following, to simplify notation the cut-off $f_1(v)$
will be absorbed in the amplitude of (\ref{eqn:5th expression}).

Let us adopts NHLC's centered at 
$\tilde{x}:=\mu^\tau_{\tilde{g}^{-1}}(x_1)$, and write
$$
y=\tilde{x}+(\theta,\mathbf{v}),
$$
where $\lVert(\theta,\mathbf{v})\rVert=O(\delta)$.
Then
\begin{eqnarray}
\label{eqn:1st summand loc coor}
\lefteqn{
\imath\,\psi^\tau\left(\mu^\tau_{g^{-1}}(x_1),y\right)
=\imath\,\psi^\tau\left(\mu^\tau_{e^{-\xi}}\left(\tilde{x}\right),
\tilde{x}+(\theta,\mathbf{v})
\right)}\\
&=&
\imath\,\psi^\tau\left(
\tilde{x}+\Big(\langle\Phi(\tilde{x}),\xi\rangle
+R_3(\xi),
-\xi^\sharp_{X^\tau}(\tilde{x})+\mathbf{R}_2(\xi)\Big),\tilde{x}+(\theta,\mathbf{v})
\right)\nonumber\\
&=&\imath\,\big(\langle\Phi(\tilde{x}),\xi\rangle-\theta  \big)
-\frac{1}{4\,\tau^2}\,\big(\langle\Phi(\tilde{x}),\xi\rangle-\theta  \big)^2
-\frac{1}{2}\,\left\| \xi^\sharp_{X^\tau}(\tilde{x}) \right\|^2
+R_3(\xi,\theta).
\nonumber
\end{eqnarray}

Let us fix a system of NHLC's at 
$\Gamma^\tau_{-\tilde{t}}(\tilde{x})$.
Furthermore, arguing as in Lemma \ref{lem:Mt1transl}, for a suitable symplectic matrix $M_{\tilde{t}}$,
we obtain
\begin{eqnarray}
\label{eqn:ytansfGamma}
\Gamma^\tau_{-a-\tilde{t}}(y)
&=&\Gamma^\tau_{-a}\circ \Gamma^\tau_{-\tilde{t}}
\Big(\tilde{x}+(\theta,\mathbf{v}\big)\Big)\\
&=&\Gamma^\tau_{-a}\left(\Gamma^\tau_{-\tilde{t}}(\tilde{x})+
\Big(\theta+R_3(\theta), M_{\tilde{t}}\,\mathbf{v}+
\mathbf{R}_2(\theta,\mathbf{v})\Big)\right)
\nonumber
\\
&=&\Gamma^\tau_{-\tilde{t}}(\tilde{x})+
\Big(\theta+\tau\,a+R_3(\theta,a), M_{\tilde{t}}\,\mathbf{v}+
\mathbf{R}_2(\theta,\mathbf{v},a)\Big)
\nonumber
\end{eqnarray}

By assumption
$\Gamma^\tau_{-\tilde{t}}(\tilde{x})
=x_2+\big(R_1(\delta),\mathbf{R}_1(\delta)\big)$. 
Hence by Corollary \ref{cor:change of HLCs}
\begin{eqnarray}
\label{eqn:change tilde x2}
\lefteqn{
\Gamma^\tau_{-a-\tilde{t}}(y)=\Gamma^\tau_{-\tilde{t}}(\tilde{x})+
\Big(\theta+\tau\,a+R_3(\theta,a), M_{\tilde{t}}\,\mathbf{v}+
\mathbf{R}_2(\theta,\mathbf{v},a)\Big)
}\\
&=&\Big( x_2+\big(R_1(\delta),\mathbf{R}_1(\delta)\big) \Big)+
\Big(\theta+\tau\,a+R_3(\theta,a), M_{\tilde{t}}\,\mathbf{v}+
\mathbf{R}_2(\theta,\mathbf{v},a)\Big)\nonumber\\
&=&x_2+\Big(\theta+\tau\,a+R_3(\theta,a)+R_1(\delta),
M_{\tilde{t}}\,\mathbf{v}+\mathbf{R}_1(\delta)+
\mathbf{R}_2(\theta,\mathbf{v},a)
\Big).
\nonumber
\end{eqnarray}

Therefore,
\begin{eqnarray}
\label{eqn:second term psi}
\lefteqn{
\imath\,\psi^\tau\left(\Gamma_{-a-\tilde{t}}(y),x_2\right)}\\
&=&\imath\,(\theta+\tau\,a)-\frac{1}{4\,\tau^2}\,(\theta+\tau\,a)^2
-\frac{1}{2}\,\|M_{\tilde{t}}\,\mathbf{v}\|^2+R_2(\delta,a).
\end{eqnarray}
Given that $|a|,\,|\theta|\le C\,(\epsilon+\delta)$, by (\ref{eqn:fasePsi'})
we conclude that
\begin{eqnarray}
\label{eqn:derivata rel a}
\partial _\theta\Psi&=&-u+v+O(\epsilon+\delta).
\end{eqnarray}

The variable $\theta$ is also compactly supported,
and $v$ ranges in a certain neighbourhood of $1/\tau$;
therefore, upon choosing $\epsilon$ and $\delta$
suitably small, we conclude that the contribution to the
asymptotics of the locus where $0< u \ll 1/\tau$ or
$u\gg 1/\tau$ is also negligible. In other words, the 
asymptotics of (\ref{eqn:5th expression}) are unaltered,
if the integrand is multiplied by a cut-off function of the
form $f_2(u)$, where $f_2\in \mathcal{C}^\infty_c(\mathbb{R}_+)$
is identically equal to $1$ on an appropriate neighbourhood of
$1/\tau$.

The proof of Proposition \ref{prop:compact u and v} is complete.

\end{proof}

By Corollary 1.3 in \cite{bs}, there exists a constant
$C^\tau>0$ such that for any $x',x''\in X^\tau$ 
$$
\left|\psi^\tau(x',x'')\right|\ge 
\Im \left( \psi^\tau(x',x'') \right)
\ge C^\tau\,\mathrm{dist}_{X^\tau}(x',x'')^2.
$$

Let $\rho_\lambda'=\rho_\lambda'(x_1;g,y)$
be a smooth function
identically equal to $1$ where
$\mathrm{dist}_{X^\tau}\left(\mu^\tau_{g^{-1}}(x_1),y\right)
< a^\tau\,\lambda^{\epsilon'-1/2}$
and vanishes where $\mathrm{dist}_{X^\tau}\left(\mu^\tau_{g^{-1}}(x_1),y\right)
> A^\tau\,\lambda^{\epsilon'-1/2}$, for certain constants
$A^\tau>a^\tau>0$. 
Similarly, let  
$\rho_\lambda''=\rho_\lambda''(x_2;t,y)$ 
be a smooth function
identically equal to $1$ where
$\mathrm{dist}_{X^\tau}\left(\Gamma^\tau_{-t}(y,)x_2\right)
< b^\tau\,\lambda^{\epsilon'-1/2}$
and vanishes where $\mathrm{dist}_{X^\tau}\left(\Gamma^\tau_{-t}(y,)x_2\right)
> B^\tau\,\lambda^{\epsilon'-1/2}$, for certain constants
$B^\tau>b^\tau>0$.

\begin{lem}
\label{lem:g y t shrink}
The asymptotics of (\ref{eqn:5th expression}) are unaltered, if the
integrand is further multiplied by 
$\rho_\lambda'(x_1;g,y)\,\rho_\lambda''(x_2;t,y)$.

\end{lem}

\begin{proof}
On the support of
$1-\rho_\lambda'$, we can write 
$$
e^{\imath\,\lambda\,u\,\psi^\tau\left(\mu^\tau_{g^{-1}}(x_1),y\right)}=
-\frac{\imath}{\lambda}\,\psi^\tau\left(\mu^\tau_{g^{-1}}(x_1),y\right)^{-1}
\,\partial_u\left( e^{\imath\,\lambda\,u\,\psi^\tau\left(\mu^\tau_{g^{-1}}(x_1),y\right)}   \right).
$$
Furthermore, for some constant $D^\tau>0$
$$
\left|\frac{\imath}{\lambda}\,\psi^\tau\left(\mu^\tau_{g^{-1}}(x_1),y\right)^{-1}
\right|\le D^\tau
\lambda^{-2\,\epsilon'}.
$$
Hence, by iteratively integrating by parts in the compactly
supported variable $u$ we conclude that the contribution to
the asymptotics of the locus where 
$\rho_\lambda '\neq 1$ is $O\left(\lambda^{-\infty}\right)$.

The argument for $\rho''_\lambda$ is similar.
\end{proof}

Let us choose $r>1$ such that
$$
\mathrm{dist}_{X^\tau}(x',x'')
\le 
r\,
\mathrm{dist}_{X^\tau}\left(\Gamma^\tau_t(x'),\Gamma^\tau_t(x'')\right)
\quad \forall \,x',\,x''\in X^\tau,\quad -t\in \mathrm{supp}(\chi).
$$
On the support of the product $\chi (t)\cdot 
\rho_\lambda'(x_1;g,y)\,\rho_\lambda''(x_2;t,y)$, we have
\begin{eqnarray}
\label{eqn:estimate xgt dist}
\mathrm{dist}_{X^\tau}
\left(x_1,\Gamma^\tau_{t}\circ\mu^\tau_{g}(x_2)\right)&=&\mathrm{dist}_{X^\tau}
\left(\mu^\tau_{g^{-1}}(x_1),\Gamma^\tau_{t}(x_2)\right)
\\
&\le&\mathrm{dist}_{X^\tau}
\left(\mu^\tau_{g^{-1}}(x_1),y\right)+\mathrm{dist}_{X^\tau}
\left(y,\Gamma^\tau_{t}(x_2)\right)\nonumber\\
&\le&r\,\left[  \mathrm{dist}_{X^\tau}
\left(\mu^\tau_{g^{-1}}(x_1),y\right) +
\mathrm{dist}_{X^\tau}
\left(\Gamma^\tau_{-t}(y),x_2\right)  \right]\nonumber\\
&=&\lambda^{\epsilon'-1/2}\cdot r\,\left(A^\tau+B^\tau\right).
\nonumber
\end{eqnarray}

Given any $C>0$ and $\epsilon'>0$, on the other hand, 
for $0<\epsilon''<\epsilon'$ and $\lambda\gg 0$ one has
$$
\lambda^{\epsilon''-1/2}\cdot r\,\left(A^\tau+B^\tau\right)<
C\,\lambda^{\epsilon'-1/2}.
$$
This completes the proof of Part 1.
of Theorem \ref{thm:main 1}.

\end{proof}

\begin{proof}
[Proof of Theorem \ref{thm:main 1}, Part 2.]
By the arguments in the proof of Part 1., integration
in (\ref{eqn:5th expression})
may be restricted to the locus where 
$\mathrm{dist}_{X^\tau}
\left(\mu^\tau_{g^{-1}}(x_1),y\right)\le A\,\lambda^{\epsilon'-1/2}$.
Hence we may refine (\ref{eqn:psidiff diagonal}) to
\begin{eqnarray}
\label{eqn:psidiff diagonal1}
\mathrm{d}_{\big(\mu^\tau_{g^{-1}}(x_1),y\big)}\psi^\tau
&=&\left(\alpha^\tau_{\mu^\tau_{g^{-1}}(x_1)},
-\alpha^\tau_{\mu^\tau_{g^{-1}}(x_1)}\right)+O\left(\lambda^{\epsilon'-\frac{1}{2}} \right).
\end{eqnarray}
Working locally near a given $g_0\in G$, we can write
$g=g_0\,e^\xi$, where $\xi\in \mathfrak{g}$.
In view of Corollary \ref{cor: the moment map on tubes} and (\ref{eqn:fasePsi'}), 
we conclude that on the domain of integration
$$
\partial_\xi\Psi=u\,\langle\Phi(x_1),\xi\rangle+O\left(\lambda^{\epsilon'-\frac{1}{2}} \right)
$$
(recall Notation \ref{notn:moment map}).
In other words,
$$
\partial_g\Psi=u\,\Phi(x_1)+O\left(\lambda^{\epsilon'-\frac{1}{2}} \right).
$$
Since $u$ is now bounded away from zero, we conclude by
iteratively integrating by parts in $g$ that the contribution to
the asymptotics of
the locus where 
$\|\Phi(x_1)\|\ge C\,\lambda^{\epsilon'-\frac{1}{2}}$
is negligible. Given that by assumption 
$0$ is a regular value of $\Phi$,
we conclude that
$$
\Pi^\tau_{\chi,\nu,\lambda}(x_1,x_2)
=O\left(\lambda^{-\infty}\right),
$$
uniformly for $\mathrm{dist}_{X^\tau}(x_1,Z^\tau)\ge 
C\,\lambda^{\epsilon'-1/2}$. The similar statement for
$x_2$ in the following way: set $\chi_{-}(t):= \chi(-t)$, so that $\overline{\hat{\chi}}=\widehat{\chi_-}$, then by definition we have:
\begin{eqnarray}\label{eqn:conjugate swap}
\overline{\Pi^\tau_{\chi,\nu,\lambda}(x_{1},x_{2})}&=&
\sum_j\overline{\hat{\chi}(\lambda-\lambda_j)}\,\sum_k
\rho^{(\nu)}_{j,k}(x_{2})\cdot \overline{\rho^{(\nu)}_{j,k}(x_{1})}\nonumber\\
&=&\sum_j\widehat{\chi_-}(\lambda-\lambda_j)\,
\sum_k
\rho^{(\nu)}_{j,k}(x_{2})\cdot \overline{\rho^{(\nu)}_{j,k}(x_{1})}\nonumber\\
&=&\Pi^\tau_{\chi_-,\nu,\lambda}(x_{2},x_{1}).
\end{eqnarray}
Applying a similar argument to $\Pi^\tau_{\chi_{-},\nu,\lambda}(x_2,x_1)$ yields the same conclusion for $x_2$.
\end{proof}

\section{Scaling asymptotics for $\Pi^\tau_{\chi,\nu,\lambda}$}

In this section, we shall prove Theorems 
\ref{thm:main 3} (for $\Pi^\tau_{\chi,\nu,\lambda}$), 
\ref{thm:main 2}, and \ref{thm:diagonal case}; the adaptations required
for $P_{\chi,\nu,\lambda}$ will be dealt with in 
\S \ref{sctn:complexified eigenf lapl}.
Before discussing the proofs, let us interject some recalls
and remarks.

If $x_1\in Z^\tau$ and $x_1\in x_2^{G\times \mathrm{supp}(\chi)}$, 
then $x_2\in Z^\tau$ and
the action of $G\times \mathbb{R}$ is locally free at
both $x_1$ and $x_2$; furthermore, 
$x_1^G\cap x_2^{\mathrm{supp}(\chi)}=\{x_{12}\}$ for a unique point $x_{12}$ (Corollary \ref{cor:unique t1chi}),
and $\Sigma_\chi(x_1,x_2)$ is as in (\ref{eqn:Sigmax12chi}).
Let us fix NHLC's at $x_1$ and $x_2$.
Given $h_1$ as in (\ref{eqn:Sigmax12chi}), we obtain from
the NHLC's at $x_1$ a system of NHLC's at $x_{12}=
\mu^\tau_{h_1^{-1}}(x_1)$, simply by composing with 
$\mu^\tau_{h_1}$.

We shall first set the common stage for the
proofs, and then specialize the argument for each Theorem separately.

\begin{rem}
\label{rem:omega0hlcs}
Having chosen (normal) Heisenberg local coordinates
at $x_1$ and $x_2$, in the
following computations we may replace $\omega_{x_j}$ (restricted
to $\mathcal{H}^\tau_{x_j}$) by the standard symplectic form 
$\omega_0$ on $\mathbb{R}^{2d-2}$ (see
(63) of \cite{p24}).
\end{rem}

\subsection{General arguments for the scaling asymptotics
of $\Pi^\tau_{\chi,\nu,\lambda}$}

With the aim to study the asymptotics of 
$\Pi^\tau_{\chi,\nu,\lambda}(x_{1,\lambda},x_{2,\lambda})$,
with $x_{j,\lambda}$ as in (\ref{eqn:rescaled coord 12}),
we start from 
(\ref{eqn:5th expression}), with $x_{j,\lambda}$ in place of
$x_j$. The following arguments will eventually depend on an
application of the Stationary Phase Lemma, and all the previously
introduced cut-offs are identically
equal to $1$ near the stationary point. With abuse of notation,
we shall 
occasionally implicitly 
absorb the cut-offs in the amplitude,
unless it is relevant for the argument to do otherwise.
Thus we may write
\begin{eqnarray}
\label{eqn:6th expression}
\lefteqn{\Pi^\tau_{\chi,\nu,\lambda}(x_{1,\lambda},x_{2,\lambda})}\\
&\sim&\lambda^2\,\frac{\dim(\nu)}{\sqrt{2\,\pi}}\,
\int_G\,\mathrm{d}V_G(g)\,\int_{-\infty}^{+\infty}\,\mathrm{d}t
\,\int_{X^\tau}\,\mathrm{d}V_{X^\tau}(y)\,
\int_0^{+\infty}\,\mathrm{d}u\,
\int_0^{+\infty}\,\mathrm{d}v\nonumber\\
&&\left[e^{\imath\,\lambda\,\Psi(x_{1,\lambda},x_{2,\lambda};g,t,y,u,v)}\,
\Xi_\nu 
\left( g^{-1} 
\right)\,
\chi (t)\,  
%
s^\tau\left(\mu^\tau_{g^{-1}}(x_{1,\lambda}),y,\lambda\,u\right)\,
r^\tau_t\left(y,x_{2,\lambda},\lambda\,v\right)
\right].\nonumber
\end{eqnarray}
where  
\begin{equation}
\label{eqn:fasePsi'lambda}
\Psi(x_{1,\lambda},x_{2,\lambda};g,t,y,u,v)=
u\,\psi^\tau\left(\mu^\tau_{g^{-1}}(x_{1,\lambda}),y\right)
+v\,\psi^\tau
\left(\Gamma^\tau_{-t}(y),x_{2,\lambda}\right)-t.
\end{equation}
By the previous reductions, integration has been reduced to a 
locus where 
$$
\max\left\{\mathrm{dist}_{X^\tau}
\left( \mu^\tau_{g^{-1}}(x_{1,\lambda}),y\right),
\mathrm{dist}_{X^\tau}
\left( \Gamma^\tau_{-t}(y),x_{2,\lambda} \right)\right\}
=O\left(\lambda^{\epsilon'-\frac{1}{2}}\right);
$$
hence, we also have
$$
\mathrm{dist}_{X^\tau}
\left( \mu^\tau_{g^{-1}}(x_{1}),
\Gamma^\tau_t(x_{2}) \right)
=O\left(\lambda^{\epsilon'-\frac{1}{2}}\right).
$$
Since the action of $G\times \mathbb{R}$ on $Z^\tau$ 
is locally free, this entails that (once the previous cut-offs have been taken into account) $(g,t)$ ranges in a neighbourhood
of radius $O\left(\lambda^{\epsilon'-1/2}\right)$ of the set
$\Sigma_\chi(x_1,x_2)$ in (\ref{eqn:Sigmax12})
and (\ref{eqn:Sigmax12chi}). We can reformulate this 
slightly more explicitly as follows.

Let 
$\gamma^{\mathfrak{g}}:
\mathfrak{g}\rightarrow \mathbb{R}$ 
denote a bump function
supported in a small neighbourhood of $0\in \mathfrak{g}$,
and identically equal to $1$ sufficiently
close to $0$. 
Similarly, let 
$\gamma^\mathbb{R}:
\mathbb{R}\rightarrow \mathbb{R}$ denote a compactly 
supported bump function compactly supported on a small
neighbourhood of $0\in \mathbb{R}$, and identically equal
to $1$ sufficiently close to $0$. 
We conclude that the asymptotics of 
$\Pi^\tau_{\chi,\nu,\lambda}(x_{1,\lambda},x_{2,\lambda})$
are unaltered, if the integrand in (\ref{eqn:6th expression})
is further multiplied by the expression 
\begin{equation}
\label{eqn:new cut off RG}
\gamma^\mathbb{R}\left(\lambda^{-\epsilon'+1/2}\,
(t-t_1)\right)\cdot \sum_{l=1}^r\gamma^{\mathfrak{g}}\left(
\lambda^{-\epsilon'+1/2}\,\log_G\left(
g\,h_1^{-1}\,\kappa_l^{-1}\right)\right),
\end{equation}
where $\log_G=\exp_G^{-1}$ is defined on some neighbourhood
of $e_G$.
The cut-off (\ref{eqn:new cut off RG}) plays the same role as
the product $\chi (t)\cdot 
\rho_\lambda'(x_1;g,y)\,\rho_\lambda''(x_2;t,y)$ preceding
(\ref{eqn:estimate xgt dist}), and may replace it in
the integrand.


If we multiply the integrand in (\ref{eqn:6th expression})
by the cut-off in (\ref{eqn:new cut off RG}), we obtain
\begin{equation}
\label{eqn:Pitau sum over l}
\Pi^\tau_{\chi,\nu,\lambda}(x_{1,\lambda},x_{2,\lambda})
\sim \sum_{l=1}^r\Pi^\tau_{\chi,\nu,\lambda}(x_{1,\lambda},x_{2,\lambda})_l
\end{equation}
where $\Pi^\tau_{\chi,\nu,\lambda}(x_{1,\lambda},x_{2,\lambda})_l$
is given by right hand side of
(\ref{eqn:6th expression}), with the integrand
multiplied by the $l$-th summand in (\ref{eqn:new cut off RG}).

We are thus reduced to computing the asymptotics of each
summand in (\ref{eqn:Pitau sum over l}). To this end,
let us make the change of variables
$$
t\mapsto t_1+t,\,g=e^\xi\,\kappa_l\,h_1,
$$
where $(t,\xi)$ now ranges in a neighbourhood
of the origin in $\mathbb{R}\times \mathfrak{g}$.
The $l$-th summand in (\ref{eqn:new cut off RG}) now takes the 
form 
\begin{equation}
\label{eqn:new cut off RG lth term}
\gamma^{\mathbb{R}}\left(\lambda^{-\epsilon'+1/2}\,
t\right)\cdot \gamma^{\mathfrak{g}}\left(
\lambda^{-\epsilon'+1/2}\,\xi\right).
\end{equation}

On the current domain of integration we have
\begin{eqnarray*}
\mathrm{dist}_{X^\tau}\left(\mu^\tau_{g^{-1}}(x_1),x_{12}\right)
&=&\mathrm{dist}_{X^\tau}\left(\mu^\tau_{h_1^{-1}}\circ\mu^\tau_{\kappa_l^{-1}}\circ\mu^\tau_{e^{-\xi}}(x_1),
\mu^\tau_{h_1^{-1}}(x_1)\right)\\
&=&
O\left(\lambda^{\epsilon'-\frac{1}{2}}\right),
\end{eqnarray*}
and $\mathrm{dist}_{X^\tau}\left(\mu^\tau_{g^{-1}}(x_1),y\right)
=O\left(\lambda^{\epsilon'-\frac{1}{2}}\right)$, whence also
$\mathrm{dist}_{X^\tau}\left(y,x_{12}\right)
=O\left(\lambda^{\epsilon'-\frac{1}{2}}\right)$.
We can then express $y$ in NHLC's at $x_{12}$ as 
$$
y=x_{12}+\left(\theta,\mathbf{u}\right),
$$
where $\|(\theta,\mathbf{u})\|\le C'\,\lambda^{\epsilon'-\frac{1}{2}}$ for some constant $C'>0$.
By (67) and (117) of 
\cite{p24},
\begin{equation}
\label{eqn:volume term HLCs}
\mathrm{d}V_{X^\tau}(y)=\mathcal{V}(\theta,\mathbf{u})\,
\mathrm{d}\theta\,\mathrm{d}\mathbf{u},
\quad 
\mathcal{V}(0,\mathbf{0})=\frac{2^{d-1}}{\tau}.
\end{equation}

Let us operate the rescaling
\begin{equation}
\label{eqn:rescaling all}
(\theta,\,\mathbf{u},\,\xi,\,t)\mapsto
\left(\frac{\theta}{\sqrt{\lambda}},
\frac{\mathbf{u}}{\sqrt{\lambda}},
\,\frac{\xi}{\sqrt{\lambda}},\,\frac{t}{\sqrt{\lambda}}\right);
\end{equation}
and accordingly rewrite the parametrization of $y$ as
\begin{equation}
\label{eqn:y rescaled}
y_\lambda (\theta,\mathbf{u}):=x_{12}
+\left( \frac{\theta}{\sqrt{\lambda}}, \frac{\mathbf{u}}{\sqrt{\lambda}} \right).
\end{equation}
In the rescaled variables, the cut-off 
(\ref{eqn:new cut off RG lth term}) takes the form
\begin{equation}
\label{eqn:new cut off RG lth term resc}
\gamma^\mathbb{R}\left(\lambda^{-\epsilon'}\,
t\right)\cdot \gamma^{\mathfrak{g}}\left(
\lambda^{-\epsilon'}\,\xi\right),
\end{equation}
so that integration in $(\xi,t)$ is now over a ball centered 
at the origin and radius $O\left(\lambda^{\epsilon'}\right)$
in $\mathfrak{g}\times \mathbb{R}$.
Similarly, since prior to rescaling 
$(\theta,\mathbf{v})$ range in a shrinking ball centered at
the origin of radius $O\left(\lambda^{\epsilon'-1/2}\right)$,
the rescaled variables will range in an expanding ball 
of radius $O\left(\lambda^{\epsilon'}\right)$ in
$\mathbb{R}\times \mathbb{R}^{2d-2}$.

Let $\mathrm{d}\xi$ denote the Lebesgue measure on
$\mathfrak{g}$ associated to a bi-invariant metric on
$G$ whose Riemannian density is the Haar measure.
We obtain
\begin{eqnarray}
\label{eqn:6th expression lth summand}
\lefteqn{\Pi^\tau_{\chi,\nu,\lambda}(x_{1,\lambda},x_{2,\lambda})_l}\\
&=&\lambda^{2-d-d_G/2}\,\frac{\dim(\nu)}{\sqrt{2\,\pi}}\,
\int_{\mathfrak{g}}\,\mathrm{d}\xi\,
\int_{-\infty}^{+\infty}\,\mathrm{d}t
\,\int_{-\infty}^{+\infty}\,\mathrm{d}\theta\,
\int_{\mathbb{R}^{2d-2}}\,\mathrm{d}\mathbf{u}\,
\int_0^{+\infty}\,\mathrm{d}u\,
\int_0^{+\infty}\,\mathrm{d}v\nonumber\\
&&\left[e^{\imath\,\lambda\,\Psi_\lambda(x_j,
\theta_1,\mathbf{v}_1,\theta_2,\mathbf{v}_2,
t,\xi,\theta,\mathbf{v},u,v)_l}\,
\mathcal{A}_\lambda(
x_j,\theta_1,\mathbf{v}_1,\theta_2,\mathbf{v}_2,
t,\xi,\theta,\mathbf{u},u,v)_l
\right],\nonumber
\end{eqnarray}
where, recalling (\ref{eqn:fasePsi'}),
\begin{eqnarray}
\label{eqn:Psilambda}
\lefteqn{\Psi_\lambda(x_j,
\theta_1,\mathbf{v}_1,\theta_2,\mathbf{v}_2,
t,\xi,\theta,\mathbf{u},u,v)_l   }\\
&:=&\Psi\left(x_{1,\lambda},x_{2,\lambda};e^{\xi/\sqrt{\lambda}}\,\kappa_l\,h_1,t_1+\frac{t}{\sqrt{\lambda}},
y_\lambda (\theta,\mathbf{u}),u,v\right),
\nonumber
\end{eqnarray}
\begin{eqnarray}
\label{eqn:Alambda}
\lefteqn{\mathcal{A}_\lambda(
x_j,
\theta_1,\mathbf{v}_1,\theta_2,\mathbf{v}_2,
t,\xi,\theta,\mathbf{u},u,v) _l  }\\
&:=&
\Xi_\nu 
\left( h_1^{-1}\,\kappa_l^{-1}\, e^{-\xi/\sqrt{\lambda}}
\right)\,
\chi \left(t_1+\frac{t}{\sqrt{\lambda}}\right)
\cdot \mathcal{V}\left(\frac{\theta}{\sqrt{\lambda}},
\frac{\mathbf{u}}{\sqrt{\lambda}}\right)
\nonumber\\
&&\cdot 
s^\tau\left(\mu^\tau_{(e^{\xi/\sqrt{\lambda}}\kappa_l h_1)^{-1}}(x_{1,\lambda}),y_\lambda (\theta,\mathbf{u}),\lambda\,u\right)\,
r^\tau_t\big(y_\lambda (\theta,\mathbf{u}),x_{2,\lambda},\lambda\,v\big)
\nonumber\\
&&\cdot \tilde{\gamma}_\lambda (t,\theta,\xi,\mathbf{u})\cdot 
f_1(v)\cdot f_2(u);
\nonumber
\end{eqnarray}
here $\mathcal{V}$ is as in (\ref{eqn:volume term HLCs}),
and we have collected in $\tilde{\gamma}$
the cut-offs in the rescaled variables (which may be assumed to be
all of the form (\ref{eqn:new cut off RG lth term resc})); finally,
$f_1$ and $f_2$ are as in Proposition
\ref{prop:compact u and v}.

The next step will be to expand (\ref{eqn:Psilambda}) in
descending powers of $\lambda$.
Recalling (\ref{eqn:fasePsi'lambda}), we shall first expand the
individual summands. In view of (\ref{eqn:local phase at x12}),
\begin{eqnarray}
\label{eqn:1st summand Psi expd}
\lefteqn{ \imath\,u\,\psi^\tau\left(\mu^\tau_{(e^{\xi/\sqrt{\lambda}}\kappa_l h_1)^{-1}}(x_{1,\lambda}),y_\lambda (\theta,\mathbf{u})\right)  }\\
&=&\frac{\imath}{\sqrt{\lambda}}\,u\,(\theta_1-\theta
)\nonumber\\
&&
+\frac{1}{\lambda}\,u\,\left[ \imath\,\omega_{x_1}\big(\xi_{X^\tau}(x_1),\mathbf{v}_1\big) -\frac{1}{4\,\tau^2}\,(\theta_1-\theta
)^2  +\psi_2\left( \mathbf{v}^{(l)}_1
-\xi_{X^\tau}(x_1)^{(l)},\mathbf{u}   \right)                                           \right]\nonumber\\
&&+u\,R_3\left(
\frac{\theta_1}{\sqrt{\lambda}},\frac{\mathbf{v}_1}{\sqrt{\lambda}},
\frac{\theta}{\sqrt{\lambda}},\frac{\mathbf{u}}{\sqrt{\lambda}},
\frac{t}{\sqrt{\lambda}},\frac{\xi}{\sqrt{\lambda}}   \right); \nonumber
\end{eqnarray}
in the following, we shall replace 
$\omega_{x_j}$ by $\omega_0$ (Remark \ref{rem:omega0hlcs}), 
and abridge $k$-th order 
remainder terms such as the 
one on the latter line of (\ref{eqn:1st summand Psi expd}) by
the short-hand $R_k(\bullet/\sqrt{\lambda})$.

Regarding the second summand in (\ref{eqn:fasePsi'lambda}), in view of (\ref{eqn:psi at nearx2x2}) we have
\begin{eqnarray}
\label{eqn:2nd summand Psi expd}
\lefteqn{
\imath\,v\,\psi\left( \Gamma^\tau_{-t_1-t/\sqrt{\lambda}}\left(x_{12}+\left(\frac{\theta}{\sqrt{\lambda}},\frac{\mathbf{u}}{\sqrt{\lambda}}\right)\right),
x_2+\left(\frac{\theta_2}{\sqrt{\lambda}},
\frac{\mathbf{v}_2}{\sqrt{\lambda}}\right)
\right)}\\
&=& \frac{\imath}{\sqrt{\lambda}}\,v\,(  \theta+\tau\,t-\theta_2)\nonumber\\
&&+\frac{v}{\lambda}\,\left[
-\frac{1}{4\,\tau^2}   \,(  \theta+\tau\,t-\theta_2)^2
+\psi_2\big(B\mathbf{u},\mathbf{v}_2\big)\right]
+R_3\left(\frac{\bullet}{\sqrt{\lambda}}\right).
 \nonumber
\end{eqnarray}

The upshot is an expansion for (\ref{eqn:Psilambda}) of the form
\begin{eqnarray}
\label{eqn:Psi expanded}
\lefteqn{\imath\,\lambda\, \Psi_\lambda(
\theta_1,\mathbf{v}_1,\theta_2,\mathbf{v}_2,
t,\xi,\theta,\mathbf{u},u,v)_l }\\
&=&-\imath\,\lambda\,t_1 +\imath\,\sqrt{\lambda}\,
\Psi^\tau_{\theta_1,\theta_2}(t,v,\theta,u)
+\mathcal{S}(x,\theta_1,\mathbf{v}_1,\theta_2,\mathbf{v}_2,t,v,\theta,u,
\mathbf{u},\xi)_l\nonumber\\
&&+\lambda\,R_3\left(\frac{\bullet}{\sqrt{\lambda}}\right),
 \nonumber
\end{eqnarray}
where:
\begin{equation}
\label{eqn:Psitheta12}
\Psi^\tau_{\theta_1,\theta_2}(t,v,\theta,u):=
u\,(\theta_1-\theta
)+v\,(  \theta+\tau\,t-\theta_2)  -t  ,
\end{equation}
\begin{eqnarray}
\label{eqn:Stheta12v12}
\lefteqn{\mathcal{S}(\theta_1,\mathbf{v}_1,\theta_1,\mathbf{v}_1,t,v,\theta,u,
\mathbf{u},\xi)_l}\\
&:=&u\,\left[ \imath\,\omega_0\big(\xi_{X^\tau}(x_1),\mathbf{v}_1\big) -\frac{1}{4\,\tau^2}\,(\theta_1-\theta
)^2  +\psi_2\left( \mathbf{v}^{(l)}_1
-\xi_{X^\tau}(x_1)^{(l)},\mathbf{u}   \right)                                           \right]\nonumber\\
&&+v\,\left[
-\frac{1}{4\,\tau^2}   \,(  \theta+\tau\,t-\theta_2)^2
+\psi_2\big(B\mathbf{u},\mathbf{v}_2\big)\right].
\nonumber
\end{eqnarray}

Let us set
\begin{eqnarray}
\label{eqn:amplitude Blambda}
\lefteqn{
\mathcal{B}_\lambda(x_j, \theta_1,\mathbf{v}_1,\theta_2,\mathbf{v}_2,t,v,\theta,u,
\mathbf{u},\xi)_l}\\
&:=&
e^{\mathcal{S}(\theta_1,\mathbf{v}_1,\theta_2,\mathbf{v}_2;t,v,\theta,u,
\mathbf{u},\xi)_l}\,
e^{\lambda\,R_3\left(\frac{\bullet}{\sqrt{\lambda}}\right)}
\cdot\mathcal{A}_\lambda(
\theta_1,\mathbf{v}_1,\theta_2,\mathbf{v}_2,
t,\xi,\theta,\mathbf{u},u,v)_l .\nonumber
\end{eqnarray}

Then (\ref{eqn:6th expression lth summand})
may be rewritten as
\begin{eqnarray}
\label{eqn:7th expression lth summand}
\lefteqn{\Pi^\tau_{\chi,\nu,\lambda}(x_{1,\lambda},x_{2,\lambda})_l}\\
&=&e^{-\imath\,\lambda\,t_1}\cdot\lambda^{2-d-d_G/2}\,\frac{\dim(\nu)}{\sqrt{2\,\pi}}
\cdot \int_{\mathfrak{g}}\,\mathrm{d}\xi\,
\int_{\mathbb{R}^{2d-2}}\,\mathrm{d}\mathbf{u}\,\big[
I_\lambda (x_j,\theta_1,\mathbf{v}_1,\theta_2,\mathbf{v}_2,\xi,\mathbf{u})_l\big],
\nonumber
\end{eqnarray}
where
\begin{eqnarray}
\label{eqn:defn di Ilambda}
\lefteqn{
I_\lambda (x_j,\theta_1,\mathbf{v}_1,\theta_2,\mathbf{v}_2,\xi,\mathbf{u})_l}
\\
&:=&
\int_{-\infty}^{+\infty}\,\mathrm{d}t
\,\int_{-\infty}^{+\infty}\,\mathrm{d}\theta\,
\int_0^{+\infty}\,\mathrm{d}u\,
\int_0^{+\infty}\,\mathrm{d}v\nonumber\\
&&\left[e^{\imath\,\sqrt{\lambda}\,
\Psi^\tau_{\theta_1,\theta_2}(t,v,\theta,u)}\cdot 
\mathcal{B}_\lambda
(x_j,\theta_1,\mathbf{v}_1,\theta_2,\mathbf{v}_2,t,v,\theta,u,
\mathbf{u},\xi)_l
\right].\nonumber
\end{eqnarray}

We can now pair Taylor expansion 
in the rescaled variables in (\ref{eqn:Alambda}) and
in the factor $e^{\lambda\,R_3\left(\bullet/\sqrt{\lambda}\right)}$
with the asymptotic expansions of the classical symbols $s^\tau$ and $r^\tau_t$. 
The same arguments leading to the asymptotic expansion (142) 
of \cite{p24} in the action-free case yields a similar expansion 
in the present setting.
Before stating it, let us make the following remarks regarding
the leading order terms in $s^\tau$ and $r^\tau_t$ in NHLC's.

First, by Theorem \ref{thm: phase expression in HLC}, in NHLC's at $x_{12}$
we have
\begin{eqnarray}
\label{eqn:leading stau0}
s^\tau_0\left(\mu^\tau_{(\kappa_l h_1)^{-1}}(x_{1}),
x_{12}\right)=s^\tau_0\left(x_{12},
x_{12}\right)=\frac{\tau}{(2\,\pi)^d}.
\end{eqnarray}
Second, recalling (\ref{eqn:lead r tau t 0}),
\begin{eqnarray}
\label{eqn:lead r tau t 0 1}
r_{t_1,0}^\tau(x_{12},x_{2}\big)&=&
\sigma_{t_1,0}^\tau (x_{12})\cdot 
s^\tau_0\left(\Gamma^\tau_{-t_1}(x_{12}),x_2\right)
\nonumber\\
&=&\sigma_{t_1,0}^\tau (x_{1})\cdot 
s^\tau_0\left(x_2,x_2\right)=\sigma_{t_1,0}^\tau (x_{1})\cdot 
\frac{\tau}{(2\,\pi)^d},
\end{eqnarray}
where the front factor is as in (\ref{eqn:correcting factor tau t}).
We then have the following.
\begin{lem}
\label{lem:asympt-exp-B-lambda}
As $\lambda\rightarrow +\infty$, there is an asymptotic
expansion 
\begin{eqnarray*}
\lefteqn{\mathcal{B}_\lambda(x_j,\theta_1,\mathbf{v}_1,\theta_2,\mathbf{v}_2,t,v,\theta,u,
\mathbf{u},\xi)_l}\\
&\sim&e^{\mathcal{S}(\theta_1,\mathbf{v}_1,\theta_2,\mathbf{v}_2,t,v,\theta,u,
\mathbf{u},\xi)_l}\cdot 
\chi (t_1)\cdot \overline{\Xi_\nu 
\left(\kappa_l\, h_1\right)}\cdot\frac{2^{d-1}}{\tau}\cdot 
\lambda^{2\,(d-1)}\cdot 
(u\,v)^{d-1}\cdot \frac{\tau^2}{(2\,\pi)^{2d}}\\
&&\cdot\sigma_{t_1,0}^\tau (x_{1})\cdot \beta \left(\lambda^{-\epsilon'}\,(t,\xi,\theta,\mathbf{u})\right)
\cdot f_1(u)\cdot f_2(v)\\
&&\cdot \left[   
1+\sum_{k\ge 1}\lambda^{-k/2}\,B_k(u,v,\theta_1,\mathbf{v}_1,\theta_2,\mathbf{v}_2,t,v,\theta,u,
\mathbf{u},\xi)_l\right],
\end{eqnarray*}
where $B_k(u,v,\cdot)$ is a polynomial in the rescaled variables,
of degree $\le 3\,k$ and parity $k$.
\end{lem}

The latter is indeed an asymptotics expansion for $\epsilon'\in 
(0,1/6)$.

Integration in $u$ and $v$ is compactly supported. Furthermore,
in view of (\ref{eqn:Psitheta12})
$$
\partial _u\Psi^\tau_{\theta_1,\theta_2}(t,v,\theta,u)=
\theta_1-\theta,\quad 
\partial _v\Psi^\tau_{\theta_1,\theta_2}(t,v,\theta,u)=
  \theta+\tau\,t-\theta_2.
$$
Hence the partial differential 
$\partial_{u,v}\Psi^\tau_{\theta_1,\theta_2}$ satisfies 
\begin{eqnarray*}
\left\| \partial_{u,v}\Psi^\tau_{\theta_1,\theta_2}(t,v,\theta,u)  \right\|
&\ge&C_\tau\,
\left\|\left(\theta-\theta_1,
t-\frac{\theta_2-\theta_1}{\tau}\right)\right\|.
\end{eqnarray*}
For any given 
$\delta>0$,
by iteratively integrating by parts in $(u,v)$ on
the locus where $\left\|\big(\theta-\theta_1,
(\theta_2-\theta_1)/\tau\big)\right\|\ge \delta$, one introduces at
each step a factor $O\left(\lambda^{-1/2}\right)$. 
On the other hand, the radius 
of domain of
integration in (\ref{eqn:defn di Ilambda}) 
grows like $\lambda^{\epsilon'}$; furthermore, once divided by
$\lambda^{2\,(d-1)}$ the amplitude
$\mathcal{B}_\lambda$
remains bounded 
on the latter domain by Lemma \ref{lem:asympt-exp-B-lambda} and 
(\ref{eqn:Stheta12v12}).
One then has the following.

\begin{lem}
\label{lem:compact theta t}
Only a rapidly decreasing contribution to the asymptotics of
(\ref{eqn:defn di Ilambda}) - and 
(\ref{eqn:7th expression lth summand}) - is lost, if integration 
in $(\theta,t)$ is restricted to a fixed and arbitrarily small
neighbourhood of $\big(\theta_1,
(\theta_2-\theta_1)/\tau\big)$.
\end{lem}

Leaving a corresponding cut-off function in $(\theta,t)$ implicit,
we may now study the asymptotics of $I_\lambda (x_j,\theta_1,\mathbf{v}_1,\theta_2,\mathbf{v}_2,\xi,\mathbf{u})_l$ 
in (\ref{eqn:defn di Ilambda})
using the stationary phase Lemma.
The phase (\ref{eqn:Psilambda}) has already been considered in \cite{p24} (where it is denoted $\Upsilon^\tau$).
By Lemma 64 of the same paper, 
(\ref{eqn:Psitheta12}) has a unique stationary point
$P_s=(t_s,v_s,\theta_s,u_s)$, given by
\begin{equation}
\label{eqn:stationary Ps}
P_s=\left(
\frac{\theta_2-\theta_1}{\tau},\frac{1}{\tau},\theta_1,\frac{1}{\tau}
\right);
\end{equation}
furthermore, the Hessian determinant and signature at the
critical point are, respectively, $\tau^2$ and $0$.
Arguing as in \S 4 of 
\cite{p24} (in particular, as in the derivation 
of (147) in \textit{loc. cit.})
we then obtain an asymptotic expansion for 
$I_\lambda (x_j,\theta_1,\mathbf{v}_1,\theta_2,\mathbf{v}_2,\xi,\mathbf{u})_l$ in (\ref{eqn:defn di Ilambda})
of the following form:
\begin{eqnarray}
\label{eqn:asyIlambda}
\lefteqn{
I_\lambda (x_j,\theta_1,\mathbf{v}_1,\theta_2,\mathbf{v}_2,\xi,\mathbf{u})_l}\\
&\sim&
\left(\frac{2\,\pi}{\sqrt{\lambda}}  \right)^2\cdot \frac{1}{\tau}
\cdot e^{\imath\,\sqrt{\lambda}\,\frac{\theta_1-\theta_2}{\tau}}
\cdot 
e^{\frac{1}{\tau}\,\left[\imath\,\omega_{0}\big(\xi_{X^\tau}(x_1),\mathbf{v}_1\big)+\psi_2\left( \mathbf{v}^{(l)}_1
-\xi_{X^\tau}(x_1)^{(l)},\mathbf{u}   \right)                                          +\psi_2\big(B\mathbf{u},\mathbf{v}_2\big)  \right]}
\nonumber\\
&&\cdot \sigma_{t_1,0}^\tau (x_{1})\cdot
\chi (t_1)\cdot \overline{\Xi_\nu 
\left(\kappa_l\, h_1\right)}\cdot\frac{2^{d-1}}{\tau}\cdot 
\left(\frac{\lambda}{\tau}\right)^{2\,(d-1)}
\cdot \frac{\tau^2}{(2\,\pi)^{2d}}\nonumber\\
&&\cdot \beta' \left(\lambda^{-\epsilon'}\,(\xi,\mathbf{u})\right)
\cdot \left[   
1+\sum_{k\ge 1}\lambda^{-k/2}\,F_k^\tau(\theta_1,\mathbf{v}_1,
\theta_2,\mathbf{v}_2;
\mathbf{u},\xi)_l\right]\nonumber\\
&=&e^{\imath\,\sqrt{\lambda}\,\frac{\theta_1-\theta_2}{\tau}}
\cdot \frac{\lambda^{2d-3}}{(2\,\pi^2\,\tau^2)^{d-1}}
\cdot \sigma_{t_1,0}^\tau (x_{1})\cdot
\chi (t_1)\cdot \overline{\Xi_\nu 
\left(\kappa_l\, h_1\right)}\cdot 
e^{\mathcal{S}(\theta_1,\mathbf{v}_1,\theta_2,\mathbf{v}_2;
\mathbf{u},\xi)_l}\nonumber\\
&&\cdot \beta' \left(\lambda^{-\epsilon'}\,(\xi,\mathbf{u})\right)
\cdot \left[   
1+\sum_{k\ge 1}\lambda^{-k/2}\,F_k(\theta_1,\mathbf{v}_1,\theta_2,\mathbf{v}_2;
\mathbf{u},\xi)_l\right]\nonumber
\end{eqnarray}
where
\begin{eqnarray}
\label{eqn:defn di Sc}
\lefteqn{ \mathcal{S}(\theta_1,\mathbf{v}_1,\theta_2,\mathbf{v}_2;
\mathbf{u},\xi)_l }\\
&:=&\frac{1}{\tau}\,\left[\imath\,\omega_{0}\big(\xi_{X^\tau}(x_1),\mathbf{v}_1\big)+\psi_2\left( \mathbf{v}^{(l)}_1
-\xi_{X^\tau}(x_1)^{(l)},\mathbf{u}   \right)                                          +\psi_2\big(B\mathbf{u},\mathbf{v}_2\big)  \right],
\nonumber
\end{eqnarray}
$\beta'$ is a cut-off identically equal to $1$ near the origin, and 
$F_k(\cdot)_l$ is a polynomial in the indicated variables,
of degree $\le 3\,k$ and parity $k$ 
(and implicitly depending on $x_1,\,x_2$).

More precisely, let us set
$$
L:=\frac{1}{\tau}\,\left(\frac{\partial^2}{\partial t\,\partial\,u}    
+\frac{\partial^2}{\partial t\,\partial\,v} 
\right)-\frac{\partial^2}{\partial \theta\,\partial\,u}.
$$
Then,
when applying 
the Stationary Phase Lemma to (\ref{eqn:defn di Ilambda}),
the $k$-th summand in the asymptotic expansion of the amplitude in 
Lemma \ref{lem:asympt-exp-B-lambda} will yield an asymptotic expansion
whose $j$-th term (where $j\ge 0$) is a multiple of
\begin{eqnarray*}
\lefteqn{ \frac{4\,\pi^2}{\tau}\,\sigma_{t_1,0}^\tau (x_{1})\cdot\chi (t_1)\cdot \overline{\Xi_\nu 
\left(\kappa_l\, h_1\right)}\cdot\frac{2^{d-1}}{\tau}\cdot 
\lambda^{2d-3-\frac{k+j}{2}} \cdot \frac{\tau^2}{(2\,\pi)^{2d}} }\\
&&\left.L^j\left(e^{\mathcal{S}(\theta_1,\mathbf{v}_1,\theta_2,\mathbf{v}_2,t,v,\theta,u,
\mathbf{u},\xi)_l}\cdot 
(u\,v)^{d-1}\cdot B_k(u,v,\theta_1,\mathbf{v}_1,\theta_2,\mathbf{v}_2,t,v,\theta,u,
\mathbf{u},\xi)_l\right)\right|_{P_s}.
\end{eqnarray*}
Using that $\mathcal{S}$ is homogeneous of degree two in the rescaled variables, one can check
inductively that the resulting expansion has the stated form.

Again in view of the cut-offs and the exponential, the
asymptotic expansion (\ref{eqn:asyIlambda}) 
may be integrated term by term. Thus
\begin{eqnarray}
\label{eqn:integrale Iuxi}
\lefteqn{ \int_{\mathfrak{g}}\,\mathrm{d}\xi\,
\int_{\mathbb{R}^{2d-2}}\,\mathrm{d}\mathbf{u}\,\big[
I_\lambda (\theta_1,\mathbf{v}_1,\theta_2,\mathbf{v}_2,\xi,\mathbf{u})_l\big]
 }\\
 &\sim&e^{\imath\,\sqrt{\lambda}\,\frac{\theta_1-\theta_2}{\tau}}
\cdot \frac{\lambda^{2d-3}}{(2\,\pi^2\,\tau^2)^{d-1}}
\cdot \sigma_{t_1,0}^\tau (x_{1})\cdot
\chi (t_1)\cdot \overline{\Xi_\nu 
\left(\kappa_l\, h_1\right)}
\nonumber\\
&&
\cdot \sum_{k=0}^{+\infty}\lambda^{-k/2}
\mathcal{I}_k(\theta_1,\mathbf{v}_1,\theta_2,\mathbf{v}_2)_l,
\nonumber
\end{eqnarray}
where
\begin{eqnarray}
\label{eqn:defn di mathcalIl}
\lefteqn{\mathcal{I}_k(\theta_1,\mathbf{v}_1,\theta_2,\mathbf{v}_2)_l}\\
&:=&\int_{\mathfrak{g}}\,\mathrm{d}\xi\,
\int_{\mathbb{R}^{2d-2}}\,\mathrm{d}\mathbf{u}\,\left[
e^{\mathcal{S}(\theta_1,\mathbf{v}_1,\theta_2,\mathbf{v}_2;
\mathbf{u},\xi)_l}
\cdot  F_k(\theta_1,\mathbf{v}_1,\theta_2,\mathbf{v}_2;
\mathbf{u},\xi)_l\right];\nonumber
\end{eqnarray}
we have set $F_0=1$.

Let us first consider the leading order terms in (\ref{eqn:integrale Iuxi})
and in (\ref{eqn:7th expression lth summand}),
bearing in mind the direct sum decompositions of 
$T_xX^\tau$ induced by $\alpha^\tau$ at any $x\in X^\tau$, and 
by $\mu^\tau$ at any $x\in Z^\tau$ (see (\ref{eqn:dicrect sum vert hor})
and Definition \ref{defn:vth dec Ztau}).

Since in NHLC's we are
unitarily identifying 
$\mathcal{H}^\tau_{x_1}\cong \mathcal{H}^\tau_{x_{12}}\cong \mathbb{R}^{2d-2}$ (recall Notation \ref{rem:NHLC euclid decomp} above), we have
\begin{eqnarray*}
\int_{\mathbb{R}^{2d-2}}\,\mathrm{d}\mathbf{u}&=&
\int_{\mathbb{R}^{d_G}_v}\,\mathrm{d}\mathbf{u}^v\,
\int_{\mathbb{R}^{d_G}_t}\,\mathrm{d}\mathbf{u}^t\,
\int_{\mathbb{R}^{2d-2-2d_G}_h}\,\mathrm{d}\mathbf{u}^h.
\end{eqnarray*}

Furthermore, let us make the following remarks on
the geodesic flow $\Gamma^\tau_t:X^\tau\rightarrow X^\tau$:

\begin{enumerate}
\item $\Gamma^\tau_t$ preserves $\alpha$ and commutes with $\mu^\tau$,
hence it preserves the vector bundle decompositions (\ref{eqn:dicrect sum vert hor})
on $X^\tau$ and in Definition \ref{defn:vth dec Ztau} on $Z^\tau$;

\item for any $\xi\in \mathfrak{g}$ and $t\in \mathbb{R} $,
the induced 
vector field $\xi_{X^\tau}$
is self-correlated under $\Gamma^\tau_t$.

\end{enumerate}

Let us consider the case $k=0$ in 
(\ref{eqn:defn di mathcalIl}).  
We have (Remark \ref{rem:omega0hlcs}): 
\begin{eqnarray}
\label{eqn:defn di mathcalI0}
\lefteqn{\mathcal{I}_0(\theta_1,\mathbf{v}_1,\theta_2,\mathbf{v}_2)_l}\\
&=&
\int_{\mathbb{R}^{2d-2}}\,\mathrm{d}\mathbf{u}\,\int_{\mathfrak{g}}\,\mathrm{d}\xi\,\left[
e^{\mathcal{S}(\theta_1,\mathbf{v}_1,\theta_2,\mathbf{v}_2;
\mathbf{u},\xi)_l}
\right]\nonumber\\
&=&\int_{\mathbb{R}^{2d-2}}\,\mathrm{d}\mathbf{u}\,\int_{\mathfrak{g}}\,\mathrm{d}\xi\,\left[
e^{\frac{1}{\tau}\,\left[\imath\,\omega_0\big(\xi_{X^\tau}(x_1),\mathbf{v}_1\big)+\psi_2\left( \mathbf{v}^{(l)}_1
-\xi_{X^\tau}(x_1)^{(l)},\mathbf{u}   \right)                                          +\psi_2\big(B\mathbf{u},\mathbf{v}_2\big)  \right]}
\right].\nonumber
\end{eqnarray}

Let us set 
\begin{equation}
\label{eqn:rescale vj}
\mathbf{v}_j(\tau):=\frac{1}{\sqrt{\tau}}\,\mathbf{v}_j,
\end{equation}
and perform the change of coordinates
$$
\xi\rightarrow\sqrt{\tau}\,\xi,\quad\mathbf{u}
\rightarrow\sqrt{\tau}\,\mathbf{u}.
$$ 
We obtain
\begin{eqnarray}
\label{eqn:defn di mathcalI01}
\lefteqn{\mathcal{I}_0(\theta_1,\mathbf{v}_1,\theta_2,\mathbf{v}_2)_l}\\
&=&
\tau^{d-1+d_G/2}\,\int_{\mathbb{R}^{2d-2}}\,\mathrm{d}\mathbf{u}\,\int_{\mathfrak{g}}\,\mathrm{d}\xi\,\left[
e^{\imath\,\omega_0\big(\xi_{X^\tau}(x_1),\mathbf{v}_1(\tau)\big)+\psi_2\left( \mathbf{v}_1(\tau)^{(l)}
-\xi_{X^\tau}(x_1)^{(l)},\mathbf{u}   \right)                                          +\psi_2\big(B\mathbf{u},\mathbf{v}_2(\tau)\big)  }
\right].\nonumber
\end{eqnarray}
We have
\begin{eqnarray}
\label{eqn:exponentSl}
\lefteqn{\imath\,\omega_0\big(\xi_{X^\tau}(x_1),\mathbf{v}_1(\tau)\big)+\psi_2\left( \mathbf{v}_1(\tau)^{(l)}
-\xi_{X^\tau}(x_1)^{(l)},\mathbf{u}   \right)                                          +\psi_2\big(B\mathbf{u},\mathbf{v}_2(\tau)\big)}
\nonumber
\\
&=&\psi_2\big(B\mathbf{u},\mathbf{v}_2(\tau)\big)
-\imath\,\omega_x\left( \mathbf{v}_1(\tau)^{(l)}
,\mathbf{u}   \right)\nonumber\\
&&-\frac{1}{2}\,\left\| \mathbf{v}^t_1(\tau)^{(l)}
-\mathbf{u}^t \right\|^2  
-\frac{1}{2}\,\left\| \mathbf{v}^h_1(\tau)^{(l)}
-\mathbf{u}^h \right\|^2    \nonumber\\
&&   +\imath\,\omega_0\left(\xi_{X^\tau}(x_1)^{(l)},
\mathbf{v}_1^t(\tau)^{(l)}+\mathbf{u}^t\right)
-\frac{1}{2}\,
\left\| \xi_{X^\tau}(x_1)^{(l)}+\mathbf{u}^v  \right\|^2               
.
\end{eqnarray}

We first compute the $\xi$-integral. We can transfer the
integral over $\mathfrak{g}$ to $T_{x_1}^vX^\tau
\cong \mathbb{R}^{d_G}_v$, the tangent space to 
$x_1^G$. If $r=|G_{x_1}|$, the action map $G\rightarrow G\cdot x_1$
is $r_{x_1}:1$ ($r_{x_1}=|G_{x_1}|$).
We then have the replacement 
$$
\int_{\mathfrak{g}}\,\mathrm{d}\xi\quad
\text{by}\quad 
\frac{1}{r_{x_1}\cdot V_{eff}(x_1)}\cdot \int_{T_{x_{1}}^vX^\tau}\mathrm{d}\mathbf{w}^v=\frac{1}{r_{x_1}\cdot V_{eff}(x_1)}\cdot 
\int_{\mathbb{R}^{d_G}_v}\mathrm{d}\mathbf{a}.
$$
Then
\begin{eqnarray}
\label{eqn:xi computint}
\lefteqn{\int_{\mathfrak{g}} e^{\imath\,\omega_0\left(\xi_{X^\tau}(x_1)^{(l)},
\mathbf{v}_1^t(\tau)^{(l)}+\mathbf{u}^t\right)
-\frac{1}{2}\,
\left\| \xi_{X^\tau}(x_1)^{(l)}+\mathbf{u}^v  \right\|^2 } 
\,\mathrm{d}\xi  }\\
&=&\frac{1}{r_{x_1}\cdot V_{eff}(x_1)}\cdot \int_{\mathbb{R}^{d_G}}
e^{\imath\,\omega_0\left(\mathbf{a}^{(l)},
\mathbf{v}_1^t(\tau)^{(l)}+\mathbf{u}^t\right)
-\frac{1}{2}\,
\left\| \mathbf{a}^{(l)}+\mathbf{u}^v  \right\|^2}
\mathrm{d}\mathbf{a}\nonumber\\
&=&   \frac{(2\,\pi)^{d_G/2}}{r_{x_1}\cdot V_{eff}(x_1)}\cdot
e^{-\imath\,\omega_0\left( \mathbf{u}^v,\mathbf{v}_1^t(\tau)^{(l)}+\mathbf{u}^t  \right)
-\frac{1}{2}\,\left\|\mathbf{v}_1^t(\tau)^{(l)}+\mathbf{u}^t \right\|^2}.
\nonumber
\end{eqnarray}

We can insert (\ref{eqn:xi computint}) in (\ref{eqn:defn di mathcalI01})
and obtain
\begin{eqnarray}
\label{eqn:defn di mathcalI02}
\lefteqn{\mathcal{I}_0(\theta_1,\mathbf{v}_1,\theta_2,\mathbf{v}_2)_l}\\
&=& \frac{(2\,\pi)^{d_G/2}}{r\cdot V_{eff}(x_1)}\cdot
\tau^{d-1+d_G/2}\,\int_{\mathbb{R}^{2d-2}}\,
e^{A(\mathbf{v}_1(\tau),\mathbf{v}_2(\tau),\mathbf{u})_l  }\,\mathrm{d}\mathbf{u}
,\nonumber
\end{eqnarray}
where 
\begin{eqnarray}
\label{eqn:defn di Al}
A(\mathbf{v}_1(\tau),\mathbf{v}_2(\tau),\mathbf{u})_l
 &=&
\psi_2\big(B\mathbf{u},\mathbf{v}_2(\tau)\big)
-\imath\,\omega_0\left( \mathbf{v}_1(\tau)^{(l)}
,\mathbf{u}   \right)\\
&&-\frac{1}{2}\,\left\| \mathbf{v}^t_1(\tau)^{(l)}
-\mathbf{u}^t \right\|^2  
-\frac{1}{2}\,\left\| \mathbf{v}^h_1(\tau)^{(l)}
-\mathbf{u}^h \right\|^2    \nonumber\\
&&   -\imath\,\omega_0\left( \mathbf{u}^v,\mathbf{v}_1^t(\tau)^{(l)}+\mathbf{u}^t  \right)
-\frac{1}{2}\,\left\|\mathbf{v}_1^t(\tau)^{(l)}+\mathbf{u}^t \right\|^2\nonumber\\
&=&-\left\|\mathbf{v}_1^t(\tau)\right\|^2
-\left\|\mathbf{u}^t\right\|^2
-\imath\,\omega_0\left( \mathbf{v}^h_1(\tau)^{(l)}
,\mathbf{u}^h   \right) 
-\frac{1}{2}\,\left\| \mathbf{v}^h_1(\tau)^{(l)}
-\mathbf{u}^h \right\|^2\nonumber\\
&&-\imath\,\omega_0\left(\mathbf{u}^v,\mathbf{u}^t\right)
+\psi_2\big(B\mathbf{u},\mathbf{v}_2(\tau)\big).\nonumber
\end{eqnarray}

\subsection{The leading order term in the action free case}
\label{sctn:action-free}

In this section, we shall determine the leading order term in
in Theorem \ref{thm:main 3}. To avoid repetitions, we shall 
later give a general argument for the lower order terms covering the
general equivariant case.

\begin{proof}
[Determination of the leading order term in Theorem \ref{thm:main 3}]
In the action-free case, we have $\mathbf{v}_j=\mathbf{v}_j^h$, and
the suffix $l$ may be omitted. Thus
\begin{eqnarray}
\label{eqn:defn di mathcalI02actionfree}
\mathcal{I}_0(\theta_1,\mathbf{v}_1,\theta_2,\mathbf{v}_2)
= 
\tau^{d-1}\,\int_{\mathbb{R}^{2d-2}}\,
e^{F(\mathbf{v}_1(\tau),\mathbf{v}_2(\tau),\mathbf{u})  }\,\mathrm{d}\mathbf{u}
,
\end{eqnarray}
where 
\begin{eqnarray*}
F\big(\mathbf{v}_1(\tau),\mathbf{v}_2(\tau),\mathbf{u}\big)
=\psi_2\big(\mathbf{v}_1(\tau),\mathbf{u}\big)
+\psi_2\big(B\mathbf{u},\mathbf{v}_2(\tau)\big).\nonumber
\end{eqnarray*}
%
%

As is Definition \ref{defn:sympl compl mat}, following \cite{fo}, 
let us write
\begin{equation}
\label{eqn:APQ}
B^{-1}_c=
\begin{pmatrix}
P&Q\\
\overline{Q}&\overline{P}
\end{pmatrix},
\end{equation}
where $P$ is invertible and $\|P\mathbf{v}\|\ge \|\mathbf{v}\|$
for every $\mathbf{v}\in \mathbb{C}^d$.
Furthermore, let us choose a metaplectic lift
$\tilde{B}^{-1}$ of $B^{-1}$, and denote by $K_{1,\tilde{B}^{-1}}$ the 
integral kernel of the metaplectic representation of 
$\tilde{B}^{-1}$. With the present normalizations
(here 
$\omega_0=(\imath/2)\,
\sum_j\mathrm{d}z_j\wedge\mathrm{d}\overline{z}_j$), 
it follows from
the discussion in \S 4 of \cite{zz18} (and Ch. 4 of \cite{fo})
that, with the notation of Definition \ref{defn:quadratice form}, $K_{1,\tilde{B}^{-1}}: 
\mathbb{C}^{d-1}\times \mathbb{C}^{d-1}\rightarrow \mathbb{C}$
is
given by
\begin{eqnarray}
\label{eqn:Pi1ZZ}
K_{1,\tilde{B}^{-1}}(Z,W)
:=
\pi^{-(d-1)}\cdot \det(P)^{-1/2}\cdot\exp
\left[\Psi_{B_c^{-1}}(Z,W)\right],
\end{eqnarray}
the square root being well-defined at the metaplectic level.
We may alternatively view $K_{1,\tilde{B}^{-1}}$ as defined on
$K_{1,\tilde{B}^{-1}}:\mathbb{R}^{2d-2}\times \mathbb{R}^{2d-2}
\rightarrow \mathbb{C}$.

On the other hand, let $\Pi_1$ be the integral kernel of the
level-$1$ Szeg\H{o} kernel on $(\mathbb{C}^{d-1},\omega_0)$ with the
standard polarization. 
Explicitly, this is given by
$$
\Pi_1(Z,W):=\pi^{-(d-1)}\,e^{\psi_2(Z,W)}.
$$
where $\psi_2$ is as in Definition \ref{defn:psi2}
and Notation \ref{notn:psi2}
(there is some abuse of language here, since the Szeg\H{o} kernel
is really defined on $X\times X$, where 
$X\cong \mathbb{C}^{d-1}\times S^1$ is an appropriate 
unit circle bundle on $\mathbb{C}^{d-1}$, and $\Pi_1$
above is in fact its pull-back under the map
$Z\mapsto (Z,1)$). Again, we can equivalently view
$\Pi_1$ as defined on $\mathbb{R}^{2d-2}\times \mathbb{R}^{2d-2}$.

The relation between $K_{1,\tilde{B}^{-1}}$
and $\Pi_1$ is described in Proposition 4.4 of \cite{zz18}
(building on the theory of \cite{dau80}).
Namely, if $Z,\,W\in \mathbb{C}^{d-1}$ 
correspond to $\mathbf{v},\,\mathbf{w}\in \mathbb{R}^{2d-2}$, then
\begin{eqnarray}
\label{eqn:ZZ18KePi}
K_{1,\tilde{B}^{-1}}(Z,W)
&=&K_{1,\tilde{B}^{-1}}(\mathbf{v},\mathbf{w})\\
&=&\det(P^*)^{1/2}\,
\int_{\mathbb{R}^{2d-2}}\,\Pi_{1}(\mathbf{v},B^{-1}\mathbf{u})\,
\Pi_1(\mathbf{u},\mathbf{w})\,\mathrm{d}\mathbf{u}.
\nonumber
\end{eqnarray}
Given (\ref{eqn:Pi1ZZ}) and (\ref{eqn:ZZ18KePi}), we may rewrite
(\ref{eqn:defn di mathcalI02actionfree}) as
\begin{eqnarray}
\label{eqn:defn di mathcalI03}
\mathcal{I}_0(\theta_1,\mathbf{v}_1,\theta_2,\mathbf{v}_2)
&=& 
\tau^{d-1}\,\int_{\mathbb{R}^{2(d-1)}}\,
e^{\psi_2(\mathbf{v}_1(\tau),\mathbf{u})
+\psi_2(B\mathbf{u},\mathbf{v}_2(\tau))}\,
\mathrm{d}\mathbf{u}\nonumber\\
&=&\pi^{2(d-1)}\,\tau^{d-1}\,\int_{\mathbb{C}^{d-1}}\,
\Pi_1\big(\mathbf{v}_1(\tau),\mathbf{u}\big)\,
\Pi_1\big(B\mathbf{u},\mathbf{v}_2(\tau)\big)
\,\mathrm{d}\mathbf{u}\nonumber\\
&=&\pi^{2(d-1)}\,\tau^{d-1}\,\int_{\mathbb{C}^{d-1}}\,
\Pi_1\big(\mathbf{v}_1(\tau),B^{-1}\mathbf{u}\big)\,
\Pi_1\big(\,\mathbf{u},\mathbf{v}_2(\tau)\big)
\,\mathrm{d}\mathbf{u}\nonumber\\
&=&\pi^{2(d-1)}\,\tau^{d-1}\,\det(P^*)^{-1/2}\,
K_{1,\tilde{B}_c^{-1}}\big(\mathbf{v}_1(\tau),\mathbf{v}_2(\tau)\big)\nonumber\\
&=&(\pi\,\tau)^{d-1}\,|\det(P)|^{-1}\,
\cdot\exp
\left[\Psi_{B^{-1}}\big(\mathbf{v}_1(\tau),
\mathbf{v}_2(\tau)\big)\right].
\end{eqnarray}

Using (\ref{eqn:defn di mathcalI03}), we obtain that to leading
order (\ref{eqn:integrale Iuxi}) is
\begin{eqnarray}
\label{eqn:integrale Iuxi action-free}
\lefteqn{ \int_{\mathfrak{g}}\,\mathrm{d}\xi\,
\int_{\mathbb{R}^{2d-2}}\,\mathrm{d}\mathbf{u}\,\big[
I_\lambda (\theta_1,\mathbf{v}_1,\theta_2,\mathbf{v}_2,\xi,\mathbf{u})_l\big]
 }\\
 &\asymp&e^{\imath\,\sqrt{\lambda}\,\frac{\theta_1-\theta_2}{\tau}}
\cdot \frac{\lambda^{2d-3}}{(2\,\pi\,\tau)^{d-1}}
\cdot \sigma_{t_1,0}^\tau (x_{1})\cdot
\chi (t_1)
\nonumber\\
&&
\cdot |\det(P)|^{-1}\,
\cdot\exp
\left[\frac{1}{\tau}\,\Psi_{B^{-1}}
\big(\mathbf{v}_1,
\mathbf{v}_2\big)\right]
\nonumber
\end{eqnarray}

\begin{lem}
\label{lem:product sigma detP}
Let the unitary factor $e^{\imath\,\theta^\tau_{t}(x)}$ be 
as in (\ref{eqn:correcting factor tau t}). Then
$$
\sigma_{t_1,0}^\tau (x_{1})\cdot\left|\det(P)\right|^{-1}=
e^{\imath\,\theta^\tau_{t_1}(x_1)}.
$$

\end{lem}

\begin{proof}
[Proof of Lemma \ref{lem:product sigma detP}]
Let $B$ be as in (\ref{eqn:defn of M12}). In view of (\ref{eqn:APQ})
and the computations in Ch. 4 of \cite{fo} (especially Proposition 4.17) we have 
$$
B_c=
\begin{pmatrix}
P^*&-Q^\dagger\\
-Q^*&P^\dagger
\end{pmatrix}.
$$
Let us write $B$ in square block form:
$$
B=\begin{pmatrix}
A'&B'\\
C'&D'
\end{pmatrix}.
$$
It follows from the computation in the proof of Lemma
3.3 of \cite{z20} (based on \cite{dau80}) that,
with the notation in (\ref{eqn:correcting factor tau t}),
\begin{eqnarray}
\label{eqn:calcolo di sigma-sigma}
\langle \sigma_J^{(x)},\sigma_{J_{t_1}}^{(x_1)}\rangle
&=&\frac{2^{d-1}}{\left|\det\big( A'+D'+\imath\,(B'-C')   \big)\right|}
\nonumber\\
&=&\frac{2^{d-1}}{\left|\det\big( 2\,P   \big)\right|}=
\frac{1}{|\det(P)|},
\end{eqnarray}
where on the last line we have made use of (54) of \cite{z20}.
The claim follows.
\end{proof}

Thus, in the action-free case,
going back to (\ref{eqn:7th expression lth summand}) we obtain
that to leading order
\begin{eqnarray}
\label{eqn:7th expression lth summand action free}
\lefteqn{\Pi^\tau_{\chi,\nu,\lambda}(x_{1,\lambda},x_{2,\lambda})}\\
&\asymp&e^{-\imath\,\lambda\,t_1}\cdot\lambda^{2-d}\,
\frac{1}{\sqrt{2\,\pi}}
\cdot e^{\imath\,\sqrt{\lambda}\,\frac{\theta_1-\theta_2}{\tau}}
\cdot \frac{\lambda^{2d-3}}{(2\,\pi\,\tau)^{d-1}}
\cdot \sigma_{t_1,0}^\tau (x_{1})\cdot
\chi (t_1)
\nonumber\\
&&
\cdot |\det(P)|^{-1}\,
\cdot\exp
\left[\frac{1}{\tau}\,\Psi_{B^{-1}}
\big(\mathbf{v}_1,
\mathbf{v}_2\big)\right]
\nonumber\\
&=&e^{-\imath\,\lambda\,t_1}\cdot\frac{\chi (t_1)}{\sqrt{2\,\pi}}
\cdot\left( \frac{\lambda}{2\,\pi\,\tau}  \right)^{d-1}
\cdot e^{\imath\,\theta^\tau_{t_1}(x_1)}\nonumber\\
&&\cdot \exp
\left\{\frac{1}{\tau}\,\left[\imath\,\sqrt{\lambda}\,(\theta_1-\theta_2)+\Psi_{B^{-1}}
\big(\mathbf{v}_1,
\mathbf{v}_2\big)\right]\right\},
\end{eqnarray}
as claimed (with $\theta^\tau(x_1,x_2)=\theta^\tau_{t_1}(x_1)$).

\end{proof}

\subsection{Proof of Theorem \ref{thm:main 2}}

Let us return to the general equivariant setting.

\begin{proof}
[Proof of statement 1 of Theorem \ref{thm:main 2}]
We are assuming $\mathbf{v}_j=\mathbf{v}_j^t$.
As in Definition \ref{defn:vth dec Ztau}, let us decompose 
$
\mathbf{u}=\mathbf{u}^{h}+\mathbf{u}^t+\mathbf{u}^v
$.
Then (\ref{eqn:defn di Al}) reduces to 
\begin{eqnarray}
\label{eqn:defn di Altonly}
A\big(\mathbf{v}_1(\tau),\mathbf{v}_2(\tau),\mathbf{u}\big)_l
 &=&-\left\|\mathbf{v}_1^t(\tau)\right\|^2
 -\left\|\mathbf{u}^t\right\|^2\\
&& -\frac{1}{2}\,\left\|\mathbf{u}^h \right\|^2
-\imath\,\omega_0\left(\mathbf{u}^v,\mathbf{u}^t\right)
+\psi_2\big(B\mathbf{u},\mathbf{v}_2^t(\tau)\big);\nonumber
\end{eqnarray}
in particular, in this case
$A(\mathbf{v}_1,\mathbf{v}_2,\mathbf{u})_l$
is independent of $l$.
Hence (\ref{eqn:defn di mathcalI02}) may be rewritten 
\begin{eqnarray}
\label{eqn:defn di mathcalI02l}
\lefteqn{\mathcal{I}_0(\theta_1,\mathbf{v}_1,\theta_2,\mathbf{v}_2)_l}\\
&=& \frac{(2\,\pi)^{d_G/2}}{r\cdot V_{eff}(x_1)}\cdot
\tau^{d-1+d_G/2}\,e^{-\left\|\mathbf{v}_1^t(\tau)\right\|^2}\cdot 
A_{x_1,x_2}\left(\mathbf{v}_2^t(\tau)\right),
\nonumber
\end{eqnarray}
where $A_{x_1,x_2}\left(\mathbf{v}_2^t(\tau)\right)$ is independent of 
$\mathbf{v}_1(\tau)$
(and $l$).

Let $A_\chi:\mathfrak{X}_\chi^\tau\rightarrow \mathbb{C}$ be as in Definition \ref{eqn:defn di Achi}.

\begin{lem}
Let $\chi_-(t):=\chi(-t)$.
With the previous hypothesis and notation, 

\begin{eqnarray*}
\overline{e^{\imath\,\theta^\tau_{t_1}(x_1)}\cdot
A_\chi(x_1,x_2)}
&=&e^{-\imath\,\theta^\tau_{-t_1}(x_1)}\cdot
A_{\chi_-}(x_2,x_1),\\
e^{-\left\|\mathbf{v}_1^t(\tau)\right\|^2}\cdot 
A_{x_1,x_2}(\mathbf{v}_2^t(\tau))&=&e^{-\left\|\mathbf{v}_1^t(\tau)\right\|^2
-\left\|\mathbf{v}_2^t(\tau)\right\|^2}\cdot A_\chi(x_1,x_2).
\end{eqnarray*}
\end{lem}

\begin{proof}
Let $\chi_-(t)=\chi(-t)$ then
$\overline{\hat{\chi}}=\widehat{\chi_-}$ and
$\chi_-\in \mathcal{C}^\infty_c\big([-t_0-\epsilon,-t_0+\epsilon] \big)$. Referring to (\ref{eqn:Sigmax12}),
\begin{equation}
\label{eqn:Sigmax12-}
\Sigma_{\chi_-}(x_2,x_1)=\Sigma_{\chi}(x_1,x_2)^{-1},
\end{equation}
where inversion is meant in the group $G\times \mathbb{R}$.

The computation in \eqref{eqn:conjugate swap} yields:
\begin{equation}\label{eqn:conjugate swap 2}
    \overline{\Pi^\tau_{\chi,\nu,\lambda}(x_{1,\lambda},x_{2,\lambda})} = \Pi^\tau_{\chi_-,\nu,\lambda}(x_{2,\lambda},x_{1,\lambda}).
\end{equation}

By assumption, $G_{x_2}=h_1^{-1}\,G_{x_1}\,h_1$.
Suppose $G_{x_1}=\left\{\kappa_1,\ldots,\kappa_r   \right\}$;
then 
$$G_{x_2}=\left\{h_1^{-1}\,\kappa_1\,h_1,\ldots, h_1^{-1}\,\kappa_r\,h_1\right\}=
\left\{\tilde{\kappa}_1,\ldots, \tilde{\kappa}_r\right\},
$$
where $\tilde{\kappa}_j:=h_1^{-1}\,\kappa_j^{-1}\,h_1$.
For every $l=1,\ldots,r$,
$$
\overline{\Xi_\nu \left(\tilde{\kappa}_l\,h_1^{-1}  \right)}=
\overline{\Xi_\nu \left(h_1^{-1}\,\kappa_l^{-1}\,h_1\,h_1^{-1}  \right)}
=\overline{\Xi_\nu \left(h_1^{-1}\,\kappa_l^{-1}  \right)}
=\Xi_\nu \left(\kappa_l\,h_1  \right).
$$

Let us apply (\ref{eqn:7th expression lth summand}),
(\ref{eqn:integrale Iuxi}), and 
(\ref{eqn:defn di mathcalI02}) with $x_{1,\lambda}$ and $x_{2,\lambda}$
swapped and $\chi$ replaced by $\chi_-$.
Then $t_1$ is replaced with $-t_1$, and
$B$ with $B^{-1}$; hence
$P$ gets replaced with $P^*$ in (\ref{eqn:APQ}) (see
the proof of Proposition 4.17 of \cite{fo}).
Therefore, 
$\left|\sigma_{-t_1,0}^\tau (x_{2})\right|=
\left| \sigma_{t_1,0}^\tau (x_{1})  \right|$.
More precisely, by Lemma \ref{lem:product sigma detP},
\begin{equation}
\label{eqn:sigma theta P}
\sigma_{t_1,0}^\tau (x_{1})=
e^{\imath\,\theta^\tau_{t_1}(x_1)}\cdot 
\left|\det(P)\right|,
\quad
\sigma_{-t_1,0}^\tau (x_{2})=
e^{\imath\,\theta^\tau_{-t_1}(x_2)}\cdot 
\left|\det(P)\right|.
\end{equation}
Then on the one hand we have that to leading order
\begin{eqnarray}
\label{eqn:7th expression lth summand chi-}
\lefteqn{\Pi^\tau_{\chi_-,\nu,\lambda}(x_{2,\lambda},x_{1,\lambda})_l}\\
&\asymp&e^{\imath\,\lambda\,t_1}\cdot
\left(\frac{\lambda}{2\,\pi\,\tau}\right)^{d-1-d_G/2}\,\frac{\dim(\nu)}{\sqrt{2\,\pi}}
\cdot e^{\imath\,\sqrt{\lambda}\,\frac{\theta_2-\theta_1}{\tau}}
\cdot e^{\imath\,\theta^\tau_{-t_1}(x_2)}\cdot 
\left|\det(P)\right|
\nonumber\\
&&\cdot
\chi (t_1)\cdot \Xi_\nu 
\left(\kappa_l\, h_1\right)\cdot \frac{1}{r\cdot V_{eff}(x_1)}\cdot
e^{-\left\|\mathbf{v}_2^t(\tau)\right\|^2}\cdot 
\frac{A_{x_2,x_1}\left(\mathbf{v}_1^t(\tau)\right)}{\pi^{d-1}}.
\nonumber
\end{eqnarray}
On the other hand, given (\ref{eqn:conjugate swap 2}), taking the complex conjugate of the expansion for $\Pi^\tau_{\chi,\nu,\lambda}(x_{1,\lambda},x_{2,\lambda})_l$ we also have
\begin{eqnarray}
\label{eqn:7th expression lth summand conj}
\lefteqn{
\overline{
\Pi^\tau_{\chi,\nu,\lambda}(x_{1,\lambda},x_{2,\lambda})_l}
}\\
&\asymp&e^{\imath\,\lambda\,t_1}\cdot
\left(\frac{\lambda}{2\,\pi\,\tau}\right)^{d-1-d_G/2}\,\frac{\dim(\nu)}{\sqrt{2\,\pi}}
\cdot e^{\imath\,\sqrt{\lambda}\,\frac{\theta_2-\theta_1}{\tau}}
\cdot e^{-\imath\,\theta^\tau_{t_1}(x_1)}\cdot 
\left|\det(P)\right|\nonumber\\
&&\cdot
\chi (t_1)\cdot \Xi_\nu 
\left(\kappa_l\, h_1\right)\cdot \frac{1}{r\cdot V_{eff}(x_1)}\cdot
e^{-\left\|\mathbf{v}_1^t(\tau)\right\|^2}\cdot 
\frac{\overline{A_{x_1,x_2}\left(\mathbf{v}_2^t(\tau)\right)}}{\pi^{d-1}}.\nonumber
\end{eqnarray}

Therefore,
$$
e^{-\left\|\mathbf{v}_2^t(\tau)\right\|^2}\cdot 
A_{x_2,x_1}(\mathbf{v}_1^t(\tau))\cdot e^{\imath\,\theta^\tau_{-t_1}(x_2)}=e^{-\left\|\mathbf{v}_1^t(\tau)\right\|^2}\cdot 
\overline{e^{\imath\,\theta^\tau_{t_1}(x_1)}\cdot A_{x_1,x_2}(\mathbf{v}_2^t(\tau))}.
$$
If 
$$
A_\chi'(x_1,x_2):=
e^{\left\|\mathbf{v}^t\right\|^2}\cdot 
 A_{x_1,x_2}(\mathbf{v}^t),
$$ 
we conclude that $A_\chi(x_1,x_2)$ is 
independent of $\mathbf{v}^t$, and this implies the
stated equalities.
Setting 
$\mathbf{v}_j=\mathbf{0}$, one obtains 
$A'_\chi(x_1,x_2)=A_\chi(x_1,x_2)$.

\end{proof}

Returning to (\ref{eqn:7th expression lth summand}), 
again in view of (\ref{eqn:sigma theta P})
we obtain that
to leading order
\begin{eqnarray}
\label{eqn:8th expression lth summand}
\lefteqn{
\Pi^\tau_{\chi,\nu,\lambda}(x_{1,\lambda},x_{2,\lambda})_l}
\\
&\asymp&\frac{e^{-\imath\,\lambda\,t_1}}{\sqrt{2\,\pi}}\,
\left( \frac{\lambda}{2\,\pi\,\tau}  \right)^{d-1-d_G/2}
\cdot e^{\imath\,\theta^\tau_{t_1}(x_1)}
\cdot \mathcal{F}_{\chi}(x_1,x_2)
\cdot \mathcal{B}_\nu(x_1,x_2)_l\nonumber\\
&&\cdot e^{\imath\,\sqrt{\lambda}\,\frac{1}{\tau}
\,(\theta_1-\theta_2)
-\frac{1}{\tau}\left(\|\mathbf{v}_1^t\|^2
+\|\mathbf{v}_2^t\|^2\right)},\nonumber
\end{eqnarray}
where
$\mathcal{F}_{\chi}(x_1,x_2)$ and 
$\mathcal{B}_\nu(x_1,x_2)_l$
are as in Definition \ref{defn:FchiBnu}
(here $g_l=\kappa_l\, h_1$).

Let us now consider the lower order terms in the
asymptotic expansion of 
(\ref{eqn:7th expression lth summand}), hence in 
(\ref{eqn:integrale Iuxi}).
For each $k$, 
arguing as in (\ref{eqn:defn di mathcalI01}) and
(\ref{eqn:exponentSl}) one verifies that
$\mathcal{I}_k(\theta_1,\mathbf{v}_1,\theta_2,\mathbf{v}_2)_l$ in (\ref{eqn:defn di mathcalIl})
is a linear combination of integrals of the form
\begin{eqnarray}
\label{eqn:I_k general summand}
\lefteqn{\theta_1^a\,\mathbf{v}_1^A\,\theta_2^{a'}\,
\mathbf{v}_2^{A'}}\nonumber\\
&&\cdot \int_{\mathbb{R}^{2d-2}}\,\mathrm{d}\mathbf{u}\,
\left[\mathbf{u}^C\,e^{
\frac{1}{\tau}\,\left[
\psi_2\big(B\mathbf{u},\mathbf{v}_2\big)
-\imath\,\omega_0\left( \mathbf{v}_1^{(l)}
,\mathbf{u}   \right)-\frac{1}{2}\,\left\| (\mathbf{v}^t_1)^{(l)}
-\mathbf{u}^t \right\|^2  
-\frac{1}{2}\,\left\| (\mathbf{v}^h_1)^{(l)}
-\mathbf{u}^h \right\|^2  \right] 
}\right.\nonumber\\
&&\left.\cdot \int_{\mathfrak{g}}\,\mathrm{d}\xi\,\left[
e^{\imath\,\omega_0\left(\xi_{X^\tau}(x_1)^{(l)},
(\mathbf{v}_1^t)^{(l)}+\mathbf{u}^t\right)
-\frac{1}{2}\,
\left\| \xi_{X^\tau}(x_1)^{(l)}+\mathbf{u}^v  \right\|^2}
\cdot  
\xi^D\right]\right],
\end{eqnarray}
where $a,\,a'\ge 0$, $A,A',C,D\ge 0$ are multi-indexes and 
$a+a'+|A|+|A'|+|C|+|D|$ has the same parity as $k$
and is $\le 3\,k$.

Let us consider the inner $\xi$-integral in (\ref{eqn:I_k general summand}). 
Arguing as for (\ref{eqn:xi computint}), we can transfer integration
to $T^vX^\tau\cong \mathbb{R}^{d_G}$, and conclude that, up to
some multiplicative constant factor, the $\xi$-integral
is given by 
\begin{eqnarray}
\label{eqn:sumD'D''}
\lefteqn{ 
\int_{\mathbb{R}^{d_G}}
\mathbf{a}^D\,
e^{\frac{1}{\tau}\left[\imath\,\omega_0\left(\mathbf{a}^{(l)},
(\mathbf{v}_1^t)^{(l)}+\mathbf{u}^t\right)
-\frac{1}{2}\,
\left\| \mathbf{a}^{(l)}+\mathbf{u}^v  \right\|^2\right]}
\mathrm{d}\mathbf{a}
}\\
&=&\sum_{D'+D''=D}\beta_{D',D''}\cdot \mathbf{u}^{D''}\,
e^{-\imath\,\omega_0\left(\mathbf{u}^v, (\mathbf{v}_1^t)^{(l)}+\mathbf{u}^t \right)}
\int_{\mathbb{R}^{d_G}}
\mathbf{b}^{D'}\,
e^{\frac{1}{\tau}\left[\imath\,\omega_0\left(\mathbf{b},
(\mathbf{v}_1^t)^{(l)}+\mathbf{u}^t\right)
-\frac{1}{2}\,
\left\| \mathbf{b} \right\|^2\right]}
\mathrm{d}\mathbf{a},\nonumber
\end{eqnarray}
for certain $\beta_{D',D''}\in \mathbb{C}$.

On the other hand, for any $\mathbf{r}\in \mathbb{R}^{d_G}$
for certain constants we have
\begin{eqnarray}
\label{eqn:general equality}
\int_{\mathbb{R}^{d_G}}\,\mathbf{b}^{D'}\,
e^{\imath\,\mathbf{b}\cdot \mathbf{r}
-\frac{1}{2}\,\|\mathbf{b}\|^2}
\,\mathbf{d}\mathbf{b}
=\sum_{L\ge 0}c_L\mathbf{r}^{L}\,e^{-\frac{1}{2}\,\|\mathbf{r}\|^2},
\end{eqnarray}
where $c_L\neq 0$ only if $L\le D'$, and 
$|L|$ and $|D'|$ have the same parity.
We conclude that (\ref{eqn:sumD'D''}) can be rewritten
in the form
\begin{equation}
\label{eqn:sumD'D''1}
\sum_{D'+L}\tilde{\beta}_{D',L}\cdot \mathbf{u}^{D''}\,
\mathbf{v}_1^L
e^{-\imath\,\omega_0\left(\mathbf{u}^v, (\mathbf{v}_1^t)^{(l)}+\mathbf{u}^t \right)-\frac{1}{2}\,\left\|(\mathbf{v}_1^t)^{(l)}+\mathbf{u}^t\right\|^2},
\end{equation}
where $\tilde{\beta}_{D',L}\neq 0$ only if $D'+L\le D$ and
$|D'|+|L|$ has the same parity as $|D|$.

If we insert (\ref{eqn:sumD'D''1}) in (\ref{eqn:I_k general summand}), assuming that
$\mathbf{v}_j=\mathbf{v}_j^t$,
we conclude that the latter is in turn a linear combination of integrals of the form
\begin{eqnarray}
\label{eqn:Ik general summand noxi}
\lefteqn{\theta_1^a\,\mathbf{v}_1^A\,\theta_2^{a'}\,
\mathbf{v}_2^{A'}\,e^{-\left\| \mathbf{v}^t_1\right\|^2}}\\
&&\cdot \int_{\mathbb{R}^{2d-2}}\,\mathrm{d}\mathbf{u}\,
\left[\mathbf{u}^C\,e^{
\frac{1}{\tau}\,\left[
\psi_2\big(B\mathbf{u},\mathbf{v}_2^t\big)
-\imath\,\omega_0\left( \mathbf{u}^v
,\mathbf{u}^t   \right)
-\left\|\mathbf{u}^t \right\|^2  
-\frac{1}{2}\,\left\|\mathbf{u}^h \right\|^2  \right] 
}\right],\nonumber
\end{eqnarray}
where $a+a'+|A|+|A'|+|C|$ has the same parity as $k$,
and is $\le 3\,k$. 

Let us define $\mathcal{L}:\mathbb{R}^{2d-2}\rightarrow \mathbb{C}$
by setting
$$
\mathcal{L}(\mathbf{v}):=
\int_{\mathbb{R}^{2d-2}}\,\mathrm{d}\mathbf{u}\,
\left[\mathbf{u}^C\,e^{
\frac{1}{\tau}\,\left[
\psi_2\big(B\mathbf{u},\mathbf{v}\big)
-\imath\,\omega_0\left( \mathbf{u}^v
,\mathbf{u}^t   \right)
-\left\|\mathbf{u}^t \right\|^2  
-\frac{1}{2}\,\left\|\mathbf{u}^h \right\|^2  \right] 
}\right].
$$ 
The exponent in the integrand has the form
$$
-\left\langle\mathbf{u}, \frac{1}{\tau}J_0 B^{-1} J_0\mathbf{v}\right\rangle+\imath\,\left\langle\mathbf{u},
 \frac{1}{\tau}\, J_0\,B^{-1}\,\mathbf{v}\right\rangle-
\frac{1}{2}\,\langle\mathbf{u},R\,\mathbf{u}\rangle,
$$
where $R$ is a certain complex symmetric matrix with positive definite real part.
It follows that $\mathcal{L}(\mathbf{v})$ is a linear 
combination of terms of the form
$$
\mathbf{v}^D\,e^{-\frac{1}{2}\,\left\langle\mathbf{v},\tilde{R}^{-1}\,
\mathbf{v}\right\rangle},
$$
where $\tilde{R}$ is another complex symmetric matrix with $\Re(\tilde{R})\gg0$, $|D|\le |C|$ and $|D|$, $|C|$ have the same parity.
If we finally insert $\mathbf{v}=\mathbf{v}_2^t$, we obtain
the claimed statement.

\end{proof}

\begin{rem}
\label{rem:complete main 3}
The same arguments for the lower order term can be applied to the
action-free case (with minor modifications), thus completing the proof
of Theorem \ref{thm:main 3}.
\end{rem}

\begin{proof}
[Proof of Statement 2 of Theorem \ref{thm:main 2}]
In this case, (\ref{eqn:defn di Al}) is
\begin{eqnarray}
\label{eqn:defn di Altht}
\lefteqn{
A(\mathbf{v}_1(\tau),\mathbf{v}_2(\tau),\mathbf{u})_l
}\\
 &=&-\left\|\mathbf{v}_1^t(\tau)\right\|^2-\frac{1}{2}\,\left\| \mathbf{v}^h_1(\tau)\right\|^2-\frac{1}{2}\,\left\|\mathbf{v}_2^t(\tau)  \right\|^2
\nonumber\\
&&-\imath\,\omega_0\left( \mathbf{v}^h_1(\tau)^{(l)}
,\mathbf{u}^h   \right)+g_0\left(\mathbf{v}^h_1(\tau)^{(l)},\mathbf{u}^h   \right) -\imath\,\omega_0\left(B\mathbf{u},\mathbf{v}_2^t(\tau)\right)
+g_0(B\mathbf{u},\mathbf{v}_2^t(\tau)\big)
\nonumber\\
&&-\left\|\mathbf{u}^t\right\|^2
-\frac{1}{2}\,\left\|\mathbf{u}^h \right\|^2
-\imath\,\omega_0\left(\mathbf{u}^v,\mathbf{u}^t\right)
-\frac{1}{2}\,\left\| B\mathbf{u}\right\|^2.
\nonumber\\
&=&-\left\|\mathbf{v}_1^t(\tau)\right\|^2-\frac{1}{2}\,\left\| \mathbf{v}^h_1(\tau)\right\|^2-\frac{1}{2}\,\left\|\mathbf{v}_2^t(\tau)  \right\|^2
-\imath\,\langle\mathbf{V}(\tau),\,A_l\,\mathbf{u}\rangle
-\frac{1}{2}\,\mathbf{u}^t\,R\,\mathbf{u},
\nonumber
\end{eqnarray}
where $\mathbf{V}(\tau)^\dagger=\begin{pmatrix}
\mathbf{v}_1^t(\tau)^\dagger&\mathbf{v}_1^h(\tau)^\dagger&\mathbf{v}_2^t(\tau)^\dagger
\end{pmatrix}$, and $A_l$ and $R$ are complex matrices, with $R=R^\dagger$ 
and $\Re (R)\gg 0$.
Thus we may rewrite (\ref{eqn:defn di mathcalI02}) as
\begin{eqnarray}
\label{eqn:defn di mathcalI02tht}
\mathcal{I}_0(\theta_1,\mathbf{v}_1,\theta_2,\mathbf{v}_2)_l
&=&\frac{(2\,\pi)^{d_G/2}}{r\cdot V_{eff}(x_1)}\cdot
\tau^{d-1+d_G/2}\,e^{-\left\|\mathbf{v}_1^t(\tau)\right\|^2-\frac{1}{2}\,\left\| \mathbf{v}^h_1(\tau)\right\|^2-\frac{1}{2}\,\left\|\mathbf{v}_2^t(\tau)  \right\|^2}\nonumber\\
&&\cdot \int_{\mathbb{R}^{2d-2}}\,
e^{-\imath\,\langle\mathbf{V}(\tau),\,A_l\,\mathbf{u}\rangle
-\frac{1}{2}\,\mathbf{u}^t\,R\,\mathbf{u}}
\,\mathrm{d}\mathbf{u}\nonumber\\
&=&\frac{(2\,\pi)^{d_G/2}}{r\cdot V_{eff}(x_1)}\cdot
\tau^{d-1+d_G/2}\,\frac{(2\,\pi)^{d-1}}{\mathrm{det}(R)^{\frac{1}{2}}}\,e^{-\frac{1}{2}\,\langle
\mathbf{V}(\tau), D_l\,\mathbf{V}(\tau)\rangle},
\end{eqnarray}
for a certain matrix $D=D^\dagger$. The ambiguity in the choice of the square root of the determinant of the matrix $R$ is resolved using analytic continuation arguments, as in Appendix A of \cite{dau80}. 
Therefore, after performing the integral, in place of
(\ref{eqn:8th expression lth summand}) we obtain 
\begin{eqnarray}
\label{eqn:9th expression lth summand tht}
\lefteqn{
\Pi^\tau_{\chi,\nu,\lambda}(x_{1,\lambda},x_{2,\lambda})_l}
\\
&\asymp&\frac{e^{-\imath\,\lambda\,t_1}}{\sqrt{2\,\pi}}\,
\left( \frac{\lambda}{2\,\pi\,\tau}  \right)^{d-1-d_G/2}
\cdot\mathrm{det}(R)^{-\frac{1}{2}}
\cdot e^{\imath\,\theta^\tau_{t_1}(x_1)}\cdot 
\mathcal{F}_{\chi}(x_1,x_2)
\cdot \mathcal{B}_\nu(x_1,x_2)_l\nonumber\\
&&\cdot e^{\imath\,\sqrt{\lambda}\,\frac{1}{\tau}
\,(\theta_1-\theta_2)
-\frac{1}{2}\,\langle
\mathbf{V}(\tau), D_l\,\mathbf{V}(\tau)\rangle}.
\end{eqnarray}

Furthermore, for $\theta_j=0$ and $\mathbf{V}=\mathbf{0}$ 
we obtain an asymptotic expansion for $\Pi_{k,\nu,\lambda}(x_1,x_2)$
which must agree with the one obtained setting 
$\theta_j=0$, $\mathbf{v}_j=\mathbf{v}_j^t=\mathbf{0}$ in the previous case.
Hence $\det (R)=1$.

In order to verify that $\Re(D_l)\gg 0$, let us consider the
special case where $\nu$ is the
trivial representation and $\theta_j=0$.
Then $\mathcal{B}_\nu(x_1,x_2)_l=1$ and
so
\begin{eqnarray}
\label{eqn:10th expression lth summand tht triv}
\lefteqn{
\Pi^\tau_{\chi,\nu,\lambda}(x_{1,\lambda},x_{2,\lambda})}
\nonumber
\\
&\sim&\frac{e^{-\imath\,\lambda\,t_1}}{\sqrt{2\,\pi}}\,
\left( \frac{\lambda}{2\,\pi\,\tau}  \right)^{d-1-d_G/2}
\cdot e^{\imath\,\theta^\tau_{t_1}(x_1)}\cdot \mathcal{F}_{\chi}(x_1,x_2)
\cdot 
\sum_{l=1}^{r_{x_1}}
e^{
-\frac{1}{2}\,\langle
\mathbf{V}(\tau), D_l\,\mathbf{V}(\tau)\rangle}
\nonumber\\
&&\cdot \left[1+O\left(\lambda^{-1/2}\right)\right]
\end{eqnarray}
uniformly for $\|\mathbf{V}\|\le C\,\lambda^{\epsilon'}$.

On the other hand, in view of Definition \ref{defn:vth dec Ztau}, we have
\begin{eqnarray*}
\label{eqn:estimate dist to Ztau}
%
\delta'_\lambda&:=&\max\left\{\mathrm{dist}_{X^\tau}\left(x_{1,\lambda},Z^\tau\right),
\mathrm{dist}_{X^\tau}\left(x_{2,\lambda},Z^\tau\right)\right\}
\nonumber\\
&\ge& \frac{1}{2\,\sqrt{\lambda}}\,\max\left\{\left\|\mathbf{v}_1^t
\right\|
,
\left\|\mathbf{v}_2^t\right\|\right\},
\end{eqnarray*}
and
\begin{eqnarray*}
\delta''_\lambda&:=&\mathrm{dist}_{X^\tau}\left(x_{1,\lambda},x_{2,\lambda}^{G\times \mathrm{supp}(\chi)}
\right)\nonumber\\
&\ge& \mathrm{dist}_{X^\tau}\left(x_{1,\lambda},x_{2,\lambda}^{G}\right)\ge \frac{1}{2\,\sqrt{\lambda}}\,\left\|\mathbf{v}_1^h\right\|.
\end{eqnarray*}
Let us set $\delta_\lambda:=\max\{\delta'_\lambda,\,
\delta''_\lambda\}$. If $\|\mathbf{V}\|=\lambda^{\epsilon'}$
for some $\epsilon'\in (0,1/6)$,
we conclude that $\delta_\lambda\ge C\,\lambda^{\epsilon-1/2}$,
and therefore 
$\Pi^\tau_{\chi,\nu,\lambda}(x_{1,\lambda},x_{2,\lambda})=
O\left(\lambda^{-\infty}\right)$ by Theorem \ref{thm:main 1}.
However, this will be true only if every summand  
in (\ref{eqn:10th expression lth summand tht triv}) is rapidly
decreasing, i.e. if $\Re(D_l)\gg 0$ for every $l$.

The remaining arguments are a repetition of previous ones.

\end{proof}

\subsection{Proof of Theorem \ref{thm:diagonal case}}

We now consider near-diagonal scaling asymptotics, thus with
$x_1=x_2=x$, and assume 
$\chi\in \mathcal{C}^\infty\big( (-\epsilon,\epsilon)\big)$
for some $\epsilon>0$ suitably small; then $t_1=0$,
$x_{12}=x$, and  $B=I_{2d-2}$ (the identity matrix)
in (\ref{eqn:defn of M12}).
Furthermore, we may
take $h_1=e$ in (\ref{eqn:Sigmax12chi}).
We adopt the notation (\ref{eqn:xx12}).

\begin{proof}
[Proof of Theorem \ref{thm:diagonal case}]
The exponent (\ref{eqn:defn di Al}) becomes
\begin{eqnarray}
\label{eqn:defn di Alt=0}
\lefteqn{
A(\mathbf{v}_1(\tau),\mathbf{v}_2(\tau),\mathbf{u})_l
}\\
 &=&-\left\|\mathbf{v}_1^t(\tau)\right\|^2
-\left\|\mathbf{u}^t\right\|^2
-\imath\,\omega_0\left( \mathbf{v}^h_1(\tau)^{(l)}
,\mathbf{u}^h   \right) 
-\frac{1}{2}\,\left\| \mathbf{v}^h_1(\tau)^{(l)}
-\mathbf{u}^h \right\|^2\nonumber\\
&&-\imath\,\omega_0\left(\mathbf{u}^v,\mathbf{u}^t\right)
+\psi_2\big(\mathbf{u},\mathbf{v}_2(\tau)\big)\nonumber\\
&=&A(\mathbf{v}_1(\tau),\mathbf{v}_2(\tau),\mathbf{u})^h_l+
A(\mathbf{v}_1(\tau),\mathbf{v}_2(\tau),\mathbf{u})^{vt}_l,
\nonumber
\end{eqnarray}
where
\begin{eqnarray}
\label{eqn:Achih}
\lefteqn{A(\mathbf{v}_1(\tau),\mathbf{v}_2(\tau),\mathbf{u})^h_l:=
-\imath\,\omega_0\left( \mathbf{v}^h_1(\tau)^{(l)}-
\mathbf{v}^h_2(\tau)
,\mathbf{u}^h   \right) }\\
&&-\frac{1}{2}\,\left\| \mathbf{v}^h_2(\tau)\right\|^2
-\frac{1}{2}\,\left\| \mathbf{v}^h_1(\tau)^{(l)}\right\|^2
-\left\|\mathbf{u}^h \right\|^2+
g_0\left(\mathbf{v}^h_1(\tau)^{(l)}+\mathbf{v}^h_2(\tau), 
\mathbf{u}^h \right),\nonumber
\end{eqnarray}
\begin{eqnarray}
\label{eqn:Achivt}
\lefteqn{A(\mathbf{v}_1(\tau),\mathbf{v}_2(\tau),\mathbf{u})^{vt}_l:=
-\left\|\mathbf{v}_1^t(\tau)\right\|^2
-\left\|\mathbf{u}^t\right\|^2}\\
&&-\imath\,\omega_0\left(\mathbf{u}^v,\mathbf{u}^t\right)
-\frac{1}{2}\,\left\|\mathbf{u}^v\right\|^2
-\omega_0\left(\mathbf{u}^v,\mathbf{v}_2^t(\tau)  \right)
-\frac{1}{2}\,\left\|\mathbf{u}^t-\mathbf{v}_2^t(\tau)\right\|^2,\nonumber
\end{eqnarray}

By standard Gaussian integrations, we obtain
\begin{eqnarray}
\label{eqn:integrale Ah}
\lefteqn{\int_{\mathbb{R}^{2\,(d-1-d_G)}}
e^{A(\mathbf{v}_1(\tau),\mathbf{v}_2(\tau),\mathbf{u})^h_l}\,\mathrm{d}\mathbf{u}^h}\nonumber\\
&=&e^{-\frac{1}{2}\,\left\| \mathbf{v}^h_2(\tau)\right\|^2
-\frac{1}{2}\,\left\| \mathbf{v}^h_1(\tau)^{(l)}\right\|^2
-\imath\,\omega_0\left(\mathbf{v}^h_1(\tau)^{(l)}, \mathbf{v}^h_2(\tau) \right)+\frac{1}{4}\,\left\| \mathbf{v}^h_1(\tau)^{(l)}
+\mathbf{v}^h_2(\tau)
\right\|^2}\nonumber\\
&&\cdot 
\int_{\mathbb{R}^{2\,(d-1-d_G)}}e^{-\|\mathbf{a}\|^2-\imath\,\omega_0\big(\mathbf{v}_1^h(\tau)^{(l)}-\mathbf{v}_2(\tau),\mathbf{a}\big)}\,\mathrm{d}\mathbf{a}
\nonumber\\
&=&\pi^{d-1-d_G}\,e^{\psi_2\left(\mathbf{v}^h_1(\tau)^{(l)}, \mathbf{v}^h_2(\tau) \right)},
\end{eqnarray}
\begin{eqnarray}
\label{eqn:intAvt}
\lefteqn{\int_{\mathbb{R}^{2\,d_G}}
e^{A(\mathbf{v}_1(\tau),\mathbf{v}_2(\tau),\mathbf{u})^{vt}_l}\,\mathrm{d}\mathbf{u}^v\,\mathrm{d}\mathbf{u}^t=e^{-\left\|\mathbf{v}_1^t(\tau)\right\|^2}}\\
&&\cdot\int_{\mathbb{R}^{d_G}}\,\mathrm{d}\mathbf{u}^t\,
\left[ 
e^{-\left\|\mathbf{u}^t\right\|^2-\frac{1}{2}\,
\left\|\mathbf{u}^t-\mathbf{v}_2^t(\tau)\right\|^2}
\int_{\mathbb{R}^{d_G}}e^{-\imath\,\omega_0\left(\mathbf{u}^v,\mathbf{u}^t+\mathbf{v}_2^t(\tau)\right)-\frac{1}{2}\,\left\|\mathbf{u}^v\right\|^2}
\,\mathrm{d}\mathbf{u}^v
\right]\nonumber\\
&=&(2\,\pi)^{d_G/2}\,e^{-\left\|\mathbf{v}_1^t(\tau)\right\|^2}
\cdot\int_{\mathbb{R}^{d_G}}\,e^{-\left\|\mathbf{u}^t\right\|^2-\frac{1}{2}\,
\left\|\mathbf{u}^t-\mathbf{v}_2^t(\tau)\right\|^2
-\frac{1}{2}\,
\left\|\mathbf{u}^t+\mathbf{v}_2^t(\tau)\right\|^2
}\,\mathrm{d}\mathbf{u}^t\\
&=&\pi^{d_G}\,e^{-\left\|\mathbf{v}_1^t(\tau)\right\|^2
-\left\|\mathbf{v}_2^t(\tau)\right\|^2}.
\nonumber
\end{eqnarray}

Hence, recalling (\ref{eqn:rescale vj}), (\ref{eqn:defn di mathcalI02}) 
reduces to
\begin{eqnarray}
\label{eqn:defn di mathcalI02t0=0diag}
\lefteqn{\mathcal{I}_0(\theta_1,\mathbf{v}_1,\theta_2,\mathbf{v}_2)_l}\\
&=& \frac{(2\,\pi)^{d_G/2}}{r\cdot V_{eff}(x_1)}\cdot
\tau^{d-1+d_G/2}\,\int_{\mathbb{R}^{2d-2}}\,
e^{A(\mathbf{v}_1(\tau),\mathbf{v}_2(\tau),\mathbf{u})_l  }\,\mathrm{d}\mathbf{u}
\nonumber\\
&=&\frac{(2\,\pi)^{d_G/2}}{r\cdot V_{eff}(x_1)}\cdot
\tau^{d-1+d_G/2}\,\pi^{d-1}\,e^{\frac{1}{\tau}\,\left[
-\left\|\mathbf{v}_1^t\right\|^2
-\left\|\mathbf{v}_2^t\right\|^2+
\psi_2\left((\mathbf{v}^h_1)^{(l)}, \mathbf{v}^h_2 \right)\right]}.
\nonumber
\end{eqnarray}
Since $\sigma^\tau_{0,0}=1$ by (\ref{eqn:sigma theta P}), 
the leading order term in
(\ref{eqn:integrale Iuxi}) becomes
\begin{eqnarray}
\label{eqn:integrale Iuxit=0diag}
\lefteqn{ \int_{\mathfrak{g}}\,\mathrm{d}\xi\,
\int_{\mathbb{R}^{2d-2}}\,\mathrm{d}\mathbf{u}\,\big[
I_\lambda (\theta_1,\mathbf{v}_1,\theta_2,\mathbf{v}_2,\xi,\mathbf{u})_l\big]
 }\\
 &\sim&e^{\imath\,\sqrt{\lambda}\,\frac{\theta_1-\theta_2}{\tau}}
\cdot \frac{\lambda^{2d-3}}{(2\,\pi^2\,\tau^2)^{d-1}}
\cdot \sigma_{0,0}^\tau (x_{1})\cdot
\chi (0)\cdot \overline{\Xi_\nu 
\left(\kappa_l\right)}
\nonumber\\
&&
\cdot \frac{(2\,\pi)^{d_G/2}}{r\cdot V_{eff}(x_1)}\cdot
\tau^{d-1+d_G/2}\,\pi^{d-1}\,e^{\frac{1}{\tau}\,\left[
-\left\|\mathbf{v}_1^t\right\|^2
-\left\|\mathbf{v}_2^t\right\|^2+
\psi_2\left((\mathbf{v}^h_1)^{(l)}, \mathbf{v}^h_2 \right)\right]}
\nonumber\\
&=&e^{\frac{1}{\tau}\,\left[\imath\,\sqrt{\lambda}\,
(\theta_1-\theta_2)
-\left\|\mathbf{v}_1^t\right\|^2
-\left\|\mathbf{v}_2^t\right\|^2+
\psi_2\left((\mathbf{v}^h_1)^{(l)}, \mathbf{v}^h_2 \right)\right]}
\cdot \frac{\lambda^{2d-3}}{(2\,\pi\,\tau)^{d-1-d_G/2}}
\cdot \frac{\chi (0)\cdot \overline{\Xi_\nu 
\left(\kappa_l\right)}}{r\cdot V_{eff}(x_1)}.\nonumber
\end{eqnarray}
Inserting (\ref{eqn:integrale Iuxit=0diag})
in (\ref{eqn:7th expression lth summand}),
we finally obtain that to leading order
\begin{eqnarray}
\label{eqn:7th expression lth summand diago}
\Pi^\tau_{\chi,\nu,\lambda}(x_{1,\lambda},x_{2,\lambda})_l
&\asymp&\left( \frac{\lambda}{2\,\pi\,\tau}\right)^{d-1-d_G/2}\,\frac{\dim(\nu)}{\sqrt{2\,\pi}}\cdot 
\frac{\chi (0)\cdot \overline{\Xi_\nu 
\left(\kappa_l\right)}}{r\cdot V_{eff}(x_1)}\nonumber\\
&&\cdot e^{\frac{1}{\tau}\,\left[\imath\,\sqrt{\lambda}\,
(\theta_1-\theta_2)
-\left\|\mathbf{v}_1^t\right\|^2
-\left\|\mathbf{v}_2^t\right\|^2+
\psi_2\left((\mathbf{v}^h_1)^{(l)}, \mathbf{v}^h_2 \right)\right]},
\nonumber
\end{eqnarray}


The argument for the lower order terms is the same as for
Theorem \ref{thm:main 2}.

\end{proof}

\section{Scaling asymptotics for $P_{k,\nu,\lambda}$}

\label{sctn:complexified eigenf lapl}

The previous techniques may be applied to the
asymptotics of the complexifications of the equivariant eigenfunctions
of the positive Laplacian operator $\Delta$ on $(M,\kappa)$
(recall Theorem \ref{thm:holo extension}).
In this section we describe the necessary adaptations to the
previous arguments.

As in (\ref{eqn:eigenvalue laplacian}),
let the $\mu_j$'s be the distinct
eigenvalues of $\sqrt{\Delta}$.
Let us choose for every $j$ a real orthonormal
basis $(\varphi_{j,k})_{k=1}^{\ell_j'}$ of the eigenspace $W_j$ of $\mu_j$.
Similarly, for every $j$ such that
the equivariant component 
$W_{j,\nu}$ of $W_{j}$ is non-zero 
let 
$(\varphi_{j,\nu,k})_{k=1}^{\ell_{j,\nu}'}$ be
a real orthonormal basis of $W_{j,\nu}$.

Let us recall the following basic facts; a detailed discussion
may be found in \cite{z12}, \cite{z13}, \cite{z20}.

\begin{enumerate}
\item The wave operator 
$U(t)=e^{\imath\,t\,\sqrt{\Delta}}:L^2(M)\rightarrow L^2(M)$
at time $t\in \mathbb{R}$
is the unitary operator with distributional kernel
$$
U(t)(m,n)=\sum_{j=1}^{+\infty}
e^{\imath\,\mu_j\,t}\,\sum_{k=1}^{\ell'_j}
\,\varphi_{j,k}(m)\cdot \varphi_{j,k}(n).
$$
\item For every $\tau>0$, the distributional kernel
$$
U(\imath\,\tau)(m,n):=\sum_{j=1}^{+\infty}
e^{-\mu_j\,\tau}\,\sum_{k=1}^{\ell'_j}
\,\varphi_{j,k}(m)\cdot \varphi_{j,k}(n)
$$
is globally real-analytic on $M	\times M$; 
there exists $\tau_1\in (0,\tau_0]$ such that if 
$\tau\in (0,\tau_1]$ then $U(\imath\,\tau)$
admits a
holomorphic extension on $M^\tau\times M$ in the first variable,
$$
\widetilde{U}(\imath\,\tau)(x,n)=\sum_{j=1}^{+\infty}
e^{-\mu_j\,\tau}\,\sum_{k=1}^{\ell'_j}
\,\widetilde{\varphi}_{j,k}(x)\cdot \varphi_{j,k}(n).
$$
\item For $\tau\in (0,\tau_1)$,
the restriction of $\widetilde{U}(\imath\,\tau)$
to $X^\tau\times M$ is the distributional kernel of an operator
\begin{equation}
\mathfrak{P}^\tau:\mathcal{C}^\infty(M)\rightarrow \mathcal{O}(X^\tau),
\label{eqn:defn Ptau}
\end{equation}
which is a Fourier integral operator with complex phase 
of degree $-(d-1)/4$; 
\item  
$\mathfrak{P}
^\tau$ controls the complexification of the eigenfunctions
$\widetilde{\varphi}_{j,k}$:
leaving restriction to $X^\tau$ implicit,
$$
\widetilde{\varphi}_{j,k}=e^{\tau\,\mu_j}\,\mathfrak{P}
^\tau(\varphi_{j,k}).
$$ 
\item For every $s\in \mathbb{R}$, 
$\mathfrak{P}^\tau$ determines a
continuous isomorphism of Sobolev spaces 
$$
\mathfrak{P}^\tau:W^s(M)\rightarrow \mathcal{O}^{s+\frac{d-1}{4}}(X^\tau),
$$
where $W^s(M)$
is the $s$-th Sobolev space of $M$, and $\mathcal{O}^{k}(X^\tau)$ is the intersection
of the $k$-th Sobolev space of $X^\tau$ with the space of
CR (generalized) functions on $X^\tau$.

\item The composition 
$U_\mathbb{C}(t+2\,\imath\,\tau):=\mathfrak{P}
^\tau\circ U(t)\circ {\mathfrak{P}
^\tau}^*$
is a Fourier integral operator with complex phase of degree 
$-(d-1)/2$ on $X^\tau$, whose distributional kernel
admits the spectral description 
\begin{eqnarray}
\label{eqn:UC_spectral}
U_\mathbb{C}(t+2\,\imath\,\tau)(x,y)&=&
\sum_j e^{(-2\,\tau+\imath\,t)\,\mu_j}\,\sum_{k=1}^{\ell'_j}\,
\widetilde{\varphi}_{j,k}(x)\,\overline{\widetilde{\varphi}_{j,k}(y)}
\\
&=&\sum_j e^{\imath\,t\,\mu_j}\,\sum_{k=1}^{\ell'_j}\,
\mathfrak{P}
^\tau\left(\widetilde{\varphi}_{j,k}\right)(x)\,
\overline{
\mathfrak{P}
^\tau(\varphi)_{j,k}(y)};\nonumber
\end{eqnarray}
we shall also use the notation
$U_\mathbb{C}(t+2\,\imath\,\tau,x,y)=
U_\mathbb{C}(t+2\,\imath\,\tau)(x,y)$.

\item For every $s\in \mathbb{R}$, 
$U_\mathbb{C}(t+2\,\imath\,\tau)$
determines a continuous isomorphism
of CR Sobolev spaces
$$
U_\mathbb{C}(t+2\,\imath\,\tau):\mathcal{O}^s(X^\tau)\rightarrow
\mathcal{O}^{s+\frac{d-1}{2}}(X^\tau).
$$
\end{enumerate}

For the following computations, we need the
description (due to Zelditch) of of the
complexified Poisson-wave operators
$U_\mathbb{C}(t+2\,\imath\,\tau)$ as a \lq dynamical
Toeplitz operator\rq. The latter relies in turn
on the study of the composition
$$\mathfrak{Q}^\tau:=
{\mathfrak{P}^\tau}^*
\circ \mathfrak{P}^\tau:\mathcal{C}^\infty(M)\rightarrow \mathcal{C}^\infty(M),$$
where $\mathfrak{P}^\tau$
is as in (\ref{eqn:defn Ptau}).
As explained in \S 3 of \cite{z07},
$\mathfrak{Q}^\tau$ is an elliptic 
pseudodifferential operator
on $M$, of degree $-(d-1)/2$, and its principal symbol 
$\sigma(\mathfrak{Q}^\tau)$
appears in the 
description of $U_\mathbb{C}(t+2\,\imath\,\tau)$ 
in terms of
dynamical Toeplitz operators (see \cite{z12}, \cite{z14}, \cite{z20}).

Since ${\mathfrak{P}^\tau}^*$ depends on the
choice of volume form on $X^\tau$, so does 
$\mathfrak{Q}^\tau$.
In this section, we review the computation of $\sigma(\mathfrak{Q}^\tau)$, which was
carried out by Zelditch in \S 3 of \cite{z07}, in light
of the choices that we have adopted (see also the discussion in \cite{p24}).
We follow the general heuristic strategy in \cite{z07}: first consider the Euclidean case,
and then reduce the general case to the latter.

Thus we first assume that $M=\mathbb{R}^d$ with the
standard metric, and $X^\tau\cong \mathbb{R}^d\times S^{d-1}_\tau$
(where $S^{d-1}_\tau=\tau\,S^{d-1}$ is the sphere of radius $\tau$ centered at the
origin).

In the real domain, the wave kernel at time $t$ for 
$\mathbb{R}^d$
with the standard metric is 
\begin{equation}
\label{eqn:wave kernel time t}
E_t(x,y):=\frac{1}{(2\,\pi)^d}\,\int_{\mathbb{R}^d}
\,e^{\imath\,(t\,\|\xi\|+\langle\xi,x-y\rangle)}\,
\mathrm{d}\xi.
\end{equation}
This may be analytically continued to the complex domain
in $t$ and $x$ by replacing
$t$ with $t+\imath\,\tau$ for $\tau>0$ and 
$x$ with $\zeta=x+\imath\,p\in \mathbb{C}^d$.
In particular, for $t=0$ we obtain the kernel
\begin{equation}
\label{eqn:Etau}
\mathfrak{E}^\tau(\zeta,y)=
\frac{1}{(2\,\pi)^d}\,\int_{\mathbb{R}^d}
\,e^{-\tau\,\|\xi\|-
\langle\xi,p\rangle
+\imath\,\langle\xi,x-y\rangle}\,
\mathrm{d}\xi,
\end{equation}
which is absolutely convergent and holomorphic
in $\zeta$ on the locus
where $\|p\|<\tau$ (which plays the role of $M^\tau$).
The distributional kernel of $\mathfrak{P^\tau}$
is given by the restriction of 
(\ref{eqn:Etau}) to the locus where $\|p\|=\tau$
(which plays the role of $X^\tau$).

With respect to the given volume forms,
$(\mathfrak{P}^\tau)^*$ is represented
by the integral kernel
$$
(\mathfrak{P}^\tau)^*(y,\zeta)=
\overline{\mathfrak{P}^\tau (\zeta,y)}=
\frac{1}{(2\,\pi)^d}\,\int_{\mathbb{R}^d}
\,e^{-\tau\,\|\xi\|-
\langle\xi,p\rangle
-\imath\,\langle\xi,x-y\rangle}\,
\mathrm{d}\xi,
$$
where $\zeta=x+\imath\,p$.

Therefore $ \mathfrak{Q}^\tau$
is represented by the operator kernel
\begin{eqnarray}
\label{eqn:operator kernel comp}
\mathfrak{Q}^\tau
(x,y)
=\int_{X^\tau}(\mathfrak{P}^\tau)^*(x,\zeta)\,
\,\mathfrak{P}^\tau(\zeta,y)\,\mathrm{d}V_{X^\tau}(\zeta).
\end{eqnarray}
Let us write $\zeta=x'+\imath\,\tau\,\omega$, where $x'\in \mathbb{R}^d$
and $\omega\in S^{d-1}$. Then 
$\mathrm{d}V_{X^\tau}(x')=\tau^{d-1}\,\mathrm{d}x'\,
\mathrm{d}\omega$.
We obtain
\begin{eqnarray}
\label{eqn:Qtauxy}
\lefteqn{\mathfrak{Q}^\tau
(x,y)}\\
&=&\frac{\tau^{d-1}}{(2\,\pi)^{2\,d}}
\,\int_{\mathbb{R}^d}\,\mathrm{d}x'\,\int_{S^{d-1}}\mathrm{d}\omega
\left[ (\mathfrak{P}^\tau)^*(x,\zeta)\,
\,\mathfrak{P}^\tau(\zeta,y)\,     \right]
\nonumber\\
&=&\frac{\tau^{d-1}}{(2\,\pi)^{2\,d}}
\,\int_{\mathbb{R}^d}\,\mathrm{d}x'\,\int_{S^{d-1}}\mathrm{d}\omega
\,\int_{\mathbb{R}^d}\,
\mathrm{d}\xi_1\,\int_{\mathbb{R}^d}\,
\mathrm{d}\xi_2
\nonumber\\
&&\left[
e^{-\tau\,(\|\xi_1\|+\|\xi_2\|)-
\tau\,\langle\xi_1+\xi_2,\omega\rangle
+\imath\,\langle\xi_2,x'-y\rangle
-\imath\,\langle\xi_1,x'-x\rangle
}\,
\right]\nonumber\\
&=&\frac{\tau^{d-1}}{(2\,\pi)^{2\,d}}
\,\int_{\mathbb{R}^d}\,\mathrm{d}x'\,\int_{S^{d-1}}\mathrm{d}\omega
\,\int_{\mathbb{R}^d}\,
\mathrm{d}\xi_1\,\int_{\mathbb{R}^d}\,
\mathrm{d}\xi_2
\nonumber\\
&&\left[
e^{-\tau\,(\|\xi_1\|+\|\xi_2\|)-
\tau\,\langle\xi_1+\xi_2,\omega\rangle
+\imath\,\langle\xi_2-\xi_1,x'\rangle
+\imath\,\langle\xi_1,x\rangle-\imath\,\langle
\xi_2,y\rangle
}\,
\right].\nonumber
\end{eqnarray}
Using the distributional identity
$$
\frac{1}{(2\,\pi)^d}
\,\int_{\mathbb{R}^d}\,
e^{\imath\,\langle\xi_2-\xi_1,x'\rangle}\,
\mathrm{d}x'=\delta (\xi_2-\xi_1),
$$
we can rewrite (\ref{eqn:Qtauxy}) as
\begin{eqnarray}
\label{eqn:Qtauxy1}
\mathfrak{Q}^\tau
(x,y)
=
\frac{\tau^{d-1}}{(2\,\pi)^{d}}
\,\int_{\mathbb{R}^d}\,e^{\imath\,\langle\xi,x-y\rangle}
\left[ \int_{S^{d-1}} e^{-2\,\tau\,\|\xi\|-
2\,\tau\,\langle\xi,\omega\rangle
}\, \mathrm{d}\omega  \right]
\mathrm{d}\xi.
\end{eqnarray}
We aim to evaluate the inner integral asymptotically for
$\xi\rightarrow \infty$. To this end, we set
$\xi=\lambda\,\eta$, where $\lambda>0$ and $\eta\in S^{d-1}$,
and let $\lambda\rightarrow +\infty$.
We obtain
\begin{eqnarray}
\label{eqn:inner integral}
\int_{S^{d-1}} e^{-2\,\tau\,\|\xi\|-
2\,\tau\,\langle\xi,\omega\rangle
}\, \mathrm{d}\omega  
&=&\int_{S^{d-1}} e^{\imath \lambda\cdot \Psi_\eta^\tau (\omega)
}\, \mathrm{d}\omega,
\end{eqnarray}
where
$$
\Psi_\eta^\tau (\omega):=
\imath\,2\,\tau\,\big(1+
\langle\eta,\omega\rangle\big).
$$
Thus $\Psi_\eta^\tau$ is purely imaginary, and
$\Im\left(\Psi_\eta^\tau\right)\ge 0$. 
Furthermore, $\Im\left(\Psi_\eta^\tau\right)$ vanishes only at
$\omega=-\eta$. Hence, without altering the asymptotics for
$\lambda\rightarrow +\infty$, we may replace integration over
$S^{d-1}$ by integration over an arbitrarily small open neighbourhood 
$S_\eta\subset S^{d-1}$
of $-\eta$.

Furthermore, given the form of
$\Psi_\eta^\tau:S^{d-1}\rightarrow \mathbb{R}$, there is
no loss of generality in assuming that $\eta$ is the last vector
of the standard basis of $\mathbb{R}^d$, that is,
$\eta=\begin{pmatrix}
\mathbf{0}&1
\end{pmatrix}^\dagger$. 
Thus any $\omega\in S_\eta$ may be written
$\omega=\begin{pmatrix}
\mathbf{u}&-\sqrt{1-\|\mathbf{u}\|^2}
\end{pmatrix}^\dagger$,
where $\mathbf{u}\in \mathbb{R}^{d-1}$ ranges in a small
neighbourhood of the origin.
Then
$
\mathrm{d}\omega=
\mathcal{V}(\mathbf{u})\,\mathrm{d}\mathbf{u}
$ on $S_\eta$
where $\mathcal{V}(\mathbf{0})=1$.
For $\mathbf{u}\sim \mathbf{0}$,
\begin{eqnarray}
\label{eqn:phase near critical pt}
\Psi_\eta^\tau (\omega)=
\imath\,2\,\tau\,\left(1-\sqrt{1-\|\mathbf{u}\|^2}\right)
=\imath\,\tau\,\left(\|\mathbf{u}\|^2+R_3(\mathbf{u})\right).
\end{eqnarray}
Thus there is a unique critical point
$\mathbf{u}=\mathbf{0}$
(that is, $\omega=-\eta$), with Hessian matrix
$
H(\Psi_\eta^\tau)=2\,\imath\,\tau\,I_{d-1}$.

Thus
$$
\sqrt{\det\left(\frac{\lambda\,H(\Psi_\eta^\tau)}{2\pi\,\imath}   \right)}
=\left(\frac{\lambda\tau}{\pi}\right)^{\frac{d-1}{2}}.
$$

Applying the Lemma of stationary phase, we obtain for
(\ref{eqn:inner integral}) an aymptotic expansion in descending powers of $\lambda=\|\xi\|$, with leading order term
$(\pi/\lambda\tau)^{(d-1)/2}$. In view of (\ref{eqn:Qtauxy1})
the principal symbol of $\mathfrak{Q}^\tau$ is therefore
\begin{equation}
\label{eqn:Qtau euclidean case}
\sigma(\mathfrak{Q}^\tau)=\tau^{d-1}\,
\left(\frac{\pi}{\lambda\,\tau}\right)
^{\frac{d-1}{2}}=
\left(\frac{\pi\,\tau}{\|\xi\|}\right)
^{\frac{d-1}{2}}.
\end{equation}

Before considering the general case, let us premise a remark considering scalar rescalings of a metric.
Suppose $X=X^\tau$ and consider the operator
$\mathfrak{P}^\tau:\mathcal{C}^\infty(M)\rightarrow \mathcal{O}(X)$.
To emphasize the role of the metric, let us write 
$\mathfrak{P}^\tau=\mathfrak{P}^\tau_\kappa$, $X=X^\tau_\kappa$. 

When $\kappa$ is replaced by
$\kappa^\lambda:= \lambda^2\,\kappa$ for some $\lambda\in \mathbb{R}_+$,
we have $X=X^{\lambda\,\tau}_{\kappa^\lambda}$.
If $\Delta_\kappa$ and $\Delta_{\kappa^\lambda}$ are the Laplacians
for $\kappa$ and $\kappa^\lambda$, respectively, then 
$\Delta_{\kappa^\lambda}=\lambda^{-2}\,\Delta_\kappa$. Hence, 
$$
e^{-\tau\,\sqrt{\Delta_\kappa}}=e^{-\lambda\,\tau\,\sqrt{\Delta_{\kappa^\lambda}}}.
$$

Thus it makes sense to denote $\mathfrak{P}^\tau_\kappa$ as $\mathfrak{P}^X$,
without reference to a specific rescaling of $\kappa$.

On the other hand, the rescaling affects the adjoint operator, since
it modifies the volume form on $X$.
Let $\mathrm{vol}_{X,\kappa}$ and let $\mathrm{vol}_{X,\kappa^\lambda}$
be the volume forms on $X$ viewed as
$X^\tau_\kappa$ and $X^{\lambda\,\tau}_{\kappa^\lambda}$.
Similarly, let $\mathrm{vol}_{M,\kappa}$ and 
$\mathrm{vol}_{M,\kappa^\lambda}$
be the Riemannian volume forms on $M$ associated to $\kappa$ and
$\kappa^\lambda$, respectively.
Then
\begin{equation}
\label{eqn:rescaled volume forms}
\mathrm{vol}_{M,\kappa^\lambda}=\lambda^d\,\mathrm{vol}_{M,\kappa},
\quad 
\mathrm{vol}_{X,\kappa^\lambda}=\lambda^{2d-1}\,\mathrm{vol}_{X,\kappa}.
\end{equation}
Let $({\mathfrak{P}}^X)^*_\kappa$ and 
$({\mathfrak{P}}^X)^*_{\kappa^\lambda}$
be the adoints of $\mathfrak{P}^X$ with respect to 
$\kappa$ and $\kappa^\lambda$ (that is, using the 
pairs of volume forms $(\mathrm{vol}_{M,\kappa},\mathrm{vol}_{X,\kappa})$,
$(\mathrm{vol}_{M,\kappa^\lambda},\mathrm{vol}_{X,\kappa^\lambda})$,
respectively).
One concludes from (\ref{eqn:rescaled volume forms}) that
\begin{equation*}
({\mathfrak{P}}^X)^*_{\kappa^\lambda}=\lambda^{d-1}\,
({\mathfrak{P}}^X)^*_{\kappa}.
\end{equation*}
Since $X=X^\tau_\kappa=X^{\lambda\,\tau}_{\kappa^\lambda}$,
we obtain 
\begin{equation}
\label{eqn:rescaling Qtau}
\mathfrak{Q}^{\lambda\,\tau}_{\kappa^\lambda}=
\lambda^{d-1}\,
\mathfrak{Q}^\tau_\kappa;
\end{equation}
therefore the same relation holds between the respective
principal symbols.
Since $\mathfrak{Q}^\tau_\kappa$ has degree
$-(d-1)/2$,
we conclude that its principal symbol 
has the form
\begin{equation}
\label{eqn:homog prop sigma}
\sigma(\mathfrak{Q}^\tau_\kappa)(m,\xi)=
c_m(\xi)\,\left(\frac{\tau\,\pi}{\|\xi\|}  \right)
^{\frac{d-1}{2}},
\end{equation}
where $c_m(\xi)$ is homogeneous of degree
$0$ in $\xi$.

Let us consider a general real-analytic
$(M,\kappa)$ and fix $m\in M$. Let us choose a real-analytic
coordinate chart $\varphi:B\rightarrow U$ centered at $m$ that is isometric at the origin. Thus, $B\subseteq \mathbb{R}^d$
is an open neighborhood of the origin (say, an open ball
centered at $\mathbf{0}$), $U$ an open neighbourhood of
$m$ in $M$, and $\varphi$ a real-analytic diffeomorphism
such that $\mathrm{d}_\mathbf{0}\varphi$ is a linear isometry
between $(\mathbb{R}^d,g_{st})$ (where $g_{st}$ is the
standard Euclidean product) and $(T_m M,\kappa_x)$.
The pull-back metric 
$\varphi^*(\kappa)$ admits a convergent power series
expansion 
$$
\varphi^*(\kappa)_{\mathbf{x}}=
g_{st}+\sum_{|I|\ge 1}x^I\,g_I \qquad
(\mathbf{x}\in B),
$$
where the $g_I$'s are fixed symmetric 2-tensors.
Here $x^I=x_1^{i_1}\cdots x_d^{i_d}$ if
$\mathbf{x}=\begin{pmatrix}
x_1&\ldots&x_d
\end{pmatrix}^\dagger$.

For some sufficiently
small $\epsilon>0$ let $U_\epsilon:=\varphi\big(B_d(\mathbf{0},
\epsilon )  \big)$ denote be the image of the open ball of
radius $\epsilon$. Since the singular support of
$\mathfrak{Q}^\tau$ is the diagonal, the computation of
the principal symbol of $\mathfrak{Q}^\tau$ at
$(x,\xi)$ may be localized
to $U_\epsilon$, meaning that in local coordinates 
it is given by the
leading order term of the asymptotic expansion  for
$\xi\rightarrow \infty$ of the integral
$$
\int_{B_d(\mathbf{0},
\epsilon )}e^{\imath\,\langle\xi,\mathbf{y}\rangle}\,
\mathfrak{Q}^\tau(\mathbf{0},\mathbf{y})\,
\rho(\epsilon^{-1}\,\mathbf{y})\,
\mathcal{V}(\mathbf{y})\,\mathrm{d}\mathbf{y},
$$
where $\rho$ is a suitable fixed cut-off function identically equal to
$1$ near the origin, and $\mathcal{V}(\mathbf{y})\,\mathrm{d}\mathbf{y}$ is the pull-back by $\varphi$ of the Riemannian
density on $M$ 
(here we occasionally blend intrinsic and local coordinate notation).
If we pull this back by the dilation 
$\rho_\epsilon: \mathbf{x}\in 
B_d(\mathbf{0},1)\mapsto \mathbf{y}:=\epsilon\,\mathbf{x}\in  B_d(\mathbf{0},\epsilon)$, we obtain the symbol at
$(0,\epsilon\,\xi)$ of the corresponding $\mathfrak{Q}^\tau$
referred to the metric
\begin{equation}
\label{eqn:rescaling metric epsilon}
(\varphi\circ \rho_\epsilon)^*(\kappa)=\epsilon^2\,g_\epsilon,
\quad \text{where}\quad g_\epsilon:=
g_{st}+\sum_{|I|\ge 1}\epsilon^{|I|}\,
x^I\,g_I.
\end{equation}

Now $g_\epsilon$ is a real-analytic Riemannian
metric on $B_d(\mathbf{0},1)$, it is defined for sufficiently 
small
$\epsilon$ and depends real analytically
on $\epsilon$; hence the same holds of the corresponding operators
$\mathfrak{Q}_{g_\epsilon}^\tau$ and their principal symbols. By (\ref{eqn:Qtau euclidean case}), 
(\ref{eqn:homog prop sigma}), 
and (\ref{eqn:rescaling metric epsilon})
we conclude that
$$
\sigma\left(\mathfrak{Q}_{g_\epsilon}^\tau\right)=
\left(\frac{\pi\,\tau}{\|\xi\|}  \right)^{\frac{d-1}{2}}\cdot
\big(1
+\epsilon\,F_m(\xi)\big),
$$
where $F_m$ is homogeneous of degree $0$ in $\xi$.

On the other hand,
in view of (\ref{eqn:rescaling Qtau})
on $U_\epsilon$ we have 
$$
\sigma\left(\mathfrak{Q}^\tau_{(\varphi\circ \rho_\epsilon)^*(\kappa)}\right)(m,\epsilon\,\xi)
=\epsilon^{d-1}\,\sigma\left(\mathfrak{Q}^{\tau/\epsilon}_{g_\epsilon}\right)(x,\epsilon\,\xi)=
\left(\frac{\pi\,\tau}{\|\xi\|}  \right)^{\frac{d-1}{2}}\cdot
\big(1
+\epsilon\,F_m(\xi)\big).
$$
Since the result must be independent of $\epsilon$,
we conclude that $F_m=0$.

Summing up, we conclude the following:

\begin{lem} 
\label{lem:princ symbol Qtau}
The principal symbol of $\mathfrak{Q}^\tau$ is
$$
\sigma(\mathfrak{Q}^\tau)(m,\xi)=
\left( \frac{\pi\,\tau}{\|\xi\|} \right)^{\frac{d-1}{2}}.
$$
\end{lem}
 
Going over the arguments in \S 2, \S 3 and especially 
\S 4 of \cite{z14}
or \S 5, \S 6 and especially \S 7 of \cite{z20}
in light of Lemma \ref{lem:princ symbol Qtau},
one obtains the following description of 
$U_\mathbb{C}(t+2\,\imath\,\tau)$
as a dynamical Toeplitz operator.

\begin{enumerate}
    \item Let us denote by $\Pi^\tau_{-t}$ the zeroth order 
Fourier integral operator on $X^\tau$ having distributional kernel 
$$
\Pi^\tau_{-t}(x,y):=\Pi^\tau\left(\Gamma^\tau_{-t}(x),y\right);
$$
then on 
$X^\tau$ there exist a smoothly varying pseudodifferential operator 
$Q^\tau_t$ of degree 
$-(d-1)/2$
and a smoothly varying operator $R^\tau_t$ with $\mathcal{C}^\infty$ kernel such that
\begin{equation}
\label{eqn:U_C compos}
U_\mathbb{C}(t+2\,\imath\,\tau)=\Pi^\tau\circ Q^\tau_t\circ \Pi^\tau_{-t}
+R^\tau_t.
\end{equation}

\item  In a conic neighoburhood of the symplectic cone $\Sigma^\tau$,
$Q^\tau_t$ admits the following microlocal description. Let us set
$D^\tau_{\sqrt{\rho}}:=\imath\,\upsilon_{\sqrt{\rho}}$. Then there exists
a polyhomogeneous classical symbol on $X^\tau\times \mathbb{R}_+$,
of the form
\begin{equation}
\label{eqn:classical symbol poisson wave}
\gamma^\tau_t (x,r)\sim \sum_{j\ge 0}\gamma^\tau_{t,j}(x)\,r^{-\frac{d-1}{2}-j},
\end{equation}
such that $Q^\tau_t\sim \gamma^\tau_t \big(x,D^\tau_{\sqrt{\rho}}\big)$.

\item Similarly to 
(\ref{eqn:correcting factor tau t}),
the leading coefficient in (\ref{eqn:classical symbol poisson wave}) is 
\begin{equation}
\label{eqn:leading coeff poisson}
\gamma^\tau_{t,0}(x)=
(\pi\,\tau)^{\frac{d-1}{2}}\cdot 
e^{\imath\,\tilde{\theta}^\tau_t(x)}\cdot 
\langle \sigma_J^{(x)},\sigma_{J_t}^{(x)}\rangle^{-1}
\end{equation}
for a certain smooth function $\tilde{\theta}^\tau_t:X^\tau\rightarrow
\mathbb{R}$. The additional factor $\tau^{\frac{d-1}{2}}$
with respect to the computation of Zelditch (see e.g. \cite{z12})
is due to our choice of volume form, which affects the construction
of ${P^\tau}^*$ (see the discussion in \cite{p24}).

\end{enumerate}

\begin{proof}
[Proof of Theorem \ref{thm:complexified Poisson wave asy}]
We are interested in the asymptotics of the kernel
$$
P^\tau_{\chi,\lambda}(x,y):=
\sum_j\hat{\chi}(\lambda-\mu_j)\,
e^{-2\,\tau\,\mu_j}\,\sum_k \widetilde{\varphi}^\tau_{j,k}(x)\,
\overline{\widetilde{\varphi}^\tau_{j,k}(y)}
\qquad 
(x,y\in X^\tau),
$$
and more generally of its equivariant version
\begin{equation}
\label{eqn:equiv spectral proj unitint}
P^\tau_{\chi,\nu,\lambda}(x,y):=
\sum_j\hat{\chi}(\lambda-\mu_j)\,
e^{-2\,\tau\,\mu_j}\,\sum_k \widetilde{\varphi}^\tau_{j,\nu,k}(x)\,
\overline{\widetilde{\varphi}^\tau_{j,\nu,k}(y)}.
\end{equation}
The analysis run parallel to the one conducted for 
$\Pi^\tau_{\chi,\nu,\lambda}$. 

More precisely, the following analogue of  (\ref{eqn:PitaunuFIO})
describes the relation between $P^\tau_{\chi,\lambda}$
and $U_\mathbb{C}(t+2\,\imath\,\tau)$
in (\ref{eqn:UC_spectral})
is given by :
\begin{equation}
\label{eqn:PtaunuFIO}
P^\tau_{\chi,\lambda}(x,y)=
\frac{1}{\sqrt{2\,\pi}}\,\int_{-\infty}^{+\infty}
e^{-\imath\,\lambda\,t}\,\chi (t)\,
U_\mathbb{C}(t+2\,\imath\,\tau,x,y)\,\mathrm{d}t;
\end{equation}
furthermore, operatorially we have the analogue of 
(\ref{eqn:Pi chi nu lambda projector}):
\begin{equation}
\label{eqn:equivariant projection complex}
P^\tau_{\chi,\nu,\lambda}=
P_\nu\circ P^\tau_{\chi,\lambda}.
\end{equation}
Arrguing as in the proof of Theorem
\ref{thm:main 1}, 
with $U_\mathbb{C}(t+2\,\imath\,\tau,x,y)$ in place of 
$U^\tau_{\sqrt{\rho}}(t;x,y)$
and (\ref{eqn:U_C compos}) in place of (\ref{eqn:Ptautcomp}),
we obtain in place of (\ref{eqn:3rd expression}):
\begin{eqnarray}
\label{eqn:3rd expression Poisson}
\lefteqn{P^\tau_{\chi,\nu,\lambda}(x_1,x_2)}\\
&\sim&\frac{\dim(\nu)}{\sqrt{2\,\pi}}\,
\int_G\,\mathrm{d}V_G(g)\,\int_{-\infty}^{+\infty}\,\mathrm{d}t
\,
\left[\Xi_\nu 
\left( g^{-1} 
\right)\,
e^{-\imath\,\lambda\,t}\,
\chi (t)\,  
\left(\Pi^\tau\circ Q^\tau_t\circ \Pi^\tau_{-t}    \right)
\left(\mu^\tau_{g^{-1}}(x_1),x_2\right)
\right].\nonumber
\end{eqnarray}

The arguments in the proof of Theorem \ref{thm:main 1}
apply, except that the leading term of the amplitude
has been multiplied by a factor 
$(\pi\,\tau)^{(d-1)/2}\,(u\,\tau)^{-(d-1)/2}$ and the
unitary factor $e^{\imath\,\tilde{\theta}^\tau_t(x)}$
in (\ref{eqn:leading coeff poisson})
replaces $e^{\imath\,\theta^\tau_t(x)}$.
We have used that
the principal symbol of
$D^\tau_{\sqrt{\rho}}$ along $\Sigma^\tau$ (or, equivalently, of $\mathfrak{D}^\tau_{\sqrt{\rho}}$) is  
$$
\sigma(D^\tau_{\sqrt{\rho}})\left(x,v\,\alpha_x^\tau\right)=v\,\tau
\qquad (v>0).
$$
In view of the rescaling $u\mapsto \lambda\,u$, this entails a
further factor $\lambda^{-(d-1)/2}$.  
Furthermore, at the critical point
(\ref{eqn:stationary Ps}) of the phase the product
$u\,\tau=1$. Thus to leading order we obtain an extra overall factor 
$(\lambda/\pi\,\tau)^{-(d-1)/2}$.
\end{proof}

\section{Near-graph uniform asymptotic expansions}

Given $(x_1,x_2)\in Z^\tau\times Z^\tau$, we have defined 
$\Sigma_\chi(x_1,x_2)\subseteq G\times \mathrm{supp}(\chi)$ in
Remark \ref{rem:unique t1chi}.

\begin{lem}
\label{lem:continuity Sigma}
$\Sigma_\chi(x_1,x_2)$ has the following properties:

\begin{enumerate}
\item $\Sigma_\chi(x_1,x_2)\neq \emptyset$ if and only if $(x_1,x_2)\in 
\mathfrak{X}^\tau_\chi$;
\item for any $(x_1,x_2)\in Z^\tau\times Z^\tau$ and any
neighbourhood $S$ of $\Sigma_\chi(x_1,x_2)$ in $G\times \mathrm{supp}(\chi)$, 
there exists a neighbourhood
$Z'$ of $(x_1,x_2)$ in $Z^\tau\times Z^\tau$ such that 
$$
(x_1',x_2')\in Z'\quad\Rightarrow\quad \Sigma_\chi(x_1',x_2')
\subseteq S.
$$
\end{enumerate}
\end{lem}

\begin{proof}
The first statement is obvious by definition. If the second was false,
for any $j=1,2,\ldots$ there would exist  
$(y_j',y_j'')\in Z^\tau\times Z^\tau$ having distance $<1/j$ from
$(x_1,x_2)$ and $(g_j,t_j)\in 
\Sigma_\chi(y_j',y_j'')$ having distance from $\Sigma_\chi(x_1,x_2)$
no less than $ \epsilon_0$, for some fixed $\epsilon_0>0$.
By compactness, we may assume without loss that $g_j\rightarrow g_\infty
\in G$ and $t_j\rightarrow t_\infty$ in $\mathrm{supp}(\chi)$.
By continuity, 
$$
y_j'=\mu^\tau_{g_j}\circ \Gamma^\tau_{t_j}(y_j'')
\quad \Rightarrow\quad x_1=\mu^\tau_{g_\infty}\circ 
\Gamma^\tau_{t_\infty}(x_2)\quad\Rightarrow\quad 
(g_\infty,t_\infty)\in \Sigma_\chi(x_1,x_2).
$$
hence $(g_j,t_j)\rightarrow  (g_\infty,t_\infty)\in \Sigma_\chi(x_1,x_2)$, absurd.
\end{proof}

\begin{proof}
[Proof of Theorem \ref{thm:near-graph-unrescaled}]
To begin with, we consider the asymptotics at fixed points
(that is, with no rescaling). Let us choose 
$(x_1,x_2)\in \mathfrak{X}^\tau_\chi$, hence
satisfying (\ref{eqn:Sigmax12}): there exist
$(g,t)\in G\times \mathrm{supp}(\chi)$ such that
$x_1=\mu^\tau_g\circ \Gamma^\tau_{t}(x_2)$.
We may assume that the possible pairs $(g,t)$ can be listed as in
(\ref{eqn:Sigmax12chi}). 
We have
\begin{eqnarray}
\label{eqn:6th expression unresc}
\lefteqn{\Pi^\tau_{\chi,\nu,\lambda}(x_1,x_2)}\\
&\sim&\lambda^2\,\frac{\dim(\nu)}{\sqrt{2\,\pi}}\,
\int_G\,\mathrm{d}V_G(g)\,\int_{-\infty}^{+\infty}\,\mathrm{d}t
\,\int_{X^\tau}\,\mathrm{d}V_{X^\tau}(y)\,
\int_0^{+\infty}\,\mathrm{d}u\,
\int_0^{+\infty}\,\mathrm{d}v\nonumber\\
&&\left[\Xi_\nu 
\left( g^{-1} 
\right)\,
\chi (t)\,  e^{\imath\,\lambda\,
\left[u\,\psi^\tau\left(\mu^\tau_{g^{-1}}(x_1),y\right)+v\,\psi^\tau
\left(\Gamma^\tau_{-t}(y),x_2\right)-t\right]}
\right.\nonumber\\
&&\left. 
\varrho_1(g,y)\,\varrho_2(t,y)\,
s^\tau\left(\mu^\tau_{g^{-1}}(x_1),y,\lambda\,u\right)\,
r^\tau_t\left(y,x_2,\lambda\,v\right)
\right].\nonumber
\end{eqnarray}
Let us multiply the integrand 
in (\ref{eqn:6th expression unresc})
by the unrescaled cut-off
\begin{equation}
\label{eqn:new cut off RGunresc 1}
\gamma^\mathbb{R}\left(
t-t_1\right)\cdot \sum_{l=1}^{r_{x_1}}\gamma^{\mathfrak{g}}\left(
\log_G\left(
g\,h_1^{-1}\,\kappa_l^{-1}\right)\right),
\end{equation}
where $G_{x_1}=\{\kappa_1,\ldots,\kappa_{r_{x_1}}\}$. 
Thus integration in $G\times \mathbb{R}$ 
has been restricted to a small
but fixed neighbourhood of $\Sigma_\chi(x_1,x_2)$, and only a rapidly
decreasing contribution to the asymptotics 
of (\ref{eqn:6th expression unresc}) is lost.
The same will then be true for the asymptotics of
$\Pi^\tau_{\chi,\nu,\lambda}(x_1',x_2')$ for any 
$(x'_1,x_2')$ in a fixed small neighbourhood of $(x_1,x_2)$, as
$\Sigma_\chi(x_1',x_2')$ is then contained in a small
neighbourhood of $\Sigma_\chi(x_1,x_2)$ by Lemma 
\ref{lem:continuity Sigma}.

Thus, uniformly on an open neighbourhood of $(x_1,x_2)$ we have
$$
\Pi^\tau_{\chi,\nu,\lambda}(x_{1}',x_{2}')
\sim \sum_{l=1}^{r_{x_1}}\Pi^\tau_{\chi,\nu,\lambda}
(x_{1}',x_{2}')_l,
$$
where for each $l$ (with the change of variable 
$t\mapsto t_1+t$) we have 
\begin{eqnarray}
\label{eqn:6th expression lth summand unresc gen}
\lefteqn{\Pi^\tau_{\chi,\nu,\lambda}(x'_1,x'_2)_l}\\
&=&e^{-\imath\,\lambda\,t_1}\,\lambda^{2}\,\frac{\dim(\nu)}{\sqrt{2\,\pi}}\,
\int_{\mathfrak{g}}\,\mathrm{d}\xi\,
\int_{-\infty}^{+\infty}\,\mathrm{d}t
\,\int_{-\infty}^{+\infty}\,\mathrm{d}\theta\,
\int_{\mathbb{R}^{2d-2}}\,\mathrm{d}\mathbf{u}\,
\int_0^{+\infty}\,\mathrm{d}u\,
\int_0^{+\infty}\,\mathrm{d}v\nonumber\\
&&\left[e^{\imath\,\lambda\,\Psi_{x'_1,x'_2}(
t,v,\theta,u,\mathbf{v},\xi)_l}\,
\mathcal{H}_\lambda(x'_1,x'_2,
t,v,\theta,u,\mathbf{v},\xi)_l
\right].\nonumber
\end{eqnarray}
At $(x_1,x_2)$
we have
\begin{eqnarray}
\label{eqn:Psithetax1x2 unresc}
\lefteqn{\Psi_{x_1,x_2}(t,v,\theta,u,\mathbf{v},\xi)_l:=
-u\,\theta
+v\,(  \theta+\tau\,t)  -t  }
\nonumber\\
&&+\imath\, u\,\left[ 
 \frac{1}{4\,\tau^2}\,\theta^2  -\psi_2\left( 
-\xi_{X^\tau}(x_1)^{(l)},\mathbf{v}   \right)                                           \right]
+\imath\, v\,\left[
\frac{1}{4\,\tau^2}   \,(  \theta+\tau\,t)^2
+\frac{1}{2}\big\|B\mathbf{v}\big\|^2\right]
\nonumber\\
&&+R_3(\mathbf{v},\xi,\theta,t)
\end{eqnarray}
and 
\begin{eqnarray}
\label{eqn:Hlambda12}
\lefteqn{\mathcal{H}_\lambda(x_1,x_2,
t,v,\theta,u,\mathbf{v},\xi)_l  }\nonumber\\
&:=&
\Xi_\nu 
\left(h_1^{-1}\, \kappa_l^{-1}\, e^{-\xi}
\right)\,
\chi \left(t+t_1\right)
\cdot \mathcal{V}\left(\theta,
\mathbf{u}\right)\cdot \gamma(t,\theta,\xi,\mathbf{u})\cdot 
f_1(v)\cdot f_2(u)
\nonumber\\
&&\cdot 
s^\tau\left(\mu^\tau_{(e^{\xi}\kappa_l )^{-1}}(x),y_\lambda (\theta,\mathbf{u}),\lambda\,u\right)\,
r^\tau_t\big(y_\lambda (\theta,\mathbf{u}),x,\lambda\,v\big),
\end{eqnarray}
where $\gamma(t,\theta,\xi,\mathbf{u})$ 
is a product of unrescaled cut-offs and $f_1,\,f_2$ are as in Proposition \ref{prop:compact u and v}.
In particular, integration is restricted to a small neighbourhood
of the locus where $\theta=t=0$, $\mathbf{v}=\xi=0$, and
is compactly supported in $(u,v)$.

We have $\Im \Psi_{x_1,x_2}(t,v,\theta,u,\mathbf{v},\xi)_l\ge 0$, and
$\Im \Psi_{x_1,x_2}(t,v,\theta,u,\mathbf{v},\xi)_l= 0$ only if
$$
\theta=t=0,\quad \mathbf{v}=\mathbf{0}_{\mathbb{R}^{d_G}},\quad 
\xi=0_{\mathfrak{g}}.
$$
We have in addition a stationary point, if we also
impose $u=v=1/\tau$. Thus the only nonnegligible contribution
to the asymptotics may come from a neighourhood of
$$
P_0:=(t_0,v_0,\theta_0,u_0,\mathbf{v}_0,\xi_0)
=\left(0,\frac{1}{\tau},0,\frac{1}{\tau},\mathbf{0}_{\mathbb{R}^{d_G}},
0_{\mathfrak{g}}\right)
$$

Let ${D_{x_1}^{l}}$ denote the $(2d-2)\times d_G$ 
matrix representing the injective linear map
$\xi\mapsto \xi_{X^\tau}(x_1)^{(l)}$ with respect to the
given orthonormal basis in $\mathfrak{g}$ and 
$\mathcal{H}^\tau_{x_1}$.
At the critical point, the Hessian of $\Psi_{x_1,x_2}(\cdot)_l$ is
$$
H_{P_0}\big( \Psi_{x_1,x_2}(\cdot)_l \big)=\begin{pmatrix}
\imath/2&\tau &\imath/(2\,\tau^2) &0 &\mathbf{0}^\dagger&\mathbf{0}^\dagger\\
\tau& 0&1 &0& \mathbf{0}^\dagger&\mathbf{0}^\dagger   \\
\imath/(2\,\tau^2)& 1  & \imath/\tau^3&-1 &\mathbf{0}^\dagger &\mathbf{0}^\dagger  \\
0& 0  &-1 &0&\mathbf{0}^\dagger&\mathbf{0}^\dagger        \\
\mathbf{0} &\mathbf{0}&\mathbf{0} &\mathbf{0}&
\mathcal{A}  &\mathcal{B}   \\
\mathbf{0}&\mathbf{0}&\mathbf{0} &\mathbf{0}
&\mathcal{B}^\dagger&\mathcal{C}
\end{pmatrix}
$$
where
$$
\mathcal{A}:=\frac{\imath}{\tau}\,\left(I_{2d-2}
+B^\dagger\,B\right),
\quad 
\mathcal{B}:=\frac{\imath}{\tau}\,(I_{2d-2}+\imath\,J_0)\,{D_{x_1}^{l}},\quad
\mathcal{C}:=\frac{\imath}{\tau}\,{D_{x_1}^{l}}^\dagger\,{D_{x_1}^{l}}.
$$
Here $J_0$ is the standard complex structure on 
$\mathbb{C}^{d-1}\cong \mathbb{R}^{d-1}\oplus \mathbb{R}^{d-1}$.
Thus the Hessian matrix is in block diagonal form, and the
determinants of the two blocks are as follows.

First we have
$$
\det\left(\frac{1}{\imath}\,
\begin{pmatrix}
\imath/2&\tau&\imath/(2\,\tau^2)&0\\
\tau&0&1&0\\
\imath/(2\,\tau^2)&1&\imath/\tau^3&-1\\
0&0&-1&0
\end{pmatrix}\right)=\tau^2.
$$
As to the second block, we have
\begin{eqnarray*}
\lefteqn{
\det\left(\frac{1}{\imath}\, \begin{pmatrix}
\mathcal{A}&\mathcal{B}\\
\mathcal{B}^\dagger&\mathcal{C}
\end{pmatrix}\right)}\\
&=&
\det\left(\frac{1}{\imath}\,\begin{pmatrix}
\frac{\imath}{\tau}\,\left(I_{2d-2}
+B^\dagger\,B\right)&\frac{\imath}{\tau}\,(I_{2d-2}+\imath\,J_0)\,{D_{x_1}^{l}}\\
\frac{\imath}{\tau}\,
{D_{x_1}^{l}}^\dagger\,(I_{2d-2}-\imath\,J_0)&\frac{\imath}{\tau}\,
{D_{x_1}^{l}}^\dagger\,{D_{x_1}^{l}}
\end{pmatrix}\right)  
\\
&=&\frac{1}{\tau^{2d-2+d_G}}\,
\det\begin{pmatrix}
\left(I_{2d-2}
+B^\dagger\,B\right)&(I_{2d-2}+\imath\,J_0)\,{D_{x_1}^{l}}\\
{D_{x_1}^{l}}^\dagger\,(I_{2d-2}-\imath\,J_0)&
{D_{x_1}^{l}}^\dagger\,{D_{x_1}^{l}}
\end{pmatrix} 
\\
&=&\frac{1}{\tau^{2d-2+d_G}}\,
\det\begin{pmatrix}
\left(I_{2d-2}
+B^\dagger\,B\right)&(I_{2d-2}+\imath\,J_0)\,{D_{x_1}^{l}}\\
0&S
\end{pmatrix},
\end{eqnarray*}
where 
\begin{eqnarray*}
S&:=&
{D_x^{l}}^\dagger\,{D_{x_1}^{l}}-{D_{x_1}^{l}}^\dagger\,(I_{2d-2}-\imath\,J_0)\,
\left(I_{2d-2}
+B^\dagger\,B\right)^{-1}\,(I_{2d-2}+\imath\,J_0)\,{D_{x_1}^{l}}.
\end{eqnarray*}
Thus $S=S^\dagger$, and the real part of $S$ is
\begin{eqnarray*}
\Re(S)&:=&{D_{x_1}^{l}}^\dagger\,{D_x^{l}}-{D_x^{l}}^\dagger\,\left(I_{2d-2}
+B^\dagger\,B\right)^{-1}\,\,{D_x^{l}}\\
&&-{D_{x_1}^{l}}^\dagger\,J_0\,
\left(I_{2d-2}
+B^\dagger\,B\right)^{-1}\,J_0\,{D_{x_1}^{l}}.
\end{eqnarray*}
We have (since $B$ is a symplectic matrix)
\begin{eqnarray*}
\left(I_{2d-2}
+B^\dagger\,B\right)\,J_0&=& J_0+B^\dagger\,B\,J_0\\
&=&J_0+J_0\,\left(B^\dagger\,B\right)^{-1}\\
&=&J_0\,\left(I_{2d-2}
+\left(B^\dagger\,B\right)^{-1}\right).
\end{eqnarray*}
Hence, taking inverses we get
$$
J_0\,\left(I_{2d-2}
+B^\dagger\,B\right)^{-1}=
\left(I_{2d-2}
+\left(B^\dagger\,B\right)^{-1}\right)^{-1}\,J_0.
$$
Thus
\begin{eqnarray*}
\Re(S)&:=&{D_{x_1}^{l}}^\dagger\,{D_{x_1}^{l}}-{D_{x_1}^{l}}^\dagger\,\left(I_{2d-2}
+B^\dagger\,B\right)^{-1}\,\,{D_{x_1}^{l}}\\
&&-{D_{x_1}^{l}}^\dagger\,\left(I_{2d-2}
+\left(B^\dagger\,B\right)^{-1}\right)^{-1}\,J_0\,J_0\,{D_{x_1}^{l}}\\
&=&{D_{x_1}^{l}}^\dagger\,\left[I_{2d-2}-\left(I_{2d-2}
+B^\dagger\,B\right)^{-1}\right]\,\,{D_{x_1}^{l}}\\
&&+{D_{x_1}^{l}}^\dagger\,\left(I_{2d-2}
+\left(B^\dagger\,B\right)^{-1}\right)^{-1}\,{D_{x_1}^{l}},
\end{eqnarray*}
whence $\Re(S)\gg 0$.

Hence $P_0$ is a non-degenerate stationary point; 
by the complex stationary phase lemma of \cite{ms},
we obtain
an asymptotic expansion for $\Pi^\tau_{\chi,\nu,\lambda}(x_1,x_2)$ 
which must agree with our previous derivation.
Moreover, since the 
complex stationary phase holds with parameters, we can replace $(x_1,x_2)$ 
by a general $(x_1',x_2')$ varying in
some small open neighbourhood of $(x_1,x_2)$,
and obtain an
asymptotic expansions
\begin{equation}
    \label{eqn:general expansion unresc}
    \Pi^\tau_{\chi,\nu,\lambda}(x_1',x_2')\sim 
\sum_{l=1}^{r_{x_1}}\Pi^\tau_{\chi,\nu,\lambda}(x'_1,x_2')_l,
\end{equation}
but for some $l$ the corresponding stationary
point (accounting for the expansion of the $l$-th summand
in (\ref{eqn:general expansion unresc}))
might cease to be real; this happens
when $(x_1',x_2')\in \mathfrak{X}_\chi^\tau$ but $r_{x_1'}<r_{x_1}$, or when 
$(x_1',x_2')\not\in \mathfrak{X}_\chi^\tau$. In this case, 
$\Pi^\tau_{\chi,\nu,\lambda}(x_1',x_2')_l
=O\left(\lambda^{-\infty}\right)$.

In particular, let us now replace
$x_j$ by, say, $x_j':=x_j+(\theta_j,\mathbf{v}_j^t)$ 
for some nearby $x_j'=x_j+
(\theta_j,\mathbf{v}_j)$, with $(\theta_j,\mathbf{v}_j)$ suitably small.
By non-degeneracy, the critical point will vary smoothly, 
except that it
will generically move to the complex domain, which accounts for the exponential decay in the scaling asymptotics. Nonetheless, the asymptotic expansions will still hold, and if we pair this with the scaling asymptotics in the previous theorems, 
we obtain the statement of Theorem \ref{thm:near-graph-unrescaled}.

\end{proof}

\begin{rem}
\label{rem:remainder est}
In our situation, $R_3$ and 
$L_{\nu,l,s}\left(x_1,x_2;\cdot\right),\,
K_{\nu,l,s}\left(x_1,x_2;\cdot\right)$
are complex valued real-analytic functions (meaning that their real and imaginary components are real-analytic).
Hence there are expansions of the form
\begin{eqnarray*}
e^{\lambda\,R_3\left(\theta_1,\mathbf{v}_1^t,
\theta_2,\mathbf{v}_2^t\right)}&=&
\sum_{j\ge 0}\frac{\lambda^j}{j!}\,R_{3}\left(\theta_1,\mathbf{v}_1^t,
\theta_2,\mathbf{v}_2^t\right)^j\\
&=&\sum_{j\ge 0}\frac{\lambda^j}{j!}\,\left(
\sum_{a\ge 0}P_{3+a}\left(\theta_1,\mathbf{v}_1^t,
\theta_2,\mathbf{v}_2^t\right)\right)^j\\
&=&\sum_{j\ge 0}\lambda^j\,\sum_{a\ge 0}P_{3j+a}\left(\theta_1,\mathbf{v}_1^t,
\theta_2,\mathbf{v}_2^t\right)
\end{eqnarray*}
where $P_{k}\left(\cdot\right)$ is a generic homogeneous polynomial
of degree $k$ and, say,
$$
L_{\nu,l,s}\left(x_1,x_2;\theta_1,\mathbf{v}_1^t,
\theta_2,\mathbf{v}_2^t\right)=
\sum_{k\ge 0}L_{\nu,l,s,k}\left(x_1,x_2;\theta_1,\mathbf{v}_1^t,
\theta_2,\mathbf{v}_2^t\right),
$$
where each $L_{\nu,l,s,k}\left(x_1,x_2;\cdot\right)$
is a homogeneous polynomial of total degree $k$.
Upon rescaling, the product of these expansions
gives rise to terms of the form
\begin{eqnarray*}
\lefteqn{ \lambda^{j-s-\frac{k}{2}
-\frac{3j}{2}-\frac{a}{2}}\,P_{3j+a}\left(\theta_1,\mathbf{v}_1^t,
\theta_2,\mathbf{v}_2^t\right)\,L_{\nu,l,s,k}\left(x_1,x_2;\theta_1,\mathbf{v}_1^t,
\theta_2,\mathbf{v}_2^t\right)   }\\
&=&\lambda^{-\frac{1}{2}\,(j+2s+a +k )}\,
\tilde{P}_{k+3j+a}\left(\theta_1,\mathbf{v}_1^t,
\theta_2,\mathbf{v}_2^t\right)
\end{eqnarray*}
where $\tilde{P}_{k+3j+a}$ is homogeneous of total degree 
$3j+a+k$.
We obtain the same conclusions as before, that is general
terms of the form $\lambda^{-k/2}\,P_r$ with
$r\le 3\,k$ and $r-k$ even.
\end{rem}

\section{Applications}

\subsection{The equivariant Weyl law for 
$\mathfrak{D}^\tau_{\sqrt{\rho}}$}

\begin{proof}
[Proof of Theorem \ref{thm:weyl law geod}]
Let $\chi$ have sufficiently small support and satisfy
$\hat{\chi}>0$ and $\chi(0)>0$. 
Recalling (\ref{eqn:Pi chi nu lambda projector}),
\begin{eqnarray}
\label{eqn:weyl law integrand0}
\lefteqn{\sum_j \hat{\chi}(\lambda-\lambda_j)\,\dim H^\tau(X)_{j,\nu}}\\
&=&\sum_j \hat{\chi}(\lambda-\lambda_j)\,\int_{X^\tau}
\,\Pi^\tau_{j,\nu}(x,x)\,\mathrm{d}V_{X^\tau}(x)\nonumber\\
&=&\int_{X^\tau}
\,\left[\sum_j \hat{\chi}(\lambda-\lambda_j)\,
\Pi^\tau_{j,\nu}(x,x)\right]\,\mathrm{d}V_{X^\tau}(x)\nonumber\\
&=&\int_{X^\tau}\,
\Pi^\tau_{\chi,\nu,\lambda}(x,x)\,\mathrm{d}V_{X^\tau}(x).
\nonumber
\end{eqnarray}
Let us fix $C>0$ and $\epsilon'\in (0,1/6)$.
By Theorem \ref{thm:main 1}, we only loose a rapidly decreasing
contribution in (\ref{eqn:weyl law integrand0}), if
integration is restricted to a shrinking
tubular neighourhood of $Z^\tau$ of radius 
$C\,\lambda^{\epsilon'-\frac{1}{2}}$.
Furthermore, any such neighbourhood 
may be locally parametrized 
using smoothly varying systems of NHLC's centered 
at moving points $x\in Z^\tau$.

More precisely, in view of Remark 
\ref{rem:mutau decomp},
for any $x\in Z^\tau$ we may find an open neighbourhood 
$Z'\subseteq Z^\tau$ of $x$ and a smoothly varying family of
normal Heisenberg coordinates centered at points $x'\in Z'$, such that
the map
\begin{equation}
\label{eqn:local parametriz}
(x',\mathbf{v}^t)\in Z'\times \mathbb{R}^{d_G}\mapsto x'+\mathbf{v}^t\in X^\tau
\end{equation}
is a diffeomorphism onto a neighbourhood of $x$ in $X^\tau$ (here we
use normal Heisenberg coordinates at each $x'$ to identify
$\mathbb{R}^{d_G}\cong\mathbb{R}^{d_G}_t\cong T_{x'}^tX^\tau$).
In view of Corollary 35 of \cite{p24},
if $x=x'+\mathbf{v}^t$ then
\begin{equation}
\label{eqn:volume form Ztau v}
\mathrm{d}V_{X^\tau}(x)=2^{d_G/2}\,
\mathcal{V}(x'+\mathbf{v}^t)\,
\mathrm{d}V_{Z^\tau}(x')\,\mathrm{d}\mathbf{v}^t,
\end{equation}
where $\mathrm{d}V_{Z^\tau}$ is the Riemannian volume form
on $Z^\tau$ for the restricted metric,
and $\mathcal{V}(x')=1$ for any $x'\in Z'$.

If we pass to the rescaled local parametrization
\begin{equation}
\label{eqn:local param normal dspl}
x_\lambda(\mathbf{v}^t):=x'+\frac{\mathbf{v}^t}{\sqrt{\lambda}}
\qquad (x'\in Z',\,\mathbf{v}^t\in \mathbb{R}^{d_G}),
\end{equation}
then uniformly for
$\|\mathbf{v}^t\|\le C\,\lambda^{\epsilon'}$ 
there are asymptotic expansion
\begin{eqnarray}
\label{eqn:diadg exp equiv}
\Pi^\tau_{\chi,\nu,\lambda}\left(
x_\lambda(\mathbf{v}^t),x_\lambda(\mathbf{v}^t)\right)
&\sim&\frac{1}{\sqrt{2\,\pi}}\,\left( \frac{\lambda}{2\,\pi\,\tau}\right)^{d-1-d_G/2}
\cdot 
\frac{\dim(\nu)^2}{ V_{eff}(x_1)}\\
&&\cdot e^{-\frac{2}{\tau}\,
\left\|\mathbf{v}^t\right\|^2
}
\cdot \left[\chi (0)+\sum_{k\ge 1}\,\lambda^{-k/2}\,F_{k,\chi,\nu}
\left(x';\mathbf{v}^t\right)\right],
\nonumber
\end{eqnarray}
where $F_{k,\chi,\nu}(x';\cdot)$ is a polynomial of degree $\le 3\,k$ and 
parity $k$.

Let us choose an open cover of $Z^\tau$ by open sets $Z'_j$ as above,
and a subordinate partition of unity $\beta_j$ on $Z^\tau$;
we can then express the latter integral
in (\ref{eqn:weyl law integrand0}) as a sum of local contributions. 
For notational simplicity, we shall force notation and
leave the partition $\{\beta_j\}$ implicit.
We obtain:
\begin{eqnarray}
\label{eqn:weyl law integrand split}
\lefteqn{\sum_j \hat{\chi}(\lambda-\lambda_j)\,\dim H^\tau(X)_{j,\nu}}\\
&\sim&2^{d_G/2}\int_{Z^\tau}\,\mathrm{d}V_{Z^\tau}(x)
\left[\int_{\mathbb{R}^{d_G}}\,
\Pi^\tau_{\chi,\nu,\lambda}(x+\mathbf{v}^t,x+\mathbf{v}^t)\,
\,\mathcal{V}(x+\mathbf{v}^t)
\,\mathrm{d}\mathbf{v}^t\right]\nonumber\\
&=&\left(\frac{\lambda}{2}\right)^{-d_G/2}\,\int_{Z^\tau}\,\mathrm{d}V_{Z^\tau}(x)
\left[\int_{\mathbb{R}^{d_G}}\,
\Pi^\tau_{\chi,\nu,\lambda}
\left(x_\lambda(\mathbf{v}^t),
x_\lambda(\mathbf{v}^t)\right)\,
\,\mathcal{V}(x_\lambda(\mathbf{v}^t))
\,\mathrm{d}\mathbf{v}^t\right].
\nonumber
\end{eqnarray}
Integration in $\mathrm{d}\mathbf{v}^t$
is now over an expanding ball of radius 
$O\left( \lambda^{\epsilon'} \right)$ in $\mathbb{R}^{d_G}$.
By multiplying the asymptotic expansion of Theorem 
\ref{thm:diagonal case} with the Taylor expansion of
$\mathcal{V}(x_\lambda(\mathbf{v}^t))$, we obtain an asymptotic
expansion for the integrand in (\ref{eqn:weyl law integrand split})
of the form
\begin{eqnarray}
\label{eqn:diadg exp equiv prod}
\lefteqn{\Pi^\tau_{\chi,\nu,\lambda}
\left(x_\lambda(\mathbf{v}^t),
x_\lambda(\mathbf{v}^t)\right)\,
\,\mathcal{V}(x_\lambda(\mathbf{v}^t))}
\\
&\sim&\frac{1}{\sqrt{2\,\pi}}\,\left( \frac{\lambda}{2\,\pi\,\tau}\right)^{d-1-d_G/2}
\cdot 
\frac{\dim(\nu)^2}{V_{eff}(x)}\cdot e^{-\frac{2}{\tau}\,
\left\|\mathbf{v}^t\right\|^2
} \nonumber\\
&&
\cdot \left[\chi (0)+\sum_{k\ge 1}\,\lambda^{-k/2}\,F_{k,\nu}
\left(x;\mathbf{v}^t\right)\right],
\nonumber
\end{eqnarray}
where $F_{k,\nu}(x;\cdot)$ is a polynomial of degree $\le 3\,k$ and 
parity $k$. The expansion (\ref{eqn:diadg exp equiv prod}) 
may be integrated term by term; 
we obtain an asymptotic expansion 
for (\ref{eqn:weyl law integrand split}) of the form
\begin{eqnarray}
\label{eqn:weyl law integrand split 1 int glob}
\lefteqn{\sum_j \hat{\chi}(\lambda-\lambda_j)\,\dim H^\tau(X)_{j,\nu}}\\
&\sim&\frac{2^{d_G/2}}{\sqrt{2\,\pi}}\,\left( \frac{\lambda}{2\,\pi\,\tau}\right)^{d-1-d_G/2}\,\lambda^{-d_G/2}
\cdot \dim(\nu)^2\nonumber\\
&& \cdot
\left(
\int_{Z^\tau}\,\frac{1}{V_{eff}(x)}\,\mathrm{d}V_{Z^\tau}(x)
\right)\cdot 
\int_{\mathbb{R}^{d_G}}\,
e^{-\frac{2}{\tau}\,
\left\|\mathbf{v}^t\right\|^2
}\,\mathrm{d}\mathbf{v}^t 
\cdot \left[\chi (0)+\sum_{k\ge 1}\,\lambda^{-k}\,\mathfrak{f} _{k,\nu}
\right]
\nonumber\\
&=&\frac{1}{2^{d_G/2}\,\sqrt{2\,\pi}}\,\left( \frac{\lambda}{2\,\pi\,\tau}\right)^{d-1-d_G}
\cdot \dim(\nu)^2\cdot
\mathrm{vol}(Z^\tau/G)
\cdot \left[\chi (0)+\sum_{k\ge 1}\,\lambda^{-k}\,\mathfrak{f} _{k,\nu}\right].
\nonumber
\end{eqnarray}

As in \cite{p24}, we shall now follow a standard
argument from spectral analysis (see e.g. \cite{gs}).
Let 
$f_\lambda:\mathbb{R}\times \mathbb{R}\rightarrow [0,+\infty)$
be defined by
\begin{equation}
f_\lambda(s,t):=\hat{\chi}(t)\cdot H(\lambda-s-t),
\label{eqn:defn di flambda}
\end{equation}
where $H$ is the Heaviside function.
Let $\mathcal{L}$ be the Lebsesgue measure on $\mathbb{R}$,
and let $\mathcal{P}^\tau_{\nu}$ be the positive
measure on $\mathbb{R}$ given by
\begin{equation}
\label{eqn:defn Ttau}
\mathcal{Q}^\tau_{\nu}:=\sum_{j\ge 1}
\dim H^\tau(X)_{j,\nu}\,\delta_{\lambda_j};
\end{equation}
here $\delta_a$ is the delta measure at $a\in \mathbb{R}$.
Let us endow $\mathbb{R}\times \mathbb{R}$ with the product
meaure $\mathcal{P}^\tau_{\nu}\times \mathcal{L}$.
By the Fubini Theorem,
\begin{equation}
\label{eqn:Fubini}
\int_\mathbb{R}\,\mathrm{d}\mathcal{L}(t)\,
\left[\int_\mathbb{R}\,f_\lambda(s,t)\, 
\mathrm{d}\mathcal{Q}^\tau_{\nu}(s)
  \right]=\int_\mathbb{R}\,
\mathrm{d}\mathcal{Q}^\tau_{\nu}(s)
\left[\int_\mathbb{R}\,f_\lambda(s,t)\,\mathrm{d}\mathcal{L}(t)
  \right].
\end{equation}

The right hand side in (\ref{eqn:Fubini}) is
\begin{eqnarray}
\label{eqn:integr rhs}
\lefteqn{\int_\mathbb{R}\,
\mathrm{d}\mathcal{Q}^\tau_{\nu}(s)
\left[\int_\mathbb{R}\,f_\lambda(s,t)\,\mathrm{d}\mathcal{L}(t)
  \right]   }\\
  &=&  \int_\mathbb{R}\,
\mathrm{d}\mathcal{Q}^\tau_{\nu}(s)
\left[\int_{-\infty}^{\lambda-s}\,\hat{\chi}(t)\,\mathrm{d}\mathcal{L}(t)
  \right]=   \int_\mathbb{R}\,
\mathrm{d}\mathcal{Q}^\tau_{\nu}(s)
\left[\int_{-\infty}^{\lambda}\,\hat{\chi}(t-s)\,\mathrm{d}\mathcal{L}(t)
  \right]              \nonumber\\
  &=&\sum_j\dim H^\tau(X)_{j,\nu}\,
  \int_{-\infty}^{\lambda}\,\hat{\chi}(t-\lambda_j)\,\mathrm{d}\mathcal{L}(t)\nonumber\\
  &=&\int_{-\infty}^{\lambda}\,
  \left[\sum_j\,
  \hat{\chi}(t-\lambda_j)\cdot \dim H^\tau(X)_{j,\nu}\right]\,\mathrm{d}\mathcal{L}(t)
  \nonumber
  \end{eqnarray}
In view of (\ref{eqn:weyl law integrand split 1 int glob}), 
we conclude
that as $\lambda\rightarrow +\infty$
\begin{eqnarray}
\label{eqn:exp tot diag dx1}
\lefteqn{\int_\mathbb{R}\,
\mathrm{d}\mathcal{Q}^\tau_{\nu}(s)
\left[\int_\mathbb{R}\,f_\lambda(s,t)\,\mathrm{d}\mathcal{L}(t)
  \right]}\\
   &=&\frac{1}{2^{d_G/2}\,\sqrt{2\,\pi}}\,
\cdot \frac{\lambda}{d-d_G}\,   
   \left( \frac{\lambda}{2\,\pi\,\tau}\right)^{d-1-d_G}
\cdot \dim(\nu)^2\cdot
\mathrm{vol}(Z^\tau/G)\nonumber\\
&&
\cdot \left[\chi (0)+\sum_{k\ge 1}\,\lambda^{-k}\,F_{k,\nu}
\left(x\right)\right].
\nonumber
\end{eqnarray}

On the other hand, the left hand side in (\ref{eqn:Fubini})
is
\begin{eqnarray}
\label{eqn:Fubini lhs}
\lefteqn{\int_\mathbb{R}\,\mathrm{d}\mathcal{L}(t)\,
\left[\int_\mathbb{R}\,f_\lambda(s,t)\, 
\mathrm{d}\mathcal{Q}^\tau_{\nu}(s)
  \right]}\\
 &=&\int_\mathbb{R}\,\left[
 \sum_{j}\dim H^\tau(X)_{j,\nu}\,\,
 \hat{\chi}(t)\cdot H(\lambda-\lambda_j-t)
 \right]\,\mathrm{d}\mathcal{L}(t)\nonumber\\
 &=&\int_\mathbb{R}\,\left[
 \sum_{\lambda_j\le \lambda-t}\dim H^\tau(X)_{j,\nu}\,
 \right]\,
 \hat{\chi}(t)\,\mathrm{d}\mathcal{L}(t)\nonumber\\
 &=&\int_\mathbb{R}\,
 \mathcal{W}_\nu^\tau(\lambda-t)\,
 \hat{\chi}(t)
 \,\mathrm{d}\mathcal{L}(t)\nonumber\\
 &=&\sqrt{2\,\pi}\,\chi (0)\,\mathcal{W}_\nu^\tau(\lambda)
+\int_{-\infty}^{+\infty}\,\left[\mathcal{W}_\nu^\tau(\lambda-t)
-\mathcal{W}_\nu^\tau(\lambda)\right]\,\hat{\chi}(t)
 \,\mathrm{d}\mathcal{L}(t).\nonumber
\end{eqnarray}

An adaptation of the argument in the proof of Lemma 70 of
\cite{p24}, 
paired with (\ref{eqn:weyl law integrand split 1 int glob}), yields the estimate 
\begin{equation}
\label{eqn:integrsl estimate}
\int_{-\infty}^{+\infty}\,\left[\mathcal{W}_\nu^\tau(\lambda-t)
-\mathcal{W}_\nu^\tau(\lambda)\right]\,\hat{\chi}(t)
 \,\mathrm{d}\mathcal{L}(t)=O\left(\lambda^{d-1-d_G}   \right)
\end{equation}
for $\lambda\rightarrow +\infty$.
We conclude
\begin{eqnarray}
\label{eqn:weyl 10}
\mathcal{W}_\nu^\tau(\lambda)&=&
\frac{2^{d_G/2}}{2^{d-1}\,(2\,\pi)}\,
\cdot \frac{\lambda}{d-d_G}\,   
   \left( \frac{\lambda}{\pi\,\tau}\right)^{d-1-d_G}
\cdot \dim(\nu)^2\cdot
\mathrm{vol}(Z^\tau/G)\nonumber\\
&&
\cdot \left[1+O\left(  \lambda^{-1} \right)\right]\nonumber\\
&=&\frac{1}{2^{d_G/2}}\cdot 
\frac{\tau}{d-d_G}\cdot \left( \frac{\lambda}{2\,\pi\,\tau}  \right)^{d-d_G}
\cdot \dim(\nu)^2\cdot
\mathrm{vol}(Z^\tau/G)\nonumber\\
&&
\cdot \left[1+O\left(  \lambda^{-1} \right)\right]\nonumber
\end{eqnarray}

\end{proof}

\subsection{Pointwise estimates on eigenfunctions}

\label{sctn:pointwise estimates}

\begin{proof}
[Proof of Proposition \ref{prop:uniform diagonal estimate}]
If $x\not\in Z^\tau$, there exists an open neighbourhood 
$Y_x\subseteq X^\tau$ of $x$ such that 
$\Pi^\tau_{\chi,\nu,\lambda}(x',x')=O\left(\lambda^{-\infty}\right)$
for $\lambda\rightarrow +\infty$, uniformly for $x'\in Y_x$.
If $x\in X^\tau$, there exists similarly a neighbourhood $Y_x$ where
the asymptotic expansion of Theorem \ref{thm:near-graph-unrescaled}
implies $\Pi^\tau_{\chi,\nu,\lambda}(x',x')\le C_x\,\lambda^{d-1-d_G/2}$
for some constant $C_x$, uniformly for $x'\in Y_x$. 
The claim follows by the compactness of $X^\tau$.

\end{proof}

\subsection{Equivariant operator norm estimates}

We can now prove the equivariant version of the
operator norm estimate of Chang and Rabinowitz (\ref{eqn:cr norm estimate});
a similar statement may be proved with minor changes for 
$P^\tau_{\chi,\nu,\lambda}$.

\begin{proof}
[Proof of Theorem \ref{thm:norm operator estimate}]
By the 
Shur-Young inequality, 
for some positive constant $D^\tau_{p}$
\begin{equation}
\label{eqn:shur-young-abs}
\left\|\Pi^\tau_{\chi,\nu,\lambda}\right\|_{L^p\rightarrow L^q}\le 
D_{p}^\tau\,\left[
\sup_{y\in X^\tau}\int_{X^\tau}\left|\Pi^\tau_{\chi,\nu,\lambda}
(y',y)   \right|^r\,\mathrm{d}V_{X^\tau}(y')
\right]^{\frac{1}{r}}, 
\end{equation}
where
$$
\frac{1}{r}:=1-\frac{1}{p}+\frac{1}{q}.
$$

Let us choose $C>0$ and $\epsilon'\in (0,1/6)$.
By Theorem \ref{thm:main 1}, $\Pi^\tau_{\chi,\lambda}
(y',y)=O\left(\lambda^{-\infty}\right)$ uniformly for
$\mathrm{dist}_{X^\tau}\left(y,Z^\tau\right)\ge C\,\lambda^{\epsilon'-1/2}$.
Thus we may assume without loss of generality that 
$\mathrm{dist}_{X^\tau}\left(y,Z^\tau\right)\le C\,\lambda^{\epsilon'-1/2}$.
Any such $y$ may be written in the form 
\begin{equation}
\label{eqn:y param Ztau}
y=x+\frac{\mathbf{v}^t}{\sqrt{\lambda}},
\quad \text{where}\quad 
x\in Z^\tau,\quad \mathbf{v}^t\in T^t_x X^\tau,
\quad \|\mathbf{v}^t\|\le 2\,C\,\lambda^{\epsilon'}.
\end{equation}

Let $0<a<A$ be such that 
$$
a\,\mathrm{dist}_{X^\tau}\left(x',
x''\right)\le \mathrm{dist}_{X^\tau}\left(\Gamma^\tau_t(x'),
\Gamma^\tau_t(x'')\right)\le 
A\,\mathrm{dist}_{X^\tau}\left(x',
x''\right)$$
for all $x',x''\in X^\tau$ and $(g,t)\in G\times \mathrm{supp}(\chi)$.

\begin{lem}
Uniformly for $y$ as in (\ref{eqn:y param Ztau}) and 
\begin{equation}
\label{eqn:dist Xtau yxG}
\mathrm{dist}_{X^\tau}\left(y',x^{G\times \chi}\right)\ge 3\,A\,C\,\lambda^{\epsilon'-1/2}.
\end{equation}
we have $\Pi^\tau_{\chi,\lambda}
(y',y)=O\left(\lambda^{-\infty}\right)$.
\end{lem}

\begin{proof}
If (\ref{eqn:dist Xtau yxG}) holds, 
then for any $(g,t)\in G\times \mathrm{supp}(\chi)$ 
\begin{eqnarray}
\lefteqn{\mathrm{dist}_{X^\tau}\left(y',\mu_g\circ \Gamma^\tau_t(y)\right)}
\nonumber\\
&\ge& \mathrm{dist}_{X^\tau}\left(y',\mu_g\circ \Gamma^\tau_t(x)\right)
-\mathrm{dist}_{X^\tau}\left(\mu_g\circ \Gamma^\tau_t(y),\mu_g\circ \Gamma^\tau_t(x)\right)\nonumber\\
&=&\mathrm{dist}_{X^\tau}\left(y',\mu_g\circ \Gamma^\tau_t(x)\right)
-\mathrm{dist}_{X^\tau}\left(\Gamma^\tau_t(y),
\Gamma^\tau_t(x)\right)\nonumber\\
&\ge&\mathrm{dist}_{X^\tau}\left(y',\mu_g\circ \Gamma^\tau_t(x)\right)
-2\,A\,C\,\lambda^{\epsilon'-\frac{1}{2}}\ge 
A\,C\,\lambda^{\epsilon'-\frac{1}{2}},
\end{eqnarray}
whence 
$\mathrm{dist}_{X^\tau}\left(y',x^{G\times \chi}\right)\ge 
A\,C\,\lambda^{\epsilon'-\frac{1}{2}}$.
The claim follows by Theorem \ref{thm:main 1}.
\end{proof}

Hence integration in (\ref{eqn:shur-young-abs}) may be restricted to 
the shrinking locus where (\ref{eqn:dist Xtau yxG}) is satisfied.
Any $y'\in X^\tau$ satisfying (\ref{eqn:dist Xtau yxG}) may 
in turn be written uniquely
(for a given choice of smoothly varying normal NHLC's)
in the form 
\begin{equation}
\label{eqn:param-resc-y'}
y'=\Gamma^\tau_t\circ \mu^\tau_{g}(x)
+\frac{\mathbf{u}^t+\mathbf{u}^h}{\sqrt{\lambda}},
\end{equation}
where $t\in \mathrm{supp}(\chi)$, $g\in G$ and 
$$
\mathbf{u}^t\in T^t_{\Gamma^\tau_t\circ \mu^\tau_{g}(x)}X^\tau,
\quad 
\mathbf{u}^h\in T^h_{\Gamma^\tau_t\circ \mu^\tau_{g}(x)}X^\tau,
\quad \left\|\mathbf{u}^t\right\|,\,\left\|\mathbf{u}^h\right\|
=O\left(\lambda^{\epsilon'}\right).
$$

In view of statement 3. of Theorem \ref{thm:main 2} and
of the previous considerations on uniformity,
there exist constants $a,\,C_\nu>0$ such that, uniformly 
for all
such choices,
\begin{eqnarray}
\label{eqn:estimate for Pi r}
\left|\Pi^\tau_{\chi,\nu,\lambda}
(y',y)   \right|^r&=&
\left|\Pi^\tau_{\chi,\nu,\lambda}
\left(\Gamma^\tau_t\circ \mu^\tau_{g}(x)
+\frac{\mathbf{u}^t+\mathbf{u}^h}{\sqrt{\lambda}},
x+\frac{\mathbf{v}^t}{\sqrt{\lambda}}  \right)   \right|^r
\nonumber\\
&\le&C_\nu\,\left( \frac{\lambda}{\tau}  \right)^{r\,(d-1-d_G/2)}\,
e^{-a\,\frac{r}{\tau}\left(\|\mathbf{u}^t\|+\|\mathbf{u}^h\|^2+
\|\mathbf{v}^t\|^2
\right)}
\end{eqnarray}

Let us set 
\begin{equation}
\label{eqn:defn di F r}
F_r(y):=\int_{X^\tau}\left|\Pi^\tau_{\chi,\nu,\lambda}
(y',y)   \right|^r\,\mathrm{d}V_{X^\tau}(y')
\qquad (y\in X^\tau).
\end{equation}

Using an unparametrized version of (\ref{eqn:param-resc-y'}),
that is, $y'=\Gamma^\tau_t\circ \mu^\tau_{g}(x)
+\mathbf{u}^t+\mathbf{u}^h$,
we have 
$\mathrm{d}V_{X^\tau}(y')=\mathcal{V}(g,t,\mathbf{u}^t,\mathbf{u}^h)
\,\mathrm{d}g\,\mathrm{d}t\,\mathrm{d}\mathbf{u}^t\,
\mathrm{d}\mathbf{u}^h$ for some positive function
$\mathcal{V}(g,t,\mathbf{u}^t,\mathbf{u}^h)$, and integration
in $\mathrm{d}\mathbf{u}^t\,
\mathrm{d}\mathbf{u}^h$ is on a shrinking ball of radius 
$O\left(\lambda^{\epsilon'-1/2}\right)$. Applying the 
rescaling in (\ref{eqn:param-resc-y'}) yields
$\mathrm{d}V_{X^\tau}(y')=\lambda^{-(d-1-d_G)-d_G/2}\,\mathcal{V}(g,t,\mathbf{0}^\dagger,\mathbf{0}^h)
\,\mathrm{d}g\,\mathrm{d}t\,\mathrm{d}\mathbf{u}^t\,
\mathrm{d}\mathbf{u}^h+L.O.T.$, where 
$L.O.T.$ denotes lower order terms in $\lambda$, and integration in
in $\mathrm{d}\mathbf{u}^t\,
\mathrm{d}\mathbf{u}^h$ is now on an expanding ball of radius 
$O\left(\lambda^{\epsilon'}\right)$.

Given this, we obtain from (\ref{eqn:estimate for Pi r}) that
uniformly for $x\in Z^\tau$ and $\mathbf{v}^t\in T^t_x X^\tau$ with 
$\|\mathbf{v}^t\|\le 2\,C\,\lambda^{\epsilon'}$ we have
(for some constant $D_\nu>0$)
\begin{eqnarray}
\label{eqn:estimate for Fr}
F_r\left( x+\frac{\mathbf{v}^t}{\sqrt{\lambda}}  \right)
&\le&D_{\nu,r}\,
\left( \frac{\lambda}{\tau}  \right)^{r\,\left(d-1-\frac{d_G}{2}\right)}\,
\left(\frac{\tau}{\lambda}\right)^{d-1-\frac{d_G}{2}}
\nonumber\\
&=&D_{\nu,r}\,
\left( \frac{\lambda}{\tau}  \right)^{(r-1)\,(d-1-d_G/2)}.
\end{eqnarray}

Hence, uniformly in $y\in X^\tau$ we have for some
constant $D_{\nu,r}'>0$ that
\begin{eqnarray*}
\left[
\sup_{y\in X^\tau}\int_{X^\tau}\left|\Pi^\tau_{\chi,\nu,\lambda}
(y',y)   \right|^r\,\mathrm{d}V_{X^\tau}(y')
\right]^{\frac{1}{r}}
&\le&D_{\nu,r}'\,
\left( \frac{\lambda}{\tau}  \right)
^{\left(1-\frac{1}{r}\right)\,\left(d-1-\frac{d_G}{2} \right)}\\
&=&D_{\nu,r}'\,
\left( \frac{\lambda}{\tau}  \right)
^{\left(\frac{1}{p}-\frac{1}{q}\right)\,\left(d-1-\frac{d_G}{2} \right)}.
\end{eqnarray*}

\end{proof}

\subsection{Proof of Proposition \ref{prop:weyl law integrand split 1 int glob 1} and Theorem \ref{thm:weyl poisson}}

\begin{proof}
[Proof of Proposition \ref{prop:weyl law integrand split 1 int glob 1}]
In the setting of Theorem \ref{thm:weyl law geod},
in place of (\ref{eqn:diadg exp equiv}) we have
\begin{eqnarray}
\label{eqn:diadg exp equiv poisson}
\lefteqn{P^\tau_{\chi,\nu,\lambda}\left(
x_\lambda(\mathbf{v}^t),x_\lambda(\mathbf{v}^t)\right)}
\\
&\sim&\frac{1}{\sqrt{2\,\pi}}\,
\left(\frac{1}{2}\right)^{d-1-d_G/2}
\left( \frac{\lambda}{\pi\,\tau}\right)^{(d-1-d_G)/2}
\cdot 
\frac{\dim(\nu)^2}{ V_{eff}(x_1)}\nonumber\\
&&\cdot e^{-\frac{2}{\tau}\,
\left\|\mathbf{v}^t\right\|^2
}
\cdot \left[\chi (0)+\sum_{k\ge 1}\,\lambda^{-k/2}\,F_{k,l,\nu}
\left(x;\mathbf{v}^t\right)\right],
\nonumber
\end{eqnarray}
where
is $x_\lambda(\mathbf{v}^t)$ is
as in (\ref{eqn:local param normal dspl}) and 
$F_{k,l,\nu}(x;\cdot)$ is a polynomial of degree $\le 3\,k$ and 
parity $k$. 
The statement of the Theorem follows by integration.
\end{proof}

\begin{proof}
[Proof of Theorem \ref{thm:weyl poisson}]
In analogy with (\ref{eqn:defn Ttau}), let us define
the positive measure
\begin{equation}
\label{eqn:defn tildeTtau}
\widetilde{\mathcal{T}}^\tau_{\nu}:=\sum_{j\ge 1}
\Lambda_{j,\nu}\,\delta_{\mu_j}
\end{equation}
where
$$
\Lambda_{j,\nu}:=e^{-2\,\tau\,\mu_j}\,\sum_k \left\|\widetilde{\varphi}^\tau_{j,\nu,k}\right\|_{L^2(X^\tau)}^2.
$$

With $f_\lambda$ as in (\ref{eqn:defn di flambda}),
we have as in (\ref{eqn:integr rhs})
\begin{eqnarray}
\label{eqn:integr rhs poisson}
\int_\mathbb{R}\,
\mathrm{d}\widetilde{\mathcal{T}}^\tau_{\nu}(s)
\left[\int_\mathbb{R}\,f_\lambda(s,t)\,\mathrm{d}\mathcal{L}(t)
  \right]   
  =  \int_{-\infty}^{\lambda}\,
  \left[\sum_j\,
  \hat{\chi}(t-\mu_j)\cdot \Lambda_{j,\nu}\right]\,\mathrm{d}\mathcal{L}(t).
  \end{eqnarray}
In view of Proposition \ref{prop:weyl law integrand split 1 int glob 1}, 
we conclude
that as $\lambda\rightarrow +\infty$
\begin{eqnarray}
\label{eqn:exp tot diag dx poisson}
\lefteqn{\int_\mathbb{R}\,
\mathrm{d}\widetilde{\mathcal{T}}^\tau_{\nu,\lambda}(s)
\left[\int_\mathbb{R}\,f_\lambda(s,t)\,\mathrm{d}\mathcal{L}(t)
  \right]}\\
   &=&\frac{1}{\sqrt{2\pi}}\,
   \frac{1}{2^{d-1-d_G/2}}\,\left( \frac{\lambda}{\pi\,\tau}\right)^{\frac{d-1}{2}-d_G}\,
\cdot \dim(\nu)^2\cdot
\mathrm{vol}(Z^\tau/G)\nonumber\\
&&
\cdot \lambda\,\left[\dfrac{\chi (0)}{\frac{d+1}{2}-d_G}+\sum_{k\ge 1}\,\lambda^{-k}\,\mathfrak{f}_k''\right].\nonumber
\end{eqnarray}
Similarly, arguing as for (\ref{eqn:Fubini lhs}) we obtain
\begin{eqnarray}
\label{eqn:Fubini lhs poisson}
\lefteqn{\int_\mathbb{R}\,\mathrm{d}\mathcal{L}(t)\,
\left[\int_\mathbb{R}\,f_\lambda(s,t)\, 
\mathrm{d}\widetilde{\mathcal{T}}^\tau_{\nu}(s)
  \right]}\\
 &=&\sqrt{2\,\pi}\,\chi (0)\,\mathcal{P}_\nu^\tau(\lambda)
+\int_{-\infty}^{+\infty}\,\left[\mathcal{P}_\nu^\tau(\lambda-t)
-\mathcal{P}_\nu^\tau(\lambda)\right]\,\hat{\chi}(t)
 \,\mathrm{d}\mathcal{L}(t)\nonumber\\
 &=&\sqrt{2\,\pi}\,\chi (0)\,\mathcal{P}_\nu^\tau(\lambda)
+O\left(\lambda ^{\frac{d-1}{2}-d_G} \right),
\nonumber
\end{eqnarray}
where the last equality follows from 
Proposition \ref{prop:weyl law integrand split 1 int glob 1} and the argument in Lemma 70 of \cite{p24}.

By the Fubini Theorem,
(\ref{eqn:Fubini lhs poisson}) equals 
(\ref{eqn:exp tot diag dx poisson}); the claim follows.

\end{proof}

\section{Index of Notation}

\begin{enumerate}
\item $(M,\kappa)$: the given $d$-dimensional Riemannian manifold;
$(\tilde{M},J)$: its complexification; 

\item $\rho:\tilde{M}\rightarrow [0,+\infty)$: the strictly plurisubharmonic
exhaustion function
determined by $\kappa$;

\item $h_0,\,\omega_0,\, g_0$: the standard hermitian, symplectic and Riemannian structure. $\Omega,\,\hat{\kappa}$: see (\ref{eqn:Omega and kappa}); 
$\omega$ and $\tilde{\kappa}$: see Remark \ref{rem:warning norm};

\item $\psi_2^{\omega}$: see Definition \ref{defn:psi2};

\item $G$: $d_G$-dimensional Lie group; $\mathfrak{g},\, \mathfrak{g}_{X^\tau}(x)$: see Notation \ref{defn:induced vct field};

\item $\tilde{M}^\tau$, $X^\tau$: see (\ref{eqn:MXtau}); 
$\alpha^\tau$: see (\ref{eqn:contact str});

\item $\mathcal{H}^\tau$, $\mathcal{R}^\tau$, $\mathcal{T}^\tau$:
see (\ref{eqn:defn Htau}), (\ref{eqn:reeb vct field}), 
(\ref{eqn:dicrect sum vert hor}); $T^vX^\tau,\,T^tX^\tau,\,T^hX^\tau$: see Definition \ref{defn:vth dec Ztau};

\item $\Delta,\,\mu_j,\,W_j, \,\varphi_{j,k}$: see (\ref{eqn:eigenvalue laplacian});
\item $\Pi^\tau$, $H(X^\tau)$: see (\ref{eqn:szego proj and ker});
$\mathfrak{D}^\tau_{\sqrt{\rho}}$: see (\ref{eqn:DsqrtrhoToeplitz});
$H(X^\tau)_j$, $\rho_{j,k}$: see (\ref{eqn:smoothed proj kern});

\item $ P^\tau_{\chi,\,\lambda}$: see
(\ref{eqn:temepered_proj_kernel_poisson}); 
$\lambda_j$, $\rho_{j,k}$,
$\Pi^\tau_{\chi,\,\lambda}$:
see (\ref{eqn:smoothed proj kern});
 
 \item $\mu:G\times M\rightarrow M$: isometric Lie group action;
$\hat{G}$, $\nu$, $V_\nu$, $\Xi_\nu$, $\dim(\nu)$: see Notation \ref{notn:irred repr char}; $V_{eff}$: see Definition \ref{defn:effective volume};

\item $\tilde{\mu}^\tau,\,Z,\,Z^\tau$: see
(\ref{eqn:defn di Ztau});

\item $x_{j,\,\lambda}$: see (\ref{eqn:rescaled coord 12}); $x_{12}$: see (\ref{eqn:defn di x12});

\item the matrix $B$: see (\ref{eqn:defn of M}), (\ref{eqn:defn of M12});

\item $W_{j,\nu},\,H(X^\tau)_{j,\nu}$: see (\ref{eqn:L2Mdec})
and (\ref{eqn:L2Mdecompequiv}); $P_\nu$: see (\ref{eqn:Pi chi nu lambda projector});

\item $P^\tau_{\chi,\nu,\lambda}$: see (\ref{eqn:temepered_proj_kernel_poisson_equiv});
$\Pi^\tau_{\chi,\nu,\lambda}$: see 
(\ref{eqn:smoothed proj kern equiv}); 

\item $\mathrm{d}V_G(g)$: the Haar measure on $G$;
$\mathrm{d}V_{X^\tau}(x)$: the measure on $X^\tau$ associated to
the Riemannian volume form $\mathrm{vol}^R_{X^\tau}$ (see 
\S 3.2.3 of \cite{p24});

\item  $\Gamma^\tau_t$, $
x^{G\times \chi}$, $
\mathfrak{X}^\tau_\chi$, $x^\chi$, $x^G$: see Definition
\ref{defn:concentration locus} and the discussion 
immediately preceding it;

\item the matrices $A_c$, $P$, $Q$ associated to a symplectic matrix
$A$: see Definition \ref{defn:sympl compl mat};
the function $\Psi_A:\mathbb{R}^{2d-2}\times \mathbb{R}^{2d-2}\rightarrow \mathbb{C}$ associated to a symplectic matrix $A$: see
Definition \ref{defn:quadratice form};

\item $e^{\imath \theta^\tau_t(x)},\,
e^{\imath \tilde{\theta}^\tau_t(x)}$:
see Remark \ref{rem:unitary factor Pi};

\item $A_\chi$: see Definition \ref{eqn:defn di Achi};
$\mathcal{F}_{\chi}$, $\mathcal{B}_\nu$:
see Definition \ref{defn:FchiBnu};

\item $\psi^\tau$, $s^\tau$: see (\ref{eqn:szego as FIO});
$r^\tau_t$: see (\ref{eqn:rttauexp});

\item $\Sigma_\chi(x_1,x_2)$: see (\ref{eqn:Sigmax12});

\item $\mathcal{W}^\tau_\nu(\lambda)$: see (\ref{eqn:defn di Wtaonu}); $\mathcal{P}^\tau_\nu(\lambda)$: see (\ref{eqn:defn Ptaunuweyl});

\item $\psi_{\gamma},\, J_{\mathrm{ad}}$: see Notation \ref{defn:defn di psigamma};

\item $\Phi,\,\varphi^{\xi},\,\xi_{X^\tau}^\sharp$: see Notation \ref{notn:moment map};

\item $\mathbf{v}_g$: see (\ref{eqn:vg}); $\mathbf{v}^{(l)}$: see (\ref{eqn:vl}); $\mathbf{v}(\tau)$: see (\ref{eqn:rescale vj}); $\mathbf{v}^v,\, \mathbf{v}^t,\, \mathbf{v}^h$: see Definition \ref{defn:vth dec Ztau};

\item $U_{\sqrt{\rho}}(t),\, \Pi^\tau_{t},\, \tilde{\Pi}^\tau_t,\, J^{\tau}_t,\,\mathcal{P}^\tau_t,\, \sigma^\tau_{t,\, j},\,\sigma^{(x)}_J$: see \S\ref{sctn:preamble1};

\item $\Psi(x_1,x_2; g, t,y ,u,v)$: see (\ref{eqn:fasePsi'}); $\mathcal{I}_k(\theta_1,\mathbf{v}_!,\theta_2,\mathbf{v}_2)$: see (\ref{eqn:defn di mathcalIl}); $A(\mathbf{v}_1 ,\mathbf{v}_2 , \mathbf{u})_l$: see (\ref{eqn:defn di Al});

\item $U(\imath\tau),\, \mathfrak{B}^\tau,\,\tilde{\varphi}_{j,k},\, U_{\mathbb{C}}(t+2\imath\tau),\, \mathcal{Q}^\tau_t, \gamma^\tau_{t}, \gamma^\tau_{t,j}$: see \S\ref{sctn:complexified eigenf lapl};

\item $\Re,\, \Im$: the real and imaginary part of a complex number;

\item The Fourier transform of $f\in \mathcal{S}(\mathbb{R})$ is
$$
\hat{f}(\lambda)=\frac{1}{\sqrt{2\,\pi}}\,\int_{-\infty}^{+\infty}\,
e^{-\imath\,\lambda\,t}\,f(t)\,\mathrm{d}t.
$$
\item $\mathrm{dist}_{X^\tau}:
X^\tau\times X^\tau\rightarrow \mathbb{R}$: 
the Riemannian distance function
on $X^\tau$.
\item $\bullet^\dagger$ indicates the transpose of a matrix or a vector, $\bullet^*$ indicates the conjugate transpose of a matrix.

\end{enumerate}

\textbf{Acknowledgments.} 
The authors are members of GNSAGA (Gruppo Nazionale
per le Strutture Algebriche, Geometriche e le loro Applicazioni) of INdAM
(Istituto Nazionale di Alta Matematica ”Francesco Severi”) and thank the
group for its support.

\end{document}